\documentclass[11pt]{amsart}

\usepackage{fullpage}

\usepackage{amsmath, amscd, amsthm, amssymb}

\usepackage{hyperref}

\def\C{{\mathbf C}}
\def\R{{\mathbf R}}
\def\Z{{\mathbf Z}}
\def\Q{{\mathbf Q}}
\def\A{{\mathbf A}}

\newtheorem{theorem}{Theorem}[subsection]

\newtheorem{lemma}[theorem]{Lemma}

\newtheorem{proposition}[theorem]{Proposition}

\newtheorem{corollary}[theorem]{Corollary}

\newtheorem{claim}[theorem]{Claim}

\theoremstyle{definition}
\newtheorem{definition}[theorem]{Definition}

\theoremstyle{remark}
\newtheorem{remark}[theorem]{Remark}

\newcommand{\mm}[4]{\left(\begin{smallmatrix} #1 & #2\\ #3 & #4\end{smallmatrix}\right)}

\DeclareMathOperator{\tr}{tr}

\DeclareMathOperator{\SO}{SO}

\DeclareMathOperator{\Sp}{Sp}
\DeclareMathOperator{\GSp}{GSp}

\DeclareMathOperator{\GU}{GU}
\DeclareMathOperator{\SU}{SU}

\DeclareMathOperator{\SL}{SL}
\DeclareMathOperator{\GL}{GL}

\def\g{{\mathfrak g}}
\def\h{{\mathfrak h}}
\def\k{{\mathfrak k}}
\def\p{{\mathfrak p}}
\def\a{{\mathfrak a}}
\def\m{{\mathfrak m}}
\def\so{{\mathfrak {so}}}
\def\sl{{\mathfrak {sl}}}
\def\sp{{\mathfrak {sp}}}

\def\qch{{\omega}}
\def\Wh{{\mathcal W}}
\def\Vm{{\mathbb{V}}}

\allowdisplaybreaks

\begin{document}
\title{The Fourier expansion of modular forms on quaternionic exceptional groups}
\author{Aaron Pollack}
\address{Department of Mathematics\\ Institute for Advanced Study\\ Princeton, NJ USA}
\email{aaronjp@math.ias.edu}
\address{Department of Mathematics\\ Duke University\\ Durham, NC USA}
\email{apollack@math.duke.edu}

\begin{abstract} Suppose that $G$ is a simple adjoint reductive group over $\Q$, with an exceptional Dynkin type, and with $G(\R)$ quaternionic (in the sense of Gross-Wallach).  Then there is a notion of modular forms for $G$, anchored on the so-called quaternionic discrete series representations of $G(\R)$.  The purpose of this paper is to give an explicit form of the Fourier expansion of modular forms on $G$, along the unipotent radical $N$ of the Heisenberg parabolic $P = MN$ of $G$.  
\end{abstract}
\maketitle

\setcounter{tocdepth}{1}
\tableofcontents
\section{Introduction} The purpose of this paper is to understand explicitly the Fourier expansion of modular forms on quaternionic exceptional groups, in the sense of \cite{ganGrossSavin}.  

\subsection{Modular forms on exceptional groups} If the symmetric space $G(\R)/K$ associated to a reductive $\Q$-group $G$ has Hermitian tube structure, then $G$ has a good theory of \emph{modular forms}.  On these groups, one can consider the automorphic forms whose infinite component lies in the holomorphic discrete series.  For example, when $G = \GSp_{2n}$, one has the Siegel modular forms.  These modular forms possess a refined Fourier expansion.  Namely, if $Z$ denotes the complex parameter in the Siegel upper half-space $\mathcal{H}_n = \{Z \in M_n(\C): Z = Z^{t}, Im(Z) > 0\}$ and $f(Z)$ is a Siegel modular form, then $f(Z) = \sum_{T \geq 0}{a(T) e^{2\pi i \tr(TZ)}}$ where the sum is over symmetric $n\times n$ matrices $T$ in $M_n(\Q)$ which are positive semi-definite and the $a(T)$ are complex numbers.

When the symmetric space of a reductive $\Q$-group $G$ does not have Hermitian structure, there is no \emph{a priori} reason to believe that any parallel theory for $G$ might exist, in terms of special automorphic forms and their Fourier expansions.  For example, no real form of the Dynkin types $G_2, F_4$ and $E_8$ has a symmetric space with Hermitian structure, so an obvious theory of modular forms is lacking if $G$ is of this type.  However, for certain groups $G$, Gross and Wallach \cite{grossWallach2} have singled out a special class of \emph{quaternionic discrete series} representations that can take the place of the holomorphic discrete series representations above.  More precisely, if $G$ is a simple reductive group over $\R$, then $G(\R)$ possesses quaternionic discrete series \cite{grossWallach2} if $G$ isogenous to one of $\Sp(n,1;\mathbf{H})$, $\SU(n,2)$, $\SO(n,4)$, the split exceptional group $G_2$, or a form of the exceptional groups $F_4, E_6, E_7, E_8$ with real rank four.  Here $\mathbf{H}$ is Hamilton's quaternions.

Throughout the paper, $G$ will be a reductive adjoint group, over $\Q$ or over $\R$, with $G(\R)$ isogenous to $\SO(n,4)$ with $n \geq 3$ or the quaternionic exceptional groups $G_2, F_4, E_6, E_7, E_8$.  The symmetric spaces associated to these groups do not have Hermitian structure.  (The forms of $E_6$ and $E_7$ with quaternionic discrete series are not the real forms that have a Hermitian symmetric space.) Nevertheless, work of Gross-Wallach \cite{grossWallach1,grossWallach2} on quaternionic discrete series representations and their continuations, Wallach \cite{wallach} on generalized Whittaker vectors, Gan \cite{ganSW} on Siegel-Weil formulas, and Gan-Gross-Savin \cite{ganGrossSavin} on modular forms on $G_2$, all suggest that the group $G$ might have a good theory of modular forms, anchored upon the quaternionic discrete series representations.  Furthermore, these works suggest that said modular forms might have a refined Fourier expansion, similar to the Siegel modular forms on $\GSp_{2n}$.

The purpose of this paper is to give the explicit Fourier expansion for modular forms on $G$.  To define the modular forms on $G$ and discuss their Fourier expansions, we first recall some features of these groups.  Denote by $K$ a maximal compact subgroup of $G(\R)^{0}$, the connected component of the identity of $G(\R)$.  Then $K \simeq (\SU(2) \times L^0)/\mu_2$ for a certain group $L^0$, with $\mu_2 = \{\pm 1\}$ embedded diagonally.  If $2n \geq \frac{1}{2} \dim G(\R)/K$, then there is a quaternionic discrete series representation $\pi_n$, whose minimal $K$-type is of the form $Sym^{2n}(V_2) \boxtimes \mathbf{1} = \Vm_n$ \cite{grossWallach2}, as a representation of $\SU(2) \times L^0$.  Here we write $V_2$ for the defining two-dimensional representation of $\SU(2)$ and $\Vm_n = Sym^{2n}(V_2) \boxtimes \mathbf{1}$.

Assume $G$ is an adjoint reductive $\Q$-group of type $G_2, F_4, E_6, E_7$ or $E_8$, with $G(\R)$ quaternionic as above.  As in \cite{ganGrossSavin} and \cite{weissman}, one can define a modular form of weight $n$ for $G$ to be a homomorphism $\varphi \in \mathrm{Hom}_{G(\R)}(\pi_n,\mathcal{A}(G(\Q)\backslash G(\A)))$, from $\pi_n$ to the space of automorphic forms on $G(\A)$, where $\pi_n$ is the quaternionic discrete series representation $G(\R)$ just discussed.  We make a slightly broader definition.  To motivate it, suppose $\varphi$ is as above.  Restricting $\varphi$ to the minimal $K$-type $\Vm_n$ in $\pi_n$, one obtains a function $F_{\varphi}: G(\A) \rightarrow \Vm_n^\vee$ defined by $F_{\varphi}(g)(v) = \varphi(v)(g)$. This function satisfies $F_{\varphi}(gk) = k^{-1} \cdot F_{\varphi}(g)$ for $g \in G(\A)$ and $k \in K$, the maximal compact subgroup of $G(\R)$.  Furthermore, by virtue of it being defined through the minimal $K$-type $\Vm_n$ of $\pi_n$, it is in the kernel of certain first-order differential operator $\mathcal{D}_n$, commonly referred to as the Schmid operator.

More precisely, suppose $\Theta$ is a Cartan involution on the complexified Lie algebra $\g$ of $G(\R)$, and $\g = \k \oplus \p$ is the decomposition into $\Theta = 1$ and $\Theta = -1$ parts.  Define $\widetilde{D}F_{\varphi} = \sum_{i}{X_iF \otimes X_i^*}$, where $X_i$ is a basis of $\p$, $X_i^*$ is the dual basis of the dual space $\p^\vee$, and $X_iF$ denotes the right-regular action.  This definition is of course independent of the choice of basis $X_i$ of $\p$, and defines a function $\widetilde{D}F_{\varphi}: G(\A) \rightarrow \Vm_n^\vee \otimes \p^\vee$.  As a representation of $\SU(2) \times L^0$, $\p \simeq V_2 \boxtimes W$ for a certain symplectic representation $W$ of $L^0$.  Consequently, there is a $K$-equivariant surjection $pr_{-}: \Vm_n^\vee \otimes \p^\vee \rightarrow Sym^{2n-1}(V_2)^\vee \boxtimes W$.  Set $\mathcal{D}_n = pr_{-} \circ \widetilde{D}$.  This is the so-called Schmid operator for the representation $\pi_n$.  

Thus $\mathcal{D}_n$ is a first-order differential operator 
\[\mathcal{D}_n: (C^\infty(G(\R)) \otimes \Vm_n^\vee)^{K} \rightarrow (C^\infty(G(\R)) \otimes Sym^{2n-1}(V_2)^\vee \boxtimes W)^{K}.\]
One has $\mathcal{D}_n F_{\varphi} = 0$.  This annihilation is simply an expression of the fact that $\pi_n$ contains the $K$-type $Sym^{2n}(V_2) \boxtimes \mathbf{1}$ but not the $K$-type $Sym^{2n-1}(V_2) \boxtimes W$.  If the group $G$ had a Hermitian symmetric space, and $\pi$ were a holomorphic discrete series, then being in the kernel of the analogous Schmid operator $\mathcal{D}$ would simply amount to being a holomorphic function. 

With these preliminaries, we make the following definition.
\begin{definition}\label{def:MFn} Suppose $n \geq 1$, and $F: G(\Q)\backslash G(\A) \rightarrow \Vm_n^\vee$ is a smooth function of moderate growth, finite under the center of the universal enveloping algebra of $\g$.  We say that $F$ is a \emph{modular form} on $G$ of weight $n$ if
\begin{enumerate}
\item $F(gk) = k^{-1} \cdot F(g)$ for all $g \in G(\A)$ and $k \in K$;
\item $\mathcal{D}_n F(g) = 0$.
\end{enumerate} \end{definition}

Thus, if $\pi_n$ is the discrete series representation above, and $F = F_{\varphi}$ for $\varphi: \pi_n \rightarrow \mathcal{A}(G(\A))$, then $F$ is a modular form of weight $n$.  If $\pi$ is not a discrete series representation, but $\pi = \pi_n'$ from \cite[Proposition 5.7]{grossWallach2} with minimal $K$-type a non-trivial even symmetric power of $V_2$, and $\varphi: \pi \rightarrow \mathcal{A}(G(\A))$ is a $G(\R)$-equivariant homomorphism, then $F_{\varphi}$ is again a modular form of weight $n$.  In particular, Definition \ref{def:MFn} encompasses the minimal representation on quaternionic $E_6,E_7, E_8$, which give rise to modular forms of weight $1, 2$ and $4$, respectively \cite[Proposition 8.6]{grossWallach2}.

The groups $G$ all possess a Heisenberg parabolic subgroup $P = MN$, with Levi subgroup $M$ and unipotent radical $N$.  Our main result is an explicit description of the Fourier expansion of a modular form of weight $n$ for $G$ along the abelianization of $N$.  More precisely, suppose that $\chi: N(\R) \rightarrow \C^\times$ is a character and $\Wh^\chi: G(\R) \rightarrow \Vm_n^\vee$ is a smooth function of moderate growth satisfying
\begin{enumerate}
\item $\Wh^\chi(gk) = k^{-1} \Wh^\chi(g)$ for all $g \in G(\R)$ and $k \in K$;
\item $\Wh^\chi(ng) = \chi(n) \Wh^\chi(g)$ for all $g \in G(\R)$ and $n \in N(\R)$;
\item $\mathcal{D}_n \Wh^\chi(g) = 0$.\end{enumerate}
In Theorem \ref{intro:mainThm} we give an explicit description of all functions $\Wh^\chi$.  Because these functions control the Fourier expansions of modular forms of weight $n$ for $G$, one immediately obtains the explicit form of their Fourier expansions as a corollary.

Functions of the form $\Wh^\chi$ are known as generalized Whittaker functions.  It is a consequence of the main result of Wallach from \cite{wallach} that when $\Wh^\chi$ comes from the discrete series representation $\pi_n$, and when $\chi$ is generic\footnote{The space of characters of $N$ is a prehogeneous vector space for the action of the Levi subgroup $M$.  A character is generic if it is in the unique Zariski open orbit of $M$.}, there is at most a one-dimensional space of such $\Wh^\chi$.  As a consequence of our explicit description, we recover this multiplicity at most one result, and extend it to the case of nontrivial but non-generic $\chi$; see Corollary \ref{cor:mult1}.

Generalized Whittaker functions as above have been studied on other groups.  When $G = \SU(2,1)$ a formula for the generalized Whittaker function is given in \cite{kosekiOda}.  When $G= \SU(2,2)$, these generalized Whittaker functions were investigated\footnote{They also appear in a forthcoming work of the author and Shrenik Shah on Beilinson's conjecture.} by Yamashita \cite{yamashita2}.  The works \cite{gon, miyazaki, moriyama} study generalized Whittaker functions, although not for pairs $(G,\pi)$ of quaternionic type.  Of course, there has been much work (e.g., \cite{shalika,stade,taniguchi1}) on understanding Whittaker functions associated to non-degenerate characters of a maximal unipotent subgroup of a group $G$, for principal series and discrete series.  In Theorem \ref{intro:mainThm} below, we give a complete, simple, and explicit description of a large class of both non-degenerate and degenerate generalized Whittaker functions.  The closest analogue of Theorem \ref{intro:mainThm} is thus the classical description of the generalized Whittaker functions for the holomorphic discrete series on groups with Hermitian tube structure, along the unipotent radical of the Siegel parabolic.

\subsection{The main results} We now setup the main results more precisely.  Suppose that $J$ is a cubic norm structure over $\R$, for which the trace pairing on $J$ is positive definite.  Then one can construct a certain simple Lie algebra\footnote{That one should consider the exceptional Lie algebras as morally being ``$\g_2(J)$'' we have learned from \cite[section 10]{ganSavin} and \cite{rumelhart}.} $\g(J)$, and its associated adjoint group $G_J = \mathrm{Aut}(\g(J))^{0}$.  The quaternionic adjoint groups of type $B_{\ell}$ with $\ell \geq 3$, $D_{\ell}$ with $\ell \geq 4$, $G_2, F_4, E_6, E_7, E_8$ are equal to $G_J$ with varying $J$.  More exactly, the correspondence between $J$ and these groups is as follows.  Denote by $H_3(C)$ the Hermitian $3 \times 3$ matrices over the composition algebra $C$.  Then:
\begin{enumerate}
\item $J = \R$, $G_J = G_2$
\item $J = H_3(\R)$, $G_J = F_4$
\item $J = H_3(\C)$, $G_J = E_6$
\item $J = H_3(\mathbf{H})$, $G_J = E_7$
\item $J = H_3(\Theta)$, with $\Theta$ the nonsplit octonions, $G_J = E_8$
\item $J = \R \times V$, with $V$ a quadratic space of signature $(1,N-3)$, $G_J = SO(N,4)$ if $N$ odd and $G_J = PSO(N,4)$ if $N$ is even.
\end{enumerate}

In these cases, the Levi subgroup $M$ of the Heisenberg parabolic $P$ is a similitude group $H_J$, with Hermitian symmetric space $\mathcal{H}_J = \{Z = X + i Y: X,Y \in J, Y >0\}$.  In the cases above, the group $H_J$ is (in order)
\begin{enumerate}
\item $\GL_2$
\item $\GSp_6$
\item closely related to $\GU(3,3)$
\item closely related to $\GU(3,3;\mathbf{H})$
\item $\mathrm{GE}_7$
\item $(\GL_2 \times \SO(N-2,2))/\mu$ where $\mu = \{1\}$ if $N$ odd and $\mu_2$ if $N$ even.\end{enumerate}

From now on we denote $G_J^{0} = G_J(\R)^{0}$.  In case $G_J$ is of type $G_2, F_4, E_6, E_7$ or $E_8$, $G_J(\R)$ is already connected; see \cite[Table II]{thang}.  The maximal compact subgroup of $G_J^{0}$ is of the form $(\SU(2) \times L_0(J))/\mu_2$ \cite{grossWallach2}, with $\mu_2$ embedded diagonally.  Here $L_0(J)$ is a compact form of the subgroup of $H_J$ with similitude equal to $1$.  

From now on, $n \geq 1$, $\chi: N(\R) \rightarrow \C^\times$ is a character, and $\Wh^\chi: G_J^0 \rightarrow \Vm_n^\vee$ is a smooth function of moderate growth, satisfying $\Wh^\chi(gk) = k^{-1} \cdot \Wh^\chi(g)$, $\Wh^\chi(ng) = \chi(n) \Wh^\chi(g)$, and $\mathcal{D}_n \Wh^\chi(g) =0$.  The main result Theorem \ref{intro:mainThm} is a formula for $\Wh^\chi(m)$ for $m$ in the Levi $H_J(\R)$. The formula is in terms of a cubic polynomial $p_\chi$ on $\mathcal{H}_J$ associated to $\chi$, the similitude $\nu: H_J(\R)\rightarrow \R^\times$, the parameter $Z \in \mathcal{H}_J$, and $K$-Bessel functions.

To setup this formula, we now describe $N^{ab} = N/[N,N]$.  In terms of $J$, the space $N^{ab} = \R \oplus J \oplus J \oplus \R =: W_J$, is Freudenthal's description of the defining representation of $H_J$.  This space has a symplectic form $\langle \;, \; \rangle$, which $H_J$ preserves up to similitude.  One writes $(a,b,c,d)$ for a typical element of $W_J$, so that $a,d \in \R$ and $b,c \in J$.  Define a character on $N(\R)$ via
\[\chi( \exp((u_1,u_2,u_3,u_4))[N,N]) = e^{ -i \langle (a,b,c,d), (u_1,u_2,u_3,u_4)\rangle}\]
for some nonzero $(a,b,c,d)  \in W_J$.  Here $(u_1,u_2,u_3,u_4) \in W_J$, and we are identifying the Lie algebra of $N/[N,N]$ with $W_J$.  Associated to the character $\chi$, define the cubic polynomial $p_\chi: \mathcal{H}_J \rightarrow \C$ via $p_\chi(Z) = a N(Z) + (b,Z^\#) + (c,Z) + d$.  Here $N$ is the cubic norm on $J$, and the other terms are quadratic, linear, and constant in $Z$.  See section \ref{sec:CNS} for the unexplained notation.

We can now state the formula for $\Wh^\chi(m)$. We have
\[\Wh^{\chi} = \sum_{-n \leq v \leq n}{\Wh^{\chi}_v \frac{x_{\ell}^{n+v}y_{\ell}^{n-v}}{(n+v)! (n-v)!}}\]
for functions $\Wh^{\chi}_v: G_J^{0} \rightarrow \C$ and a basis of weight vectors $x,y$ of $V_2(\C)$, the defining two-dimensional representation of $\SU(2)$.  There is a factor of automorphy $j: H_J(\R) \times \mathcal{H}_J \rightarrow \C;$ see section \ref{sec:CNS} for the definition of $j$. If $H_J$ is a classical group, and $g = \mm{A}{B}{C}{D}$, then (with appropriate normalizations) $j(g,Z) = \det(CZ +D) \nu(g)^{-1}$.  Finally, denote by $H_J(\R)^{\pm} = H_J(\R)^{0} \rtimes \langle w_0 \rangle$ for a certain $w_0 \in H_J(\R)$ with $\nu(w_0) = -1$ and $w_0^2 =1$.  (See section \ref{sec:final} for the definition of $w_0$.)  If $G_J$ has an exceptional Dynkin type, then $H_J(\R)^{\pm} = H_J(\R)$.  With this preparation, we have the following result.

\begin{theorem}\label{intro:mainThm} Let the notation be as above, so that the generalized Whittaker function $\Wh^\chi$ is of moderate growth and annihilated by $\mathcal{D}_n$.
\begin{enumerate}
\item Assume that the character $\chi$ is nontrivial.  
\begin{enumerate} 
\item If there exists $Z \in \mathcal{H}_J$ so that $p_\chi(Z) = 0$, then $\Wh^\chi(g) = 0$.
\item Suppose conversely that $p_\chi(Z)$ is never $0$ on $\mathcal{H}_J$.  Then $\Wh^\chi_v$ is a constant multiple (with constant of proportionality independent of $v$) of the function
\[\left(\frac{|j(g,i) p_\chi(Z)|}{j(g,i) p_\chi(Z)}\right)^{v} \nu(g)^{n}|\nu(g)| K_v(|p_\chi(Z)j(g,i)|) \]
for $g \in H_J(\R)^{\pm}$ and $Z = g \cdot i$ in $\mathcal{H}_J^{\pm}$. \end{enumerate}
\item Suppose that the character $\chi$ is trivial, so that $\Wh^\chi(n g) = \Wh^\chi(g)$ for all $n \in N(\R)$.  Then for $g \in H_J(\R)^{\pm}$,
\begin{equation}\label{eqn:MTconst}\Wh^{\chi}(g) = \nu(g)^{n}|\nu(g)|\left(\Phi(g) \frac{x^{2n}}{(2n)!} + \beta \frac{x^n y^n}{n! n!} + \Phi'(g) \frac{y^{2n}}{(2n)!}\right)\end{equation}
where $\beta \in \C$ is a constant, $\Phi'(g) = \Phi(gw_0)$, and $\Phi$ is associated to a holomorphic function of weight $n$ on $H_J$.
\end{enumerate}
\end{theorem}
In Theorem \ref{intro:mainThm}, that $\Phi$ is associated to a holomorphic function of weight $n$ on $H_J$ means that $j(g,i)^n \Phi(g)$ descends to the disconnected Hermitian symmetric space $\mathcal{H}_J^{\pm}$, is holomorphic on $\mathcal{H}_J^{+}$, and is antiholomorphic on $\mathcal{H}_J^{-}$ (see Proposition \ref{prop:const}.) 

The proof of Theorem \ref{intro:mainThm} is based upon making very explicit the expression $\mathcal{D}_n\Wh^\chi(g) = 0$.  To prove the formula, we write down and solve these differential equations.  Put another way, we write down in coordinates the differential equations satisfied by modular forms on $G_J^0/K$ (this is Theorem \ref{thm:Schmid2}), and use these equations to obtain the explicit form of the Fourier expansions. To make useful the expression $\mathcal{D}_n\Wh^\chi = 0$, we require completely explicit--but manageable--forms of the Cartan and Iwasawa decompositions of the quaternionic Lie algebras $\g(J)$, in \emph{well-chosen coordinates}.  We derive these decompositions and coordinates from an explicit, exceptional ``Cayley tranform'', see section \ref{sec:Cayley}.  While the analogues of these technical results are easy and well-known in the case of classical groups, finding good coordinates and making explicit computations is not as easy for the exceptional groups, like quaternionic $E_8$.  We expect that this explicit Cayley transform and its corollaries will be useful in other endeavors on the quaternionic exceptional groups. 

\subsubsection{Multiplicity one} One immediate corollary of Theorem \ref{intro:mainThm} concerns multiplicity one for the $\chi$-equivariant functionals on the representations $\pi_n$.
\begin{corollary}\label{cor:mult1} Suppose that the character $\chi$ is nontrivial and $\pi_n$ is a discrete series representation.  Then $\dim \mathrm{Hom}_{N(\R)}(\pi_n,\chi) \leq 1$.  That is, there is at most a one-dimensional space of $\chi$-equivariant functionals on $\pi_n$ of moderate growth.  Moreover, if $p_{\chi}$ is the cubic polynomial on $J \otimes \C$ associated to $\chi$, and $p_{\chi}(Z) = 0$ for some $Z$ in the Hermitian symmetric space $\mathcal{H}_J$, then $\mathrm{Hom}_{N(\R)}(\pi_n,\chi) = 0$.\end{corollary}

When the character $\chi$ is \emph{generic}, meaning that $\chi$ lives in the unique open orbit of $H_J(\C)$ on $W_J(\C)$, then the condition $p_{\chi} \neq 0$ on $\mathcal{H}_J$ can be related to Wallach's admissibility condition \cite{wallach}, and Corollary \ref{cor:mult1} is due to Wallach.  It is perhaps surprising that the multiplicity one result of \emph{loc.  cit.} continues to hold even when $\chi$ is nontrivial but not generic.  The same phenomenon occurs when $\pi_n$ is replaced by a holomorphic discrete series representation on a group with a Hermitian tube structure, and the characters are along the unipotent radical of the Siegel parabolic.

\subsubsection{The Fourier expansion} The main consequence of Theorem \ref{intro:mainThm} is the Fourier expansion of modular forms of weight $n$ on quaternionic exceptional groups.  Suppose $G$ is an adjoint reductive $\Q$-group of type $G_2, F_4, E_6, E_7$ or $E_8$, with $G(\R)$ quaternionic, and suppose that $G$ contains a rational Heisenberg parabolic $P$.  Suppose moreover that $F$ is a modular form of weight $n \geq 1$ for $G$.  Denote by $N_0 = [N,N]$ the one-dimensional center of the unipotent radical of the Heisenberg parabolic and set
\[F^0(g) = \int_{N_0(\Q)\backslash N_0(\A)}{F(ng)\,dn}.\]

If $w = (a,b,c,d) \in W_J$, set $p_{w}(Z) = aN(Z) + (b,Z^\#) + (c,Z) +d$, and write $p_w \geq 0$ if $p_w(Z)$ is never $0$ on $\mathcal{H}_J$.  Define
\[\Wh^w_v(g) = \left(\frac{|j(g,i) p_w(Z)|}{j(g,i) p_w(Z)}\right)^{v} \nu(g)^{n}|\nu(g)| K_v(|p_w(Z)j(g,i)|) \]
for $Z = g \cdot i$ in $\mathcal{H}_J^{\pm}$ and set $\Wh^w(g) = \sum_{-n \leq v \leq n}{\Wh^w(g) \frac{x^{n+v}y^{n-v}}{(n+v)!(n-v)!}}$.
\begin{corollary}\label{intro:Cor}  Let the notation be as above.  There are complex numbers $a_{F}(w)$ for $w \in 2\pi W_J(\Q)$ so that for $x \in W_J(\R)$ (identified with $\mathrm{Lie}(N/N_0)$) and $g \in H_J(\R)$
\[F^{0}\left(\exp(x)g\right) = F^{00}(g) + \sum_{w \in 2\pi W_J(\Q)}{a_{F}(w) e^{-i \langle w, x\rangle} \Wh^{w}(g)}\]
where the sum is over those $w$ with $p_w \geq 0$.  Here the constant term $F^{00}(g)$ is of the form (\ref{eqn:MTconst}), with $\Phi$ associated to a holomorphic modular form of weight $n$ on $H_J$.  If $\varphi$ is a cusp form, then $F^{00}(g) = 0$ and $a_{F}(w) \neq 0$ implies that $w \in W_J$ is rank four and Freudenthal's quartic form $q$ on $w$ is negative.\end{corollary}

Below we check that if $w \in W_J$ is not of rank four, then the functions $\Wh_v^{w}(g)$ are not bounded.  Thus the coefficients $a_{F}(w)$ for these $w$ must vanish when $F$ is a cusp form.  Also, since there are different sign conventions about Freudenthals quartic form on $W_J$, we give now an invariant description: If $w \in W_J(\R)$ and $q(w) > 0$, then there exists $w', w'' \in W_J(\R)$ each of rank one with $w = w' + w''$.  If $w \in W_J(\R)$ and $q(w) < 0$, then there exist $w', w''$ with $w', w'' \in W_J(\C)$ complex conjugates and rank one and $w = w' + w''$, but there does not exist such a decomposition in $W_J(\R)$.

Theorem \ref{intro:mainThm} and Corollary \ref{intro:Cor} should be compared to the well-known Fourier expansions of holomorphic modular forms on Hermitian tube domains, such as the Fourier expansion of Siegel modular forms (special automorphic forms for the group $\GSp_{2n}$) on the Siegel upper half-space described above.  On the quaternionic groups $G(\R)$ studied in this paper, the generalized Whittaker functions of Theorem \ref{intro:mainThm} thus play the role that the exponential function $e^{-2 \pi \tr(TY)}$ plays on groups with Hermitian tube structure.  It is our hope that via making explicit what the Fourier expansions of modular forms on quaternionic exceptional groups look like, we might encourage others to investigate these objects.

\subsubsection{Klingen Eisenstein series} The Siegel modular forms on $\GSp_{2n}$ possess ``Klingen''-type Eisenstein series.  That is, if $\phi$ is a holomorphic modular cusp form on $\GSp_{2n-2}$, then one can form an Eisenstein series on $\GSp_{2n}$ by inducing $\phi$, and if the weight of $\phi$ is sufficiently large, the Eisenstein series (at a special point) is again a holomorphic modular form on $\GSp_{2n}$.  The modular forms on $G_J$ have the same property.  Namely, recall that the Levi subgroup $H_J$ of the Heisenberg parabolic on $G_J$ always has Hermitian tube structure.  Thus, if $\Phi$ is a holomorphic modular cusp form on $H_J$ of weight $n$, one can form an Eisenstein series out of $\Phi$, to get an automorphic form on $G_J$.  These Eisenstein series turn out to be modular forms of weight $n$ on $G_J$, which is another corollary of Theorem \ref{intro:mainThm}.

\begin{corollary}[See Proposition \ref{prop:Eis}]\label{cor:Klingen} Suppose $n > \dim W_J$ and $\Phi$ is a cusp form on $H_J$ associated to a holomorphic modular form of weight\footnote{If $J = \R$ so that $G_J = G_2$ and $H_J \simeq \GL_2$, a weight $n$ modular form for $H_J$, in our normalization, becomes a holomorphic modular form of weight $3n$ in the classical normalization on $\GL_2$.  Thus, for example, the Klingen Eisenstein series associated to the unique level one cusp form of weight $18$ on $\GL_2$ gives rise to a modular form of weight $6$ on $G_2$, i.e., a function $F: G_2(\Q)\backslash G_2(\A) \rightarrow Sym^{12}(V_2)^\vee$ in the kernel of $\mathcal{D}_6$.} $n$.  Suppose that $f_{\Phi}(g): G_J(\A) \rightarrow \Vm_n^\vee$ satisfies $f_{\Phi}(ngk) = k^{-1} f_{\Phi}(g)$ for all $n \in N(\A)$ and $k \in K \subseteq G_J(\R)$.  Furthermore, suppose that for $g \in H_J(\A)$,
\[f_{\Phi}(g) =\nu(g)^{n}|\nu(g)|\left(\Phi(g) \frac{x^{2n}}{(2n)!} +\beta \frac{x^n y^n}{n! n!}+ \Phi'(g) \frac{y^{2n}}{(2n)!}\right)\]
with $\Phi'(g) = \Phi(gw_0)$ and $\beta \in \C$ a constant.  Then $E(g,f_{\Phi}) := \sum_{\gamma \in P(\Q)\backslash G_J(\Q)}{f_{\Phi}(\gamma g)}$ converges absolutely, and is a modular form of weight $n$ on $G_J$. \end{corollary}

Of course, embeddings of discrete series representations into induced representations is a well-studied topic.  The case $\Phi = 0$ in Corollary \ref{cor:Klingen} corresponds to embedding $\pi_n$ into the degenerate principal series $Ind_{P}^{G}(\nu^{n} |\nu|)$, which can be found in \cite[Section 6]{wallach}.  The case $\beta =0 $ in Corollary \ref{cor:Klingen} corresponds to embedding $\pi_n$ into an induction of a holomorphic discrete series representation on $H_J$; see, e.g., \cite{blankInd}.  What is new about this corollary is the complete explicitness of the relevant inducing section when $\Phi \neq 0$.

In Corollary \ref{cor:Klingen}, if $\Phi = 0$, then one only needs $n > \frac{1}{2} \dim W_J$ for the statement of the corollary to be correct.  In our parlance, the condition that $\Phi$ be associated to a modular form of weight $n$ on $H_J$ implies that $\Phi(zg) = z^{-n}\Phi(g)$ for all $z \in Z_H(\R)^{0}$, where $Z_H$ is the connected center of the similitude group $H_J$.  This central character is why the convergence properties change when $\Phi = 0$.

As an example of the various corollaries, suppose that $J = H_3(\Theta)$ is the exceptional Jordan algebra with positive-definite trace form.  Then $H_J = \mathrm{GE}_7$ has Hermitian tube structure, and out of a holomorphic modular cusp form $\Phi$ on $\mathrm{GE}_7$ of weight $n > 56$, one can construct a Klingen Eisenstein series $E(g,f_{\Phi})$ on $G_J$, the quaternionic form of $E_8$.  By Corollary \ref{cor:Klingen}, this Klingen Eisenstein series is a modular form of weight $n$ on $G_J = E_8$, and thus has a Fourier expansion given as in Corollary \ref{intro:Cor}.

\subsection{Outline of paper} We now give an outline of the sections of the paper.  In section \ref{sec:CNS}, we recall cubic norm structures, certain groups $M_J$ and $H_J$ associated to them, and describe the Hermitian symmetric space $\mathcal{H}_J$ associated to $H_J$.  In section \ref{sec:LieI}, we discuss the Lie algebras of $M_J$, $H_J$, and the Cartan involution on them.  In section \ref{sec:LieII}, we define the Lie algebra $\g(J)$ and the group $G_J$, following the constructions of \cite{freudenthalIV} and \cite{rumelhart}.  In section \ref{sec:Cayley}, we work out an explicit Cayley transform for the quaternionic groups $G_J(\R)$.  In section \ref{sec:Cartan} we derive explicit Cartan and Iwasawa decomposition for $\g(J)$.  In section \ref{sec:diffEq} we write down the differential equations $\mathcal{D}_n \Wh^\chi = 0$, and in section \ref{sec:solutions} we solve these equations.  In section \ref{sec:final}, we deduce the exact statement of the $\chi$-nontrivial part Theorem \ref{intro:mainThm}, while in section \ref{sec:positivity}, we prove some results about the polynomial $p_\chi(Z)$, and Wallach's admissibility condition (see \cite{wallach}).  In section \ref{sec:constAndEis}, we deduce the $\chi$-trivial part of Theorem \ref{intro:mainThm}, and also discuss Klingen Eisenstein series. Finally, in appendix \ref{sec:orthog}, we make precise the way that orthogonal groups are of the form $G_J$, for appropriate $J$.  

\subsection{Acknowledgments} We thank Wee Teck Gan and Benedict Gross, for asking us about the constant term of modular forms and limit discrete series representations, respectively.  The treatment of these aspects was omitted in the first version of this manuscript.  During the period in which this work was done, the author was supported by the Schmidt fund at the Institute for Advanced Study.  We thank the IAS for its hospitality and for providing an excellent working environment.

\section{Cubic norm structures and associated groups}\label{sec:CNS} We now briefly recall cubic norm structures, and certain algebraic groups defined in terms of them.  We also discuss the Hermitian symmetric space associated to the group $H_J$.  The reader might see \cite[Sections 4.2 and 4.3]{pollackLL}, and the references contained therein, for some more detail, and \cite{rumelhart} for a slightly different approach to some of the same underlying objects.  The reader might also see \cite{bhargavaHo}, \cite{grossSavin}, and \cite{ganSavin} for the use of some of the same algebraic structures in different scenarios.

Below, $F$ denotes any arbitrary ground field of characteristic $0$.  When we discuss Cartan involutions, $F$ will always be $\R$ and the cubic norm structure $J$ over $\R$ will be assumed to have a positive definite trace pairing (see below).

\subsection{Cubic norm structures} First, we recall the notion of a cubic norm structure, as a way to fix notations.  Suppose $F$ is a field of characteristic $0$ and $J$ is a finite dimensional $F$ vector space.  That $J$ is a cubic norm structure means that it comes equipped with a cubic polynomial map $N: J \rightarrow F$, a quadratic polynomial map $\#: J \rightarrow J$, an element $1_J \in J$, and a symmetric bilinear pairing $(\;,\;): J \otimes J \rightarrow F$, called the trace pairing, that satisfy the following properties.  For $x, y \in J$, set $x \times y = (x+y)^\# - x^\# - y^\#$ and denote $(\;,\;,\;): J \otimes J \otimes J \rightarrow F$ the unique symmetric trilinear form satisfying $(x,x,x) = 6N(x)$ for all $x \in J$.  Then
\begin{enumerate}
\item $N(1_J) = 1$, $1_J^\# = 1_J$, and $1_J \times x = (1_J,x) -x$ for all $x \in J$.
\item $(x^\#)^\# = N(x) x$ for all $x \in J$.
\item The pairing $(x,y) = (1_J,1_J,x)(1_J,1_J,y) - (1_J,x,y)$.
\item One has $N(x+y) = N(x) + (x^\#,y) + (x,y^\#) + N(y)$ for all $x, y \in J$.\end{enumerate}
One should see \cite{mccrimmon} for a thorough treatment of cubic norm structures.

There is a weaker notion of a cubic norm pair.  In this case, the pairing $(\;,\;)$ is between $J$ and $J^\vee$, the linear dual of $J$, the adjoint map $\#$ takes $J \rightarrow J^\vee$ and $J^\vee \rightarrow J$, and each $J, J^\vee$ have a norm map $N_J: J \rightarrow F$ and $N_{J^\vee}:J^\vee \rightarrow F$.  The adjoints and norms on $J$ and $J^\vee$ satisfy the same compatibilities as above in items (2) and (4).

If $J$ is a cubic norm structure, and $x \in J$, one defines the map $U_x: J \rightarrow J$ via $U_x(y) = - x^\# \times y + (x,y)x$.  It is a fact that $N(U_x(y)) = N(x)^2N(y)$.

If $J$ is a cubic norm structure, we define the group 
\[M_J = \{(\lambda, g) \in \GL_1 \times \GL(J): N(g X) = \lambda N(X) \text{ for all } X \in J\},\]
the group of all linear automorphisms of $J$ that preserve the norm $N$ up to scaling.  Thus if $x \in J$ with $N(x) \neq 0$, the map $U_x$ defines an element of $M_J$.  We set $M_J^{1}$ the subgroup of $M_J$ consisting of those $g$ with $\lambda(g) =1$ and we set $A_J$ the subgroup of $M_J^{1}$ that also fixes the bilinear pairing $(\;,\;)$: if $a \in A_J$, then $(ax,ay) = (x,y)$ for all $x, y \in J$.  It is in fact equivalent to define $A_J$ as the stabilizer of the element $1_J \in J$.  The group $A_J$ is the automorphism group of $J$.  If $a \in A_J$, then one also has $(ax) \times (ay) = a(x\times y)$ for all $x,y \in J$.

If $m \in M_J$, then we denote by $\tilde{m}$ the action of $m$ on the dual representation $J^\vee$, so that $(mx, \tilde{m}y) = (x,y)$ for all $x \in J$ and $y \in J^\vee$.

\subsection{The group $H_J$ and the Freudenthal construction}\label{subsec:FWJ} We now briefly recall the Freudenthal construction.  If $J$ is a cubic norm structure over a field $F$, of if $(J, J^\vee)$ is a cubic norm pair, one sets 
\[W_J = F \oplus J \oplus J^\vee \oplus F.\]
This is an $F$ vector space that comes equipped with symplectic and quartic forms essentially due to Freudenthal.  We write a typical element of $W_J$ as $(a,b,c,d)$, so that $a,d \in F$, $b \in J$ and $c \in J^\vee$.  The symplectic form is
\[\langle (a,b,c,d), (a',b',c',d') \rangle = ad'-(b,c') + (c,b') - da'.\]
The quartic form $q$ is
\[q((a,b,c,d)) = (ad-(b,c))^2 + 4aN(c) + 4dN(b) - 4(b^\#,c^\#).\]
Associated to the quartic form is a symmetric four-linear form that is normalized by the condition $(v,v,v,v) = 2q(v)$ for $v \in W_J$.  Because the symplectic form is non-degenerate, there is a trilinear map $t: W_J \otimes W_J \otimes W_J \rightarrow W_J$ defined by $\langle w, t(x,y,z) \rangle = (w,x,y,z)$ for all $w,x,y,z \in W_J$.  One sets $v^\flat = t(v,v,v)$.

One defines the group
\[H_J = \{(\nu,g) \in \GL_1 \times \GL(W_J): \langle gv,gw \rangle = \nu \langle v,w \rangle \text{ and } q(gv) = \nu^2 q(v) \text{ for all } v,w \in W_J\},\]
the group preserving the symplectic and quartic form on $W_J$ up to appropriate similitude.  We write $H_J^{1}$ for the subgroup of the $g \in H_J$ with $\nu(g) =1$.

If $x \in J$, then we set $n(x): W_J \rightarrow W_J$ the map given by 
\[n(x)(a,b,c,d) = (a,b + ax,c + b \times x + ax^\#, d + (c,x) + (b,x^\#) + aN(x)).\]
If $y \in J^\vee$, we set $n^{\vee}(y): W_J \rightarrow W_J$ the map given by
\[n^\vee(y)(a,b,c,d) = (a+(b,y) + (c,y^\#) + dN(y), b + c\times y, c + dy,d).\]
If $m \in M_J$ and $\delta$ in $\GL_1$ with $\delta^2 = \lambda(m)$, then we set $M(\delta,m): W_J \rightarrow W_J$ the map defined by 
\[M(\delta,m)(a,b,c,d) = (\delta^{-1}a,\delta^{-1}m(b), \delta \tilde{m}(c),\delta d).\]
One has that $n(x), n^\vee(y)$ and $M(\delta, m)$ are in $H_J^{1}$ for all $x,y,\delta, m$ as above.  If $z \in \GL_1$, we define $\eta(z) = M(z^{-3}, U_{z^{-1}})$.  Then $\eta(z)$ acts on $W_J$ as $\eta(z)(a,b,c,d) = (z^3 a, zb,z^{-1}c,z^{-3}d)$.  If $J^\vee$ is identified with $J$ so that $J$ is a cubic norm structure, define $J_2: W_J \rightarrow W_J$ via
\[J_2(a,b,c,d) = (d,-c,b,-a).\]
Then $J_2$ is in $H_J^{1}$.

Below we will sometimes refer to the \emph{rank} of an element $J$ or $W_J$; this notion is recalled, for instance, in \cite[Definition 4.2.9 and Definition 4.3.2]{pollackLL}.

\subsection{The Hermitian symmetric space associated to $H_J$} When the ground field $F = \R$ and the pairing $(\;,\;)$ on $J$ is positive definite, the group $H_J$ has a Hermitian symmetric space.  This space is $\mathcal{H}_J = \{Z = X + iY: Y > 0\}$.  Here $Y > 0$ means that $Y= U_y 1_J$ for some $y \in J$.  Equivalently, $\tr(Y) > 0, \tr(Y^\#) > 0$ and $N(Y) > 0$. We now recall how $H_J(\R)^{0}$ acts on $\mathcal{H}_J$.

To do this, suppose $Z \in J_{\C}$.  Define $r_0(Z) = (1,-Z,Z^\#,-N(Z)) = n(-Z)(1,0,0,0)$.  Then one has the following proposition, which must be well-known.
\begin{proposition}\label{prop:Hcal} Suppose $Z \in \mathcal{H}_J$, so that $Im(Z)$ is positive definite.  Suppose moreover that $g \in H_J(\R)^0$.  Then there is $j(g,Z) \in \C^\times$ and $gZ \in \mathcal{H}_J$ so that $g r_0(Z) = j(g,Z)r_0(gZ)$.  This equality defines the factor of automorphy $j(g,Z)$ and the action of $H_J(\R)^0$ simultaneously.\end{proposition}
Both for this proposition, and below, we will need the following lemma.  Suppose $J$ is a cubic norm structure.  Define a pairing on $W_J$ via $(v,w) = \langle J_2 v, w\rangle$.
\begin{lemma}\label{lem:sympair} The pairing $(v,w)$ is symmetric, with $(v,v) = a^2 + (b,b) + (c,c) + d^2$ if $v=(a,b,c,d)$.  Thus if the trace pairing on $J$ is positive-definite, the pairing $(\;,\;)$ on $W_J$ is as well.  Furthermore, one has $|\langle r_0(i),v \rangle|^2 = (v,v) + 2\tr(b^\#-ac) + 2\tr(c^\# -db)$.  Thus if $v$ is rank one, $|\langle r_0(i),v \rangle|^2 = (v,v).$\end{lemma}
\begin{proof} The first part of the lemma is immediate from the definitions.  For the second part,  suppose $v = (a,b,c,d)$ and recall $r(i) = (1,-i,-1,i)$.  Then $\langle v,r(i) \rangle = (\tr(b)-d) + i(a - \tr(c))$, and thus
\begin{align*} |\langle v,r(i)\rangle|^2 &= (\tr(b)-d)^2 + (a-\tr(c))^2 \\ &= d^2 - 2d\tr(b) + \tr(b)^2 + a^2 -2a\tr(c) + \tr(c)^2 \\ &= d^2 + (b,b) + a^2 + (c,c) + 2\tr(b^\#-ac) + 2\tr(c^\# -db),\end{align*}
as desired.  Here we have used the identity $\tr(x)^2 = \tr(x^2) + 2\tr(x^\#)$ for $x \in J$.  If $v$ is rank one, then $b^\#-ac = 0$ and $c^\#-db = 0$, and the lemma follows.\end{proof}

\begin{proof}[Proof of Proposition \ref{prop:Hcal}] We recall the argument sketched in \cite[Section 6.2.1]{pollackSMF}, since this argument was only applied when $H_J$ was of type $D_6$, and not in general in \emph{loc cit}.  First, if $g \in H_J^{1}(\R)$ one has $\langle g r_0(i), (0,0,0,1) \rangle \neq 0$, since $\langle g r_0(i), (0,0,0,1) \rangle =  \langle r_0(i), g^{-1} (0,0,0,1) \rangle$ and $|\langle r_0(i), g^{-1} (0,0,0,1) \rangle|^2 \neq 0$ by Lemma \ref{lem:sympair} since $g^{-1}(0,0,0,1)$ is rank one.  Thus, there is $j(g,i) \in \C^\times$ and $Z \in J_{\C}$ so that $g r_0(i) = j(g,i)r_0(Z)$.  We claim that $N(Im(Z)) > 0$.  

To see this, first note the general identity $\langle r_0(Z),r_0(W) \rangle = N(Z-W)$ for $Z, W \in J_{\C}$.  Thus
\[\nu(g) N(2i) = \nu(g) \langle r_0(i), r_0(-i) \rangle = \langle gr_0(i), g r_0(-i) \rangle = \langle gr_0(i), \overline{gr_0(i)}\rangle = |j(g,i)|^2 N(Z - \overline{Z}).\]
Thus if $\nu(g) > 0$, then $N(Y) > 0$.

Now, if $g \in H_J(\R)^{0}$, then $N(Im(g i))> 0$, and thus by continuity $Im(gi)$ is positive definite.  Hence the proposition is proved when $Z = i 1_J$.  The general case follows from the fact that the subgroup generated by the $M(N(y),U_y)$ and $n(X)$ acts transitively on $\mathcal{H}_J$. \end{proof}

\section{Lie algebras of exceptional groups, I}\label{sec:LieI}  In this section we recall the concrete description of the Lie algebras of the classical groups $\SO(V)$ and $\Sp(W)$, and also recall the concrete description of the Lie algebras of $M_J^{1}$ and $H_J^{1}$.  Below, if $\g$ is a Lie algebra and $B$ a symmetric $\g$ invariant pairing, and $\Theta: \g \rightarrow \g$ an involution, $B_{\Theta}$ denotes the pairing defined as $B_{\Theta}(X,Y) = - B(X,\Theta(Y))$.

\subsection{The Lie algebra of symplectic groups} In this section we describe the Lie algebra of $\Sp(W)$, the group preserving a symplectic form $\langle \;,\; \rangle$ on the finite dimensional vector space $W$.

\subsubsection{The commutator} First, we recall that the Lie algebra $\sp(W)$ of $\Sp(W)$ is $Sym^2(W)$.  If $w, w' \in W$, then $w \cdot w' \in Sym^2(W)$ acts on $W$ via $(w \cdot w')(x) = \langle w',x\rangle w + \langle w, x \rangle w'$.  The commutator is given by
\[ [w_1 \cdot w_1', w_2 \cdot w_2'] = \langle w_1', w_2 \rangle w_1 \cdot w_2' + \langle w_1', w_2' \rangle w_1 \cdot w_2 + \langle w_1, w_2 \rangle w_1' \cdot w_2' + \langle w_1, w_2' \rangle w_1' \cdot w_2.\]
If $\phi \in \sp(W)$, then $[\phi, w \cdot w'] = \phi(w) \cdot w' + w \cdot \phi(w')$.

\subsubsection{The Killing form} Define
\[B_{\sp}(w_1 w_1', w_2 w_2') = -2\left(\langle w_1, w_2 \rangle \langle w_1', w_2' \rangle + \langle w_1',w_2 \rangle \langle w_1, w_2' \rangle\right).\]
Then $B_{\sp}$ is a symmetric $\sp(W)$-invariant form; if the ground field $F = \R$, then $B_{\sp}$ is a positive multiple of the Killing form.

If $W= Fe + Ff$ is two-dimensional, with symplectic basis $e, f$ so that $\langle e, f \rangle = 1$, then an $\mathfrak{sl}_2$-triple is
\[E=\mm{0}{1}{0}{0} = \frac{e^2}{2},\;\; H= \mm{1}{0}{0}{-1}  = -ef,\;\; F=\mm{0}{0}{1}{0} = - \frac{f^2}{2}.\]
One has $[E,F] = H$, $[H,E]=2E$, $[H,F] = -2F$.  The form $B$ is normalized so that $B_{\sp}(E,F) = 1$.

\subsubsection{The Cartan involution} Suppose $W = F^{2n}$, $J_n = \mm{0}{1_n}{-1_n}{0}$, and $\Sp(W) = \{g \in Aut(W): \,^tg J_n g = J_n\}$.  In other words, assume that the sympletic form on $W$ is defined by $\langle w_1, w_2 \rangle = \,^tw_1 J_n w_2$ for $w_1, w_2$ column vectors in $W=F^{2n}$.  Then $J_n \in \Sp(W)$.  This induces an involution $\Theta$ on $Sym^2(W)$ via $\Theta( ww') = (J_nw) (J_n w')$.  If the ground field $F = \R$, this is a Cartan involution and $(w_1, w_2) := \langle J_n w_1, w_2\rangle$ defines a symmetric positive definite form on $W$.

\subsection{The Lie algebra of orthogonal groups} Suppose that $V$ is a non-degenerate quadratic space.  In this subsection, we briefly recall facts pertaining to the Lie algebra $\so(V)$ of $\SO(V)$.  
\subsubsection{The commutator} One has $\so(V) \simeq \wedge^2 V$.  In this identification, an element $w \wedge x$ acts on $V$ via
\[w \wedge x (v) = (x,v) w - (w,v) x.\]
The Lie bracket in this notation is
\begin{equation}\label{bracket}[w \wedge x, y \wedge z] = (x,y) w\wedge z - (x,z) w \wedge y - (w,y) x \wedge z + (w,z) x \wedge y.\end{equation}

\subsubsection{The Killing form} Define
\[B_{\so}( w \wedge x, y \wedge z) = (x,y)(w,z) - (w,y)(x,z).\]
This is a symmetric $\so(V)$ invariant form on $\so(V)$; if the ground field $F=\R$, it is a positive multiple of the Killing form.

\subsubsection{The Cartan involution} Suppose $F = \R$.  Suppose $\iota: V \rightarrow V$ is an involution, for which the quadratic form $(v,v)$ is positive definite on the subspace of $V$ for which $\iota$ is $+1$, and is negative definite where $\iota$ is $-1$.  Further assume that $\iota$ defines an element of the orthogonal group $\mathrm{O}(V)$.  Then $(v,\iota(w))$ is a positive definite symmetric bilinear form on $V$. Associated to $\iota$, one can define a Cartan involution $\Theta_{\iota}$ on the Lie algebra $\so(V) \simeq \wedge^2V$.  Namely, one sets $\Theta_{\iota}: \wedge^2 V \rightarrow \wedge^2 V$ via $\Theta_{\iota}(v \wedge w) = \iota(v) \wedge \iota(w)$.  

\subsection{The Lie algebra of $M_J$ and $A_J$} In this subsection we discuss the Lie algebras $\a(J)$ and $\m(J)$ of $A_J$ and $M_J$ respectively.  The reader might see \cite[section 1.2]{rumelhart} for a slightly different take on the same Lie algebras.

The Lie algebra $\m(J)$ consists of those $(\phi,\mu) \in End(J) \oplus \mathrm{G}_a$ so that
\[(\phi(z_1),z_2,z_3) + (z_1,\phi(z_2),z_3) + (z_1,z_2,\phi(z_3)) = \mu(z_1,z_2,z_3)\]
for all $z_1, z_2, z_3 \in J$.  The subalgebra $\m(J)^{0}$ consisting of those $\phi \in \m(J)$ with $\mu =0$ is the Lie algebra of $M_J^{1}$.  The Lie algebra $\a(J)$ consists of those $X \in \m(J)^{0}$ for which $(X z_1, z_2) + (z_1,X z_2) = 0$ for all $z_1, z_1 \in J$.  Equivalently, $\a(J)$ consists of those $X \in \m(J)$ for which $X 1_J = 0$.

We will shortly describe the internal structure of $\m(J)$ in more detail.  Before doing so, we recall some particular endomorphisms of $J$, whose definition goes back at least to Freudenthal \cite{freudenthalI, freudenthalII}.

To define these endomorphisms, suppose $\gamma \in J^\vee$ and $x \in J$.  Then $\Phi_{\gamma,x} \in End(J)$ is defined via
\[\Phi_{\gamma,x}(z) = - \gamma \times (x \times z) + (\gamma,z)x + (\gamma,x)z.\]
One has
\[(\Phi_{\gamma,x}(z_1),z_2,z_3) + (z_1,\Phi_{\gamma,x}(z_2),z_3) + (z_1,z_2,\Phi_{\gamma,x}(z_3)) = 2(\gamma,x)(z_1,z_2,z_3)\]
for all $z_1, z_2, z_3$ in $J$.  Thus $\Phi_{\gamma,x} \in \m(J)$.  Note that $\Phi_{\gamma,x}(z) = \Phi_{\gamma,z}(x)$.  One sets $\Phi'_{\gamma,x} = \Phi_{\gamma,x} - \frac{2}{3}(\gamma,x)$.  Then $\Phi'_{\gamma,x} \in \m(J)^{0}$.

We write $\widetilde{\Phi_{\gamma,x}}$ for endomorphism of the dual space $J^\vee$ induced by $\Phi_{\gamma,x}$. One has
\[\widetilde{\Phi_{\gamma,x}}(\mu) = x \times (\gamma \times \mu) - (x,\mu)\gamma - (x,\gamma)\mu.\]
With this action $(\Phi_{\gamma,x}(z),\mu) + (z,\widetilde{\Phi_{\gamma,x}}(\mu)) = 0$ for all $z \in J$ and $\mu \in J^\vee$.

Denote by $\iota$ the identification\footnote{Below, we abuse notation and use the same letter $\iota$ to denote an identification of a finite dimensional vector space with its dual given by a fixed non-degenerate bilinear form.  No confusion should arise, as the domain of $\iota$ and the bilinear form defining it will always be clear from the context.} $J \simeq J^\vee$ induced by the trace pairing $(\;,\;)$ on $J$.  If $X,Y \in J$, define
\[\Phi_{X \wedge Y} = \Phi_{\iota(X),Y} - \Phi_{\iota(Y),X}.\]
Then one has $\Phi_{X \wedge Y} \in \a(J)$.

\subsubsection{The Jordan product} The Jordan product on $J$ can be defined through the operators $\Phi_{\gamma,x}$.  More precisely, for $X,Y \in J$, set
\[\{X,Y\} := \Phi_{\iota(1_J),X}(Y) = \Phi_{\iota(1_J),Y}(X) = \{Y,X\}.\]
Then when $J$ is a \emph{special} cubic norm structure, so that $J$ is embedded in associative algebra $A$ in such a way that $U_x y = xyx$ for all $x,y \in J$, then $\{X,Y\} = XY + YX$.  Again for $J$ special, one has $\Phi_{\iota(X),Y}(z) = YXz + zXY$.

We sometimes write $\{Z,\bullet\}$ as shorthand for the map $J \rightarrow J$ given by $x \mapsto \{Z,x\}$.  Thus $\{Z,\bullet\} = \Phi_{\iota(1_J),Z}$.   One has the identities
\begin{equation}\label{eqn:Jord+}\Phi_{\iota(X),Y}(Z) + \Phi_{\iota(Y),X}(Z) = \{\{X,Y\},Z\}.\end{equation}
and
\begin{equation}\label{eqn:Jordwedge}[\Phi_{X\wedge Y}, \{Z,\bullet\}] = \{\Phi_{X \wedge Y}(Z),\bullet\}\end{equation}
for all $X, Y, Z \in J$.  For $J$ special, one has 
\[\Phi_{X \wedge Y}(Z) = -[XY-YX,Z] = YXZ+ZXY - XYZ-ZYX.\]

\subsubsection{The Lie algebra $\m(J)$} We now recall the description of the Lie algebra $\m(J)$.  One has $\m(J) = \a(J) \oplus J$ and $\m(J)^{0} = \a(J) \oplus J^0$, where $J^0 \subseteq J$ denotes the subspace of trace $0$ elements.

If $X \in J$, then we identify $X$ with the element $\Phi_{\iota(1_J),X} = \{X,\bullet\}$ of $\m(J)$.  Thus one obtains a map $\a(J) \oplus J \rightarrow \m(J)$, which turns out to be an isomorphism.  The commutator of $\phi \in a(J)$ with $X\in J$ is $\phi(X)$, i.e.,
\[[\phi,\Phi_{\iota(1_J),X}] = \Phi_{\iota(1_J),\phi(X)}.\]
The commutator of $X,Y \in J$ is $\Phi_{Y\wedge Z}$, i.e.,
\[[\Phi_{\iota(1),X},\Phi_{\iota(1),Y}]= \Phi_{Y \wedge X}.\]
The element $\Phi_{\gamma,x}$ decomposes in $\m(J) = \a(J) \oplus J$ as
\begin{equation}\label{eqn:12sum}\Phi_{\gamma,x} = \frac{\Phi_{\gamma,x}-\Phi_{\iota(x),\iota(\gamma)}}{2} + \frac{\Phi_{\gamma,x}+\Phi_{\iota(x),\iota(\gamma)}}{2} = \frac{1}{2}\Phi_{\iota(\gamma)\wedge x} + \frac{1}{2}\{\{\iota(\gamma),x\},\bullet\}.\end{equation}

\subsubsection{The Killing form}  We now define an invariant pairing on $\m(J)$. Because the $\Phi_{\gamma,x}$ span $\m(J)$, it suffices to define a pairing on these elements.  One sets
\[B_{\m}(\Phi_{\gamma,x},\Phi_{\gamma',x'}) = (\gamma,x)(\gamma',x') + (\gamma,x')(\gamma',x) - (\gamma \times \gamma', x \times x').\]
If $\phi \in \m(J)$, then $B_{\m}(\phi,\Phi_{\gamma,x}) = (\phi(x),\gamma)$.  

One has
\[B_\m(\Phi'_{\gamma,x},\Phi'_{\gamma',x'}) = B_\m(\Phi_{\gamma,x},\Phi_{\gamma',x'}) - \frac{2}{3}(\gamma,x)(\gamma',x') = \frac{1}{3}(\gamma,x)(\gamma',x') + (\gamma,x')(\gamma',x) - (\gamma \times \gamma', x \times x').\]

If $z_1, z_2 \in J \subseteq \m(J)$, then
\begin{align}\label{eqn:B0J} \nonumber B_\m(z_1,z_2) &= B_\m(\Phi_{1,z_1},\Phi_{1,z_2}) = 3(z_1,z_2) + \tr(z_1)\tr(z_2) - 2(1, z_1 \times z_2) \\ \nonumber &= 5(z_1,z_2) - \tr(z_1)\tr(z_2) \\ &= \frac{2}{3} \tr(z_1)\tr(z_2) + 5(z_1^{0},z_2^{0}).\end{align}
Here $z^0 = z - \frac{\tr(z)}{3}$ is the trace $0$ projection of $z$.

If $w,x,y,z\in J$, then one computes
\begin{equation}\label{eqn:B0wedge}\frac{1}{2}B_\m(\Phi_{w \wedge x},\Phi_{y \wedge z}) = (w,z)(x,y)-(x,z)(w,y) + (x \times y,w \times z) - (w\times y,x \times z).\end{equation}

\subsubsection{The Cartan involution} Recall that the pairing on $J$ gives rise to an $\iota: J\rightarrow J^\vee$ and thus an involution $\Theta_{\m}$ on $\m(J)$ via $\Theta_{\m}(\phi) = \iota^{-1} \circ \widetilde{\phi} \circ \iota.$  One computes immediately that $\Theta_{\m}(\Phi_{\gamma,x}) = - \Phi_{\iota(x),\iota(\gamma)}$.  Thus looking at (\ref{eqn:12sum}), $\Theta_{\iota}$ is $+1$ on $\a(J)$ and $-1$ on $J$.  It follows that $\Theta_{\m}$ is a Lie algebra involution.

From (\ref{eqn:B0J}), one sees that $B_{\Theta_\m}$ is positive definite on $J \subseteq \m(J)$.  From (\ref{eqn:B0wedge}), one obtains
\begin{equation}\label{eqn:posMJ} \frac{1}{2} B_{\Theta_\m}(\Phi_{w \wedge x}, \Phi_{w \wedge x}) = (w,w)(x,x) - (w,x)^2 - (w \times x, w\times x) + (w \times w,x \times x).\end{equation}
If the ground field $F=\R$, the expression (\ref{eqn:posMJ}) is always non-negative, and is $0$ if and only if $\Phi_{w\wedge x} = 0$.  Thus $\Theta_\m$ is a Cartan involution on $\m(J)$.

\subsection{The Lie algebra of $H_J$} In this subsection we recall results about the Lie algebra $\h(J)$ of $H_J$ and its subalgebra $\h(J)^{0}$, which is the Lie algebra of $H_J^{1}$.

\subsubsection{The $J$ decomposition of $\h(J)^0$} The Lie algebra $\h(J)^0$ decomposes as a direct sum 
\begin{equation}\label{eqn:hJgrade}\h(J)^0 = J^\vee \oplus \m(J) \oplus J.\end{equation}
This decomposition is in fact a $\Z$-grading, with $\m(J)$ in degree $0$, and $J$, resp. $J^\vee$ is degree $1$, resp. degree $-1$.  In this paragraph we describe this grading.

First we describe the internal structure of $\h(J)^0$ in terms of the grading (\ref{eqn:hJgrade}).  This is known as the Koecher-Tits construction of this Lie algebra, see \cite{koecherI,koecherII}, \cite{TitsTriality}, \cite{jacobsonBook}. The bracket is given by the action of $\m(J)$ on $J$ and $J^\vee$ from the left, $[J,J] = 0$, $[J^\vee,J^\vee] = 0$, and $[\gamma,x] =  \Phi_{\gamma,x}$.  Here $\gamma \in J^\vee$ and $x \in J$.  One checks easily that this bracket satisfies the Jacobi identity.  

We defined $\h(J)^0$ through its action on $W_J$.  We now describe the map $J^\vee \oplus \m(J) \oplus J \rightarrow \h(J)^0$, i.e., we explain the map $J^\vee \oplus \m(J) \oplus J \rightarrow \mathrm{End}(W_J)$.  One reference for this is \cite[section 9]{jacobsonBook}. To do so, first recall that $W_J = F \oplus J \oplus J^\vee \oplus F$.  For $x \in J$, define $n_{L}(x) \in \h(J)^0$ by
\[ n_L(x) (a,b,c,d) = (0,ax,b \times x,(c,x)).\]
Similarly, for $\gamma \in J^\vee$, define $n_{L}^\vee(\gamma)$ by
\[n_L^\vee(\gamma) (a,b,c,d) = ((b,\gamma),\gamma \times c, d \gamma, 0).\]
The subscripts ``L'' stand for Lie, so as not to confuse the notation with the similarly defined elements of the group $H_J$.  Thus for $x \in J$ and $\gamma \in J^\vee$, $n(x) = \exp(n_{L}(x))$ and $n^\vee(\gamma) = \exp(n_L^\vee(\gamma))$.  We also write $n_{G}(x)$ for $n(x)$, and similarly $n_{G}^\vee(\gamma)$ for $n^\vee(\gamma)$, to emphasize that $n_G(x) =  n(x)$ is in the group $H_J$.

Suppose $\phi \in M_J$, and $t(\phi)$ is the multiplier of $\phi$, i.e., 
\[(\phi(z_1),z_2,z_3) + (z_1,\phi(z_2),z_3) + (z_1,z_2,\phi(z_3)) = t(\phi)(z_1,z_2,z_3).\]
Note that the exponential of $\phi$ satisfies $N(e^{\phi}z) = e^{t(\phi)}N(z)$.  We let $\m(J)$ act on $W_J$ as
\[ M(\phi) (a,b,c,d) = \left(-\frac{t(\phi)}{2} a, -\frac{t(\phi)}{2} b + \phi(b), \frac{t(\phi)}{2} c + \widetilde{\phi}(c), + \frac{t(\phi)}{2} d\right).\]
These actions preserve the symplectic and quartic form.  Hence one obtains a map
\[ J^\vee \oplus M_J \oplus J \rightarrow \h(J)^0\]
by 
\begin{equation}\label{eqn:HJmorph}\gamma + \phi + X \mapsto n_{L}^\vee(\gamma) + M(\phi) + n_{L}(-X).\end{equation}
This map is a Lie algebra homomorphism.

We now record an easy lemma, which we will use below.
\begin{lemma}\label{lem:nconjs} Suppose $x \in J$ and $\gamma \in J^\vee$.  Then
\[n_{G}(x) n_{L}^\vee(\gamma) n_{G}(-x) = n_{L}^\vee(\gamma) + M(\Phi_{\gamma,x}) - n_{L}(U_{x}(\gamma))\]
and
\[n_{G}^\vee(\gamma) n_{L}(x) n_{G}^\vee(-\gamma) = n_{L}(x)-M(\Phi_{\gamma,x}) - n_{L}^\vee(U_{\gamma}(x)).\]
Consequently,
\[n_{G}(-i) n_{L}^\vee(\iota(Z))n_{G}(i) = n_{L}^\vee(Z) - i M(\Phi_{1,Z}) + n_{L}(Z)\]
and
\[n_{G}^\vee(i) n_{L}(Z)n_{G}^\vee(-i) = n_{L}^\vee(Z) - i M(\Phi_{1,Z}) + n_{L}(Z).\]
\end{lemma}
\begin{proof} Computing in the $3$-part $\Z$ grading of $\h(J)^0$, we have
\begin{align*} n_{G}(x) n_{L}^\vee(\gamma) n_{G}(-x) &= \exp(-x) \gamma \exp(x) \\ &=\gamma + [\gamma,x] + \frac{1}{2}[[\gamma,x],x].\end{align*}
The sum stops here because $x$ has degree $1$ so all other terms are in degree $> 1$ and thus $0$.  Now $[\gamma,x] = \Phi_{\gamma,x}$ and then $\frac{1}{2}[[\gamma,x],x] = \frac{1}{2}\Phi_{\gamma,x}(x) = U_{x}(\gamma)$.  The statement with $n_{G}^\vee(\gamma)$ is similar.  The first part of the lemma follows.

The second part of the lemma follows immediately from the first, using that $\Phi_{\iota(Z),1} = \Phi_{\iota(1),Z}$.\end{proof}

\subsubsection{The map $Sym^2(W_J) \rightarrow \h(J)^0$} Because $H_J^{1} \subseteq \Sp(W_J)$, there is an inclusion $\h(J)^{0} \subseteq \sp(W_J) \simeq Sym^2(W_J)$.  Because these Lie algebras are self-dual, one obtains an $H_J^{1}$-equivariant map $Sym^2(W_J) \rightarrow \h(J)^{0}$.  We now describe this map precisely in terms of the Freudenthal triple product; this is due to Freudenthal \cite{freudenthalIV}.  This map $Sym^2(W_J) \rightarrow \h(J)^0$ is used to define the Lie algebra structure on $\g(J)$ \cite{freudenthalIV,yamagutiAsano}, which we will do in the next section.

Recall that the symmetric four-linear form is normalized so that $(v,v,v,v) = 2q(v)$, the trilinear form is normalized via $\langle w, t(x,y,z)\rangle = (w,x,y,z)$, and $v^\flat = t(v,v,v)$.  With these normalizations, one has the following two identities, which are more-or-less equivalent to the fact that the trilinear map $t$ satisfies the identities of a Freudenthal triple system:
\begin{align} \label{eqn:FTS1} 3t(v^\flat, v^\flat, x) + 3q(v) t(v,v,x) &= \langle x, v^\flat\rangle v^\flat + q(v) \langle x,v\rangle v \\ \label{eqn:FTS2} 6t(v,v^\flat, x) &= \langle x, v\rangle v^\flat + \langle x, v^\flat \rangle v.\end{align}

\begin{remark} We remark that one convenient way to remember the identities (\ref{eqn:FTS1}), (\ref{eqn:FTS2}) is as follows.  Suppose $q(v) \neq 0$, and $\omega$ is the image of $x$ in $F[x]/(x^2 - q(v))$.  Then 
\begin{equation}\label{eqn:firstLL} 3t(\omega v + v^\flat, \omega v +v^\flat, x) = \langle x, \omega v + v^\flat \rangle (\omega v + v^\flat).\end{equation}
See \cite[Theorem 5.1.1.]{pollackLL} for some context regarding (\ref{eqn:firstLL}).  Separating the ``real'' and ``imaginary'' parts, one gets the two identities (\ref{eqn:FTS1}), (\ref{eqn:FTS2}).\end{remark}

With these normalizations, for $w, w' \in W_J$ define $\Phi_{w,w'} \in \mathrm{End}(W_J)$ as follows:
\[ \Phi_{w,w'}(x) = 6t(w,w',x) + \langle w',x\rangle w + \langle w,x \rangle w'.\]
We have the following fact.
\begin{proposition} For $w,w' \in W_J$, the endomorphism $\Phi_{w,w'}$ is in $\h(J)^0$, i.e., it preserves the symplectic and quartic form on $W_J$. Furthermore, if $\phi \in \h(J)^0$, then $[\phi,\Phi_{w,w'}] = \Phi_{\phi(w),w'} + \Phi_{w,\phi(w')}$. \end{proposition}
\begin{proof} The fact that $\Phi_{w,w'}$ preserves the symplectic form follows immediately from the definitions.  To check that is preserves the quartic form, one must evaluate $\langle \phi_{w,w'}(v), v^\flat\rangle$.  To do this, one uses (\ref{eqn:FTS2}), and obtains $0$, as desired.  The equivariance statement $[\phi,\Phi_{w,w'}] = \Phi_{\phi(w),w'} + \Phi_{w,\phi(w')}$ is easily checked.\end{proof}

\subsubsection{Formulas in $W_J$} It is convenient to know how the element $\Phi_{w,w'} \in \h(J)^0$ decomposes in the three-term $\Z$-grading (\ref{eqn:hJgrade}).  This is well-known; see, for instance, \cite{yamagutiAsano}.  To ensure all of our normalizations are correct, we briefly give this computation here.  We begin with a few formulas for expressions in $W_J$. 

We begin with formulas for $v^\flat$ and $3t(v,v,x)$ for $v, x \in W_J$.  Set $v = (a,b,c,d)$ and $x = (\alpha,\beta,\gamma,\delta)$. Then one has $v^\flat =(a^\flat, b^\flat, c^\flat,d^\flat)$ with 
\begin{itemize}
\item $a^\flat = -a^2d +a(b,c)-2n(b)$;
\item $b^\flat = - 2 c \times b^\# + 2ac^\# - (ad-(b,c))b$;
\item $c^\flat = 2b \times c^\# - 2db^\#+(ad-(b,c))c$;
\item $d^\flat = ad^2 - d(b,c)+2n(c)$.\end{itemize}

Symmetrizing this formula for $v^\flat$, one finds that $3t(v,v,x) = (a',b',c',d')$, with
\begin{align*} a' &= \alpha( (b,c) - 2ad) + (\beta, ac-2b^\#) + (\gamma, ab) - \delta a^2 \\ b' &= \alpha(2c^\# - bd) + ((b,c)-ad)\beta - 2c \times (b \times \beta) + (\beta,c)b + 2(ac-b^\#) \times \gamma + (b,\gamma)b - \delta ab \\ c' &= \alpha (dc) + 2(c^\#-db) \times \beta + (ad-(b,c))\gamma + 2b \times (c \times \gamma) - (b,\gamma)c  + (ac-2b^\#)\delta \\ d' &= \alpha d^2 + (\beta,-dc) + (2c^\# - db,\gamma) + (2ad - (b,c))\delta.\end{align*}

The element 
\begin{equation}\label{eqn:Phivv} \frac{1}{2}\Phi_{v,v}(x) = 3t(v,v,x) + \langle v,x\rangle v = (a'',b'',c'',d'')\end{equation}
has a nicer expression.  One finds
\begin{align*} a'' &= \alpha ((b,c) - 3ad) + 2(\beta, ac-b^\#) \\ b'' &= 2\alpha (c^\#-db) + \frac{1}{3}((b,c)-3ad) \beta + 2 \Phi'_{\iota(c),b}(\beta) + 2(ac-b^\#) \times \gamma \\ c'' &= 2(c^\#-db) \times \beta + \frac{1}{3}(3ad- (b,c))\gamma - 2 \Phi'_{\iota(b),c}(\gamma) + 2 \delta(ac-b^\#) \\ d'' &= 2(c^\#-db,\gamma) + (3ad - (b,c))\delta.\end{align*}
Here recall that for $x,y,z \in J$,
\[\Phi'_{\iota(x),y}(z) = - x \times (y \times z) + (x,z)y + \frac{1}{3}(x,y) z.\]

We now record how $\Phi_{w,w'}$ looks like in the $\m(J) \oplus (J \oplus J^\vee)$ picture.  Denote by $\phi_s$ the element of $\h(J)^{0}$ that acts on $W_J$ by
\[\phi_s (a,b,c,d) = (3a,b,-c,-3d).\]
This is the differential of the one parameter subgroup $z \mapsto \eta(z)$ in $H_J^{1}$, which acts on $W_J$ via $(a,b,c,d) \mapsto (z^3 a, zb, z^{-1}c,z^{-3}d)$.  Furthermore, write $\mathrm{Id}_{J}$ for the identity map $J \rightarrow J$, considered as element of $\m(J)$.  Then $\mathrm{Id}_{J}$ acts on $J^\vee$ via $\gamma \mapsto -\gamma$, and $t(\mathrm{Id}_{J}) = 3$.
\begin{proposition}\label{prop:NPhi} Suppose $v = (a,b,c,d) \in W_J$.  Then $M(\mathrm{Id}_{J}) = -\frac{1}{2}\phi_{s}$ and 
\begin{equation}\label{eqn:PhivvGood}\frac{1}{4} \Phi_{v,v} = n_{L}(c^\#-db) + n_{L}^\vee(ac-b^\#) + M(\Phi_{c,b} + (ad-(b,c))\mathrm{Id}_{J}).\end{equation}
Consequently, $n_L(X) = -\frac{1}{2}\Phi_{(0,X,0,0),(0,0,0,1)}$, $n_L^\vee(\gamma) = \frac{1}{2}\Phi_{(1,0,0,0),(0,0,\gamma,0)}$, and 
\[\frac{1}{2}\Phi_{(a,b,0,0),(a',b',0,0)} = n_{L}^\vee(-b \times b').\]
.\end{proposition}
\begin{proof} The first two statements follow immediately from the definitions and the formulas (\ref{eqn:Phivv}).  For the last part, first note that $\Phi_{v,v} = 0$ if $v$ is rank one in $W_J$.  Then the formula for $n_L(X)$ comes from applying (\ref{eqn:PhivvGood}) to $(0,X,0,-1)$.  The proof of the formula for $n_L^\vee(\gamma)$ is similar.  The final identity comes from linearizing the equality $\frac{1}{4}\Phi_{(a,b,0,0),(a,b,0,0)} = n_{L}^\vee(-b^\#)$.\end{proof}

\subsubsection{The Killing form} Define an invariant pairing $B_{\h}$ on $\h(J)^{0}$ by
\[B_{\h}((\phi,a,\gamma),(\phi',a',\gamma')) = B_{\m}(\phi,\phi') - (a,\gamma') - (a',\gamma).\]
Here $\phi,\phi' \in \m(J)$, $a, a' \in J$ and $\gamma, \gamma' \in J^\vee$.  

Thinking of $\h(J)^{0}$ as defined via its action on $W_J$, as opposed to its description via the $3$-step $\Z$-grading, it is natural to ask how the pairing $B_{\h}$ looks like on elements of the form $\Phi_{w,w'}$.  One has the following:
\[B_{\h}(\Phi_{w_1,w_1'},\Phi_{w_2,w_2'}) = -2\left(\langle w_1, w_2' \rangle \langle w_1',w_2 \rangle +\langle w_1, w_2 \rangle \langle w_1', w_2'\rangle\right) +12(w_1,w_1',w_2,w_2')\]
for $w_1,w_1',w_2,w_2' \in W_J$.  The pairing $B_{\h}$ satisfies
\[B_{\h}(\Phi_{w,w'},\phi) = 2\langle w, \phi(w')\rangle.\]

\subsubsection{The Cartan involution} Suppose the ground field $F=\R$.  The Cartan involution on $\h(J)^0$ is $\Theta_{\h}((\phi,x,\gamma)) = (\Theta_{\m}(\phi),\iota(\gamma),\iota(x))$.  Here $\Theta_{\m}$ is the Cartan involution on $\m(J)$ from above, induced by $\iota: J \rightarrow J^\vee$, or given on $\Phi$'s by $\Theta_\m(\Phi_{\gamma,x}) =  - \Phi_{\iota(x),\iota(\gamma)}$.  With this definition, $\Theta_{\h}$ is a Lie algebra homomorphism, an involution, and $B_{\Theta_{\h}}$ is positive definite and symmetric.

Thinking of $\h(J)^{0}$ as being defined through its action on $W_J$, there is another way to define a Cartan involution.  Namely, consider the map $J_2$ on $W_J$, given by $J_2(a,b,c,d) = (d,-\iota(c),\iota(b),-a)$.  Define a symmetric pairing on $W_J$ via $(v_1, v_2) := \langle J_2 v_1, v_2 \rangle$.  Since $J_2$ is in $H_J^{1}$, there is an associated involution on $\h_J$ given by $\Theta(\phi) = J_2\phi J_2^{-1}$.  One has $\Theta(\Phi_{w,w'}) = \Phi_{J_2w,J_2w'}$.

The two Cartan involutions we have defined are compatible via the map (\ref{eqn:HJmorph}).  To see this, one computes that $J_2 n_{L}(X) J_2^{-1} = n_{L}^\vee(-\iota(X))$ and similarly $J_2 n_{L}^\vee(\gamma)J_2^{-1} = n_{L}(-\iota(\gamma))$.  Furthermore, $J_2 M(\phi) J_2^{-1} = M(\iota \circ \widetilde{\phi} \circ \iota^{-1}) = M(\Theta(\phi))$.

\section{Lie algebras of exceptional groups, II}\label{sec:LieII}  In this section we define the Lie algebra $\g(J)$.  We first define $\g(J)$ via a $\Z/2$-grading.  Such a construction goes back to Freudenthal \cite{freudenthalIV}, see also \cite{yamagutiAsano}.  As is well-known, the $\Z/2$-grading is essentially equivalent to defining $\g(J)$ via a $5$-step $\Z$-grading. Rumelhart, in \cite{rumelhart}, defined the same Lie algebra in terms of a $\Z/3$-grading.  We have found it useful to use both the $\Z/2$-grading and the $\Z/3$ grading to do computations, although we mostly use the $\Z/2$ grading.  Moreover, we expect that for different problems concerning the group $G_J$, the $\Z/3$-grading will be better suited to computation.  Thus we also recall the description of $\g(J)$ from \cite{rumelhart}, and write down an explicit isomorphism between the $\Z/2$ and $\Z/3$ pictures.

\subsection{The $\Z/2$ grading on $\g(J)$} As mentioned above, this construction goes back to Freudenthal \cite{freudenthalIV}.  See also \cite{yamagutiAsano}. Denote by $V_2$ the defining two-dimensional representation of $\sl_2 = \sp_{2}$.  We define
\[\g(J) = \g(J)_0 \oplus \g(J)_1 := \left(\sl_2 \oplus \h(J)^{0}\right) \oplus \left(V_2 \otimes W_J\right).\]
Here $\g(J)_0 = \sl_2 \oplus \h(J)^0$ is the zeroth graded piece of $\g(J)$, and $\g(J)_1 = V_2 \otimes W_J$ is the first graded piece of $\g(J)$.  We fix the standard symplectic pairing on $V_2$:
\[\langle (a,b)^{t}, (c,d)^{t} \rangle = \left(\begin{array}{cc} a& b\end{array}\right)\left(\begin{array}{cc} & 1\\ -1 & \end{array}\right)\left(\begin{array}{c} c \\ d \end{array}\right) = ad - bc.\]

\subsubsection{The bracket} Fix $\alpha \in F^\times$, $F$ being the ground field (which is assumed to be characteristic $0$.)  We define a map $[\;,\;]_{\alpha}:\g(J) \otimes \g(J) \rightarrow \g(J)$ as follows: If $\phi, \phi' \in \g(J)_0 = \sl_2 \oplus \h(J)^0$, $v, v' \in V_2$, and $w, w' \in W_J$, then 
\[ [(\phi,v\otimes w),(\phi',v'\otimes w')]_{\alpha} = \left([\phi,\phi'] + \alpha \langle w,w'\rangle (v \cdot v') + \alpha \langle v,v'\rangle \Phi_{w,w'},\phi(v' \otimes w') - \phi'(v \otimes w)\right).\]
With this definition, we have the following fact.
\begin{proposition} The bracket $[\;,\;,]_{\alpha}$ on $\g(J)$ satisfies the Jacobi identity.\end{proposition} 
\begin{proof} To check the Jacobi identity $\sum_{cyc}{[X,[Y,Z]]} = 0$, by linearity it suffices to check it on the various $\Z/2$-graded pieces.  Then there are four types identities that must be checked.  Namely, if $0,1,2$ or $3$ of the elements $X,Y,Z$ are in $\g(J)_1=V_2 \otimes W_J$.  If all three of $X,Y,Z$ are in $\g(J)_0 = \sl_2 \oplus \h(J)^0$, then the Jacobi identity is of course satisfied.  If two of $X,Y,Z$ are in $\g(J)_0$, then the Jacobi identity is satisfied.  This fact is equivalent to the fact that the bracket $[\;,\;]_{\alpha}$ defines a Lie algebra action of $\g(J)_0$ on $\g(J)_1$: $[\phi,\phi']( x) = \phi(\phi'(x)) -\phi'(\phi(x))$ for $x \in \g(J)_1$ and $\phi,\phi' \in \g(J)_0$.  If one of $X,Y,Z$ is in $\g(J)_0$, then the Jacobi identity is satisfied by the equivariance of the map $\g(J)_1 \otimes \g(J)_1 \rightarrow \g(J)_0$.  Finally, when $X,Y,Z$ are all in $\g(J)_1$, a simple direct computation shows that $\sum_{cyc}{[X,[Y,Z]]} = 0$.  (The term $t(w,w',w'')$ drops out right away because it is symmetric by applying the identity $\sum_{cyc}{\langle v_2, v_3 \rangle v_1 } = 0$ for $v_1, v_2, v_3 \in V_2$, and one must only inspect the more basic terms.) \end{proof}

\begin{remark} The reader can easily verify that the Lie algebras $\g(J)$, with bracket given by $\alpha \in F^\times$, are all isomorphic for varying $\alpha$.  Below we will compute with the choice of $\alpha = \frac{1}{2}$, but for now we keep things as general as possible. \end{remark}

\subsubsection{The Killing form} Define a symmetric pairing $B_{\g}: \g(J) \otimes \g(J) \rightarrow F$ via
\[B_{\g}(\phi + v\otimes w, \phi' + v'\otimes w') = B_0(\phi,\phi') - 2\alpha \langle v,v' \rangle \langle w,w' \rangle.\]
Here $B_0$ is the invariant symmetric pairing on $\g(J)_0$ defined by 
\[B_0(\phi_2 + \phi_J,\phi_2'+ \phi_J') = B_{\sl(V_2)}(\phi_2,\phi_2') + B_{\h}(\phi_J,\phi_J')\]
for $\phi_2, \phi_2'\in \sl_2$ and $\phi_J, \phi_J' \in \h(J)^0$.

\begin{lemma} The pairing $B_{\g}$ on $\g(J)$ is invariant, i.e. $B_{\g}([x_1,x_3],x_2) = - B_{\g}([x_2,x_3],x_1)$ for all $x_1, x_2, x_3$ in $\g(J)$. \end{lemma}
\begin{proof} Writing out both $B_{\g}([x_1,x_3],x_2)$ and $- B_{\g}([x_2,x_3],x_1)$, one finds that the form $B_{\g}$ is invariant if and only if
\[ 2 \langle w,\phi(w') \rangle =  B_0(\Phi_{w,w'},\phi)\]
and
\[ 2 \langle v,\phi(v') \rangle =  B_0(v \cdot v',\phi)\]
for all $v, v' \in V_2$, $w, w' \in W_J$, and $\phi \in \g(J)_0$. These properties of $B_0$ were discussed above.\end{proof}

\subsubsection{The Cartan involution} Recall the element $J_2 \in H_J^1$, which acts on $W_J$ via $J_2 (a,b,c,d) = (d,-c,b,-a)$.  We abuse notation and also write $J_2 = \mm{}{1}{-1}{} \in \SL_2$.  (There is a natural map $\SL_2 \rightarrow H_J^1$, and the image of $J_2 \in \SL_2$ is the $J_2 \in H_J^1$.)

Using $J_2$, we define an involution $\Theta_{\g}$ on $\g(J)$ as 
\[\Theta_{\g}(\phi_2 +\phi_J, v \otimes w) = (J_2 \phi_2 J_2^{-1} + J_2 \phi_J J_2^{-1}, J_2 v \otimes J_2 w).\]
Here $\phi_2 \in \sl_2$, $\phi_J \in \h(J)^0$, $v \in V_2$ and $w \in W_J$.  It is clear that $\Theta_{\g}$ is an involution on $\g(J)$.  

If the ground field $F=\R$ and $\alpha > 0$, then $\Theta_{\g}$ defines a Cartan involution on $\g(J)$.  Indeed,
\[B_{\Theta_\g}(\phi + v \otimes w, \phi'+ v'\otimes w') = B_{\Theta_\g}(\phi,\phi') + 2 \alpha \langle v, J_2 v'\rangle \langle w,J_2w'\rangle.\]
But $B_{\Theta_\g}$ restricted to $\g(J)_0$ is symmetric positive definite, as this was discussed above, and 
\[\langle v, J_2 v'\rangle \langle w,J_2w'\rangle = \langle J_2 v, v'\rangle \langle J_2 w, w'\rangle = (v,v')(w,w').\]
That the pairings $(v,v')$ and $(w,w')$ on $V_2$ and $W_J$ are symmetric positive definite was also discussed above.

\subsection{The $\Z/3$ grading on $\g(J)$} In this subsection we recall elements from the paper \cite{rumelhart} (in different notation).  Rumelhart constructed the Lie algebra $\g(J)$ through a $\Z/3$-grading, as opposed to a $\Z/2$-grading as we have described above.  We also write down an explicit isomorpism from the $\Z/3$-model to the $\Z/2$-model.

Denote by $V_3$ the defining representation of $\sl_3$, and by $V_3^\vee$ the dual representation.  In the $\Z/3$-graded picture, one defines
\[\g(J) = \sl_3 \oplus \m(J)^{0} \oplus V_3 \otimes J \oplus V_3^\vee \oplus J^\vee.\]
We consider $V_3, V_3^\vee$ as left modules for $\sl_3$, and $J, J^\vee$ as left modules for $\m(J)^0$.

\subsubsection{The bracket} Following \cite{rumelhart}, the Lie bracket is given as follows.  First, because $V_3$ is considered as a representation of $\sl_3$, there is an identification $\wedge^2 V_3 \simeq V_3^\vee$, and similarly $\wedge^3 V_3^\vee \simeq V_3$.  If $v_1, v_2, v_3$ denotes the standard basis of $V_3$, and $\delta_1, \delta_2, \delta_3$ the dual basis of $V_3^\vee$, then $v_1 \wedge v_2 = \delta_3$, $\delta_1 \wedge \delta_2 =v_3$, and cyclic permutations of these two identifications.

Take $\phi_3 \in \sl_3$, $\phi_J \in \m(J)^0$, $v, v' \in V_3$, $\delta, \delta' \in V_3^\vee$, $X, X' \in J$ and $\gamma, \gamma' \in J^\vee$.  Then
\begin{align*} [\phi_3, v \otimes X + \delta \otimes \gamma] &= \phi_3(v) \otimes X + \phi_3(\delta) \otimes \gamma. \\ [\phi_J, v \otimes X+ \delta \otimes \gamma] &= v \otimes \phi_J(X) + \delta \otimes \phi_J(\gamma) \\ [v \otimes X,v' \otimes X'] &= (v \wedge v') \otimes (X \times X') \\ [\delta \otimes \gamma, \delta' \otimes \gamma'] &= (\delta \wedge \delta') \otimes (\gamma \times \gamma') \\ [\delta \otimes \gamma, v \otimes X] &= (X,\gamma) v \otimes \delta + \delta(v) \Phi_{\gamma,X} - \delta(v)(X,\gamma) \\ &= (X,\gamma) \left(v \otimes \delta -\frac{1}{3}\delta(v)\right) + \delta(v) \left(\Phi_{\gamma,X} - \frac{2}{3}(X,\gamma)\right).\end{align*}
Note that $v \otimes \delta -\frac{1}{3}\delta(v) \in \sl_3$ and $\Phi_{\gamma,X} - \frac{2}{3}(X,\gamma) = \Phi'_{\gamma,X} \in \m(J)^0$.  Also recall that $\Phi_{\gamma, X} \in \m(J)$ acts on $J$ via
\[\Phi_{\gamma,X}(Z) = - \gamma \times (X \times Z) + (\gamma,Z)X + (\gamma,X)Z.\]
Furthermore, the action of $\sl_3$ and $\m(J)^0$ on $V_3^\vee$ and $J^\vee$ is determined by the equalities $(\phi_3(v),\delta) + (v,\phi_3(\delta))=0$ and $(\phi_J(X),\gamma) + (X,\phi_J(\gamma)) = 0$.

\subsubsection{The Killing form} There is an invariant symmetric form $B_{\g}$ on $\g(J)$ that restricts to the form on $\sl_3$ given by $B_{\g}(m_1, m_2) = \tr(m_1 m_2)$ for $m_1, m_2 \in \mathrm{End}(V_3)$.  This form $B_{\g}$ on $\g(J)$ is given as follows:
\begin{itemize}
\item On $\sl_3$: $B_{\g}(v \otimes \phi, v' \otimes \phi') = \phi(v')\phi'(v)$
\item On $\m(J)^0$: $B_{\g}(\phi_1,\phi_2) = B_{\m}(\phi_1,\phi_2)$.
\item On $V_3 \otimes J \oplus V_3^\vee \otimes J^\vee$: $B_{\g}(v \otimes X,\delta' \otimes \gamma') = -\delta'(v)(X,\gamma')$.
\end{itemize}
When $F=\R$ this is a positive multiple of the Killing form.

\subsubsection{The Cartan involution} We endow $V_3$ with the positive definite symmetric form given by $(v,v') = \,^tv v'$.  In other words, we make the standard basis $v_1, v_2, v_3$ of $V_3$ orthonormal.  This induces an identification $\iota$ between $V_3$ and $V_3^\vee$.  

Define an involution $\Theta_{\g}$ on $\g(J)$ as follows: On $\sl_3$ it is $X \mapsto -X^{t}$.  On $\m(J)^0$ it is $\Theta_{m}$.  On $V_3 \otimes J$ it is $v \otimes X \mapsto \iota(v) \otimes \iota(X) \in V_3^\vee \otimes J^\vee$ and on $V_3^\vee \otimes J^\vee$ it is $\delta \otimes \gamma \mapsto \iota(\delta) \otimes \iota(\gamma) \in V_3 \otimes J$. The map $\Theta_{\g}$ is a Cartan involution.

\subsubsection{Comparing the $\Z/3$ and $\Z/2$ gradings}  In this paragraph, we compare the $\Z/3$ and $\Z/2$ gradings.  More precisely, fix $\alpha = \frac{1}{2}$.  Then we define a map from the $\Z/3$-picture to the $\Z/2$-picture, and check that it is a Lie algebra isomorphism.  This isomorphism carries the pairing and involution $B_{\g}$ and $\Theta_{\g}$ from the $\Z/3$-picture over to the ones of the same name in the $\Z/2$ picture.  (Which is why we abuse notation and denote by the same letters.)

Denote by $E_{ij}$ the element of $M_3(F)$ with a $1$ in the $(i,j)$ position and zeros elsewhere. Denote by $v_1, v_2, v_3$ the defining standard basis of $V_3$ and $\delta_1, \delta_2, \delta_3$ the dual basis in $V_3^\vee$.  Then $E_{ij} = v_i \otimes \delta_j$. Also, denote by $e, f$ the standard symplectic basis of $V_2$, so that under the isomorphism $\sl_2 \simeq \mathrm{Sym}^2(V_2)$, $e^2/2 = \mm{0}{1}{0}{0}$, $\mm{0}{0}{1}{0} = -f^2/2$ and $\mm{1}{}{}{-1} = -ef$.  Define a map from the $\Z/3$ picture to the $\Z/2$ picture as follows:
\begin{enumerate}
\item $E_{13} \mapsto \mm{0}{1}{0}{0} = e^2/2$
\item $E_{31} \mapsto \mm{0}{0}{1}{0} = -f^2/2$
\item $E_{11}-E_{33} \mapsto \mm{1}{}{}{-1} = -ef$.
\item $aE_{12} + v_1 \otimes b + \delta_3 \otimes c + d E_{23} \mapsto e \otimes (a,b,c,d)$
\item $a'E_{32} + v_3 \otimes b' + (-\delta_1) \otimes c' + d'(-E_{21}) \mapsto f \otimes (a',b',c',d')$.
\item $\delta_2 \otimes \gamma + \phi + v_2 \otimes X \mapsto n_{L}^\vee(\gamma) + M(\phi) + n_{L}(-X)$.  Here $\phi \in \m(J)^0$.
\end{enumerate}

\begin{proposition}\label{prop:Z3Z2} The above maps define a Lie algebra isomorphism from the $\Z/3$-picture of $\g(J)$ to the $\Z/2$-picture of $\g(J)$ with $\alpha = \frac{1}{2}$.\end{proposition}
\begin{proof} Write $(\alpha,\beta,\gamma,\delta)$ as shorthand for $\alpha E_{12} + v_1 \otimes \beta + \delta_3 \otimes \gamma + \delta E_{23}$.  One computes that
\[[v_2 \otimes X, (\alpha, \beta,\gamma,\delta)] = - (0, \alpha X, \beta \times X, (\gamma,X))\]
and
\[[\delta_2 \otimes Y,(\alpha,\beta,\gamma,\delta)] = ((\beta,Y),\gamma \times Y, \delta Y, 0).\]
Furthermore, with $\phi_s = E_{11} - 2E_{22} + E_{33}$, one finds
\[[\phi_s,(\alpha,\beta,\gamma,\delta)]= (3\alpha, \beta, -\gamma, - 3\delta).\]
By similarly computing the action of $v_2 \otimes X, \delta \otimes Y$ and $\phi_s$ on $a'E_{32} + v_3 \otimes b' + (-\delta_1) \otimes c' + d'(-E_{21})$, one finds that the image of $\delta_2 \otimes \gamma + \phi + v_2 \otimes X$ acts as it is supposed to on $W_J$.

The most nontrivial part of the proof is to compare the bracket $[e\otimes v, f \otimes v']$ on the two sides.  On the $\Z/2$-side, one gets $-\frac{\langle v,v'\rangle}{2}(-ef) + \frac{1}{2}\Phi_{v,v'}$.  Set
\[V = aE_{12} + v_1 \otimes b + \delta_3 \otimes c + d E_{23}\]
and 
\[V' = a'E_{32} + v_3 \otimes b' + (-\delta_1) \otimes c' + d'(-E_{21}).\]
Then on the $\Z/3$-side, one computes that
\begin{align*} [V,V'] &= -\frac{\langle v, v' \rangle}{2}(E_{11}-E_{33}) + \frac{1}{2}\left( -(ad'+da') + (b,c')/3 + (b',c)/3)\right) \phi_s \\ &\;\; + v_2 \otimes (db' + d'b - c\times c') + \delta_2 \otimes (a'c+ac' - b \times b') + \Phi_{\iota(c),b'}'+ \Phi'_{\iota(c'),b}.\end{align*}

Here $\phi_{s} = E_{11}-2E_{22}+E_{33}$.  Thus, symmetrizing, one must check that $\frac{1}{4}\Phi_{v,v}$ corresponds to 
\[v_2 \otimes (db-c^\#) + \delta_2 \otimes (ac-b^\#) + \Phi'_{\iota(c),b} + \frac{1}{2} (-ad + (b,c)/3) \phi_{s}.\]
But this follows from Proposition \ref{prop:NPhi}.\end{proof}

\subsection{The group $G_J$} We assume in this subsection that the Lie algebra $\g(J)$ is defined over a ground field $F$ of characteristic $0$. We define the group $G_J$ to be the connected component of the identity of the automorhpism group of $\g(J)$.  This is, of course, a connected reductive adjoint group.  For notation, we write $e_0, h_0, f_0$ for the usual $\sl_2$-triple inside $\sl_2 \subseteq \g(J)_0$, so that $e_0 = \mm{0}{1}{0}{0}$, $h_0 = \mm{1}{}{}{-1}$, and $f_0 = \mm{0}{0}{1}{0}$.

\subsubsection{The $5$-grading} We now define the $5$-grading on $\g(J)$.  Namely, the components of $\g(J)$ in each graded piece are
\begin{itemize}
\item In degree $-2$: spanned by $f_0$
\item In degree $-1$: $f \otimes W_J$
\item In degree $0$: $F h_0 \oplus \h(J)^0$
\item In degree $1$: $e \otimes W_J$
\item In degree $2$: spanned by $e_0$.\end{itemize}

Note that $F h_0 \oplus \h(J)^0$ is the Lie algebra $\h(J)$, via the map $\alpha h_0 + \phi_0 \mapsto \alpha \mathrm{Id}_{W_J} + \phi_0$, where $\phi_0 \in \h(J)^0$ and $\mathrm{Id}_{W_J}$ denotes the identity on $W_J$.  Recall that $\h(J)$ is the subalgebra of endomorphisms $(\phi,\mu(\phi))$ of $End(W_J) \oplus \mathrm{G}_a$ satisfying $\langle \phi(v_1),v_2\rangle + \langle v_1, \phi(v_2)\rangle = \mu(\phi) \langle v_1, v_2 \rangle$ and $\sum_{cyc}{(\phi(v_1),v_2,v_3,v_4)}  = 2\mu(\phi) (v_1,v_2,v_3,v_3)$ for all $v_1, v_2, v_3, v_4 \in W_J$.  The map $\h(J) \rightarrow F h_0 \oplus \h(J)^0$ in the opposite direction is given by $\phi \mapsto \frac{\mu(\phi)}{2} H + (\phi - \frac{\mu(\phi)}{2} \mathrm{Id}_{W_J})$.

\subsubsection{The Heisenberg parabolic} We now define the Heisenberg parabolic of $G_J$.   Define $P\subseteq G$ to be the $g \in G_J$ stabilizing the line $F e_0$ generated by $e_0$.  Then $P$ is clearly parabolic.  For instance, the variety $G/P$ is the subset of $\mathbf{P}(\g(J))$ consisting of those $X$ with $[X,[X,y]] + 2B_{\g}(X,y)X = 0$ for all $y \in\g(J)$.  Thus $G/P$ is cut out by closed conditions, so is projective.

Equivalently, define $\p_{Heis} \subseteq \g(J)$ to consist of the elements $X$ so that $[X,e_0] \in F e_0$.  Then the Heisenberg parabolic is equivalently defined to be the $g\in G_J$ satisfying $g \p_{Heis} = \p_{Heis}$.  (If $g \in G_J$ stabilizes the line generated by $e_0$, then $g$ stabilizes $\p_{Heis}$, since $\p_{Heis}$ is defined in terms of $e_0$.  Conversely, if $g \in G_J$ stabilizes $\p_{Heis}$, then $g$ stabilizes $F e_0$, because then $g$ acts as an automorphism of $\p_{Heis}$, and $F e_0$ is the center of the unipotent radical of $\p_{Heis}$.) Furthermore, $\p_{Heis}$ consists exactly of the element of $\g(J)$ non-negative degree in the $5$-grading.

Associated to the element $f_0$ is an increasing filtration, whose associated graded is the $5$-grading.  Namely, 
\begin{itemize}
\item $\mathcal{F}_{-2}\g(J) = F f_0$;
\item $\mathcal{F}_{-1}\g(J) = \{X \in \g(J): [[X,Y],f_0] \in F f_0 \text{ for all } Y \in (F f_0)^\perp\}$.
\item $\mathcal{F}_{0}\g(J) = \{X \in \g(J): [X, f_0] \in F f_0\}$;
\item $\mathcal{F}_{1}\g(J) = (F f_0)^\perp = \{X \in \g(J): B(X,f_0) = 0\}$;
\item $\mathcal{F}_{2}\g(J) = \g(J)$.
\end{itemize}

The Lie algebra of $P$ is $\p_{Heis}$.  Define a Levi subgroup $M$ of $P$ to be the subgroup of $P$ that preserves the $5$-grading.  Equivalently, $M$ is the subgroup of $P$ that also fixes the line spanned by $f_0$. (This follows from the filtration remark above.)  The Levi subgroup $M$ is exactly the group $H_J$, as we now prove.

\begin{lemma} The map $M \rightarrow \GL_1 \times \GL(W_J)$ defined by the conjugation action of $M$ on the degree $2$ and degree $1$ pieces of the $5$-grading defines an isomorphism $M \simeq H_J$.  The $\GL_1$-projection is the similitude.\end{lemma}
\begin{proof} Define $\nu: M \rightarrow \GL_1$ to be the action of $M$ on $e_0$, i.e., $m e_0 = \nu(m) e_0$.  First we check that the map defined in the statement of the lemma lands in $H_J$.  To see this, note that $[e \otimes v, e\otimes v'] = \langle v,v' \rangle \frac{e^2}{2} = \langle v,v' \rangle e_0$.  Thus, if $m \in M$, then 
\[\nu(m) \langle v,v' \rangle e_0 = \langle v,v' \rangle m e_0 =m [e \otimes v, e \otimes v'] = [e \otimes mv, e \otimes mv'] = \langle mv,mv'\rangle e_0.\]
Thus $m$ preserves the symplectic form on $W_J$.  To see that $M$ preserves the quartic form on $W_J$, note that $B(\exp( e \otimes v) f_0,f_0) = C_0 q(v)$, for some constant $C_0 \neq 0$ that depends on our normalizations.  Thus,
\begin{align*} C_0 q(mv) &= B(\exp(e \otimes mv) f_0, f_0) = B(m \exp(e \otimes v) m^{-1} f_0, f_0) = B(\exp( e\otimes v) m^{-1}f_0, m^{-1}f_0) \\ &= C_0 \nu(m)^2 q(v).\end{align*}
Therefore, the map $M \rightarrow \GL_1 \times \GL(W_J)$ lands in $H_J$.

We can also define a map $H_J \rightarrow M \subseteq G_J$ as follows.  If $g \in H_J$, then let $g$ act on $\g(J)$ via
\[g \cdot \left( \mm{a}{b}{c}{-a} + \phi_0 + e\otimes v + f \otimes v'\right) = \mm{a}{\nu(g)b}{\nu(g)^{-1}c}{-a} + g \phi_0 g^{-1} + e\otimes gv + \nu(g)^{-1} f \otimes gv'.\]
Via the (immediately checked) identity $\nu(g)^{-1} \Phi_{gv,gw} = g \Phi_{v,w} g^{-1}$, one verifies that the above action of $H_J$ on $\g(J)$ preserves the bracket, and thus defines a map $H_J \rightarrow \mathrm{Aut}(\g(J))$.  Since $H_J$ is connected, the image lies in the identity component $G_J$ of $\mathrm{Aut}(\g(J))$.  It is clear that this action preserves the $5$-grading, and thus we obtain a map $H_J \rightarrow M$.

Finally, it is clear that the composition $H_J \rightarrow M \rightarrow H_J$ is the identity.  Because all the groups are connected and have the same Lie algebra, the lemma follows. \end{proof}

\section{Exceptional Cayley transform}\label{sec:Cayley} From now on the ground field $F=\R$, unless explicitly indicated to the contrary.  In section \ref{sec:LieII} we have defined the Lie algebra $\g(J)$ and the Cartan involution $\Theta_{\g}$ on $\g(J)$.  Set $\k_0 = \g(J)^{\Theta_{\g} = 1}$, $\p_0 = \g(J)^{\Theta_\g = -1}$, $\k = \k_0 \otimes \C$ and $\p = \p_0 \otimes \C$. In this section we give an explicit Cayley transform on $\g(J)$.  That is, we write down an explicit automorphism $\mathcal{C}$ of $\g(J) \otimes \C$ that takes $\g(J)_0 \otimes \C$ to $\k$ and $\g(J)_{1} \otimes \C$ to $\p$.

Define 
\[C_{2} = \frac{1}{\sqrt{2}}\left(\begin{array}{cc} i & -1 \\ 1 &-i\end{array}\right) \in \SL_2(\C).\]
Thus
\[C_{2}^{-1} = \frac{1}{\sqrt{2}}\left(\begin{array}{cc} -i & 1 \\ -1 &i\end{array}\right).\]
Also define
\[C_{h} = n_{G}\left(-i\right) n_{G}^\vee\left(-\frac{i}{2}\right)\eta(2^{-1/2})\]
so that
\[C_{h}^{-1} = \eta(\sqrt{2}) n_{G}^\vee\left(\frac{i}{2}\right)n_{G}\left(i\right) \in H_J^{1}(\C).\]
Here $\eta(\lambda)$ means the element of $H_J^{1}$ that acts on $W_J$ by $(a,b,c,d) \mapsto (\lambda^{3}a,\lambda b, \lambda^{-1} c, \lambda^{-3} d)$.  The elements $C_2 \in \SL_2(\C)$ and $C_h \in H_J^{1}(\C)$ are, in fact, Cayley transforms for their respective groups.

Now, define $w_{23} = \left(\begin{array}{ccc} -1 & & \\ &&-1 \\ &-1& \end{array}\right) \in \SL_3$. Let $w_{23}$ act on $\g(J)$ via its adjoint action on the $\Z/3$-model, and the isomorphism between the $\Z/3$ and $\Z/2$-models specified in Proposition \ref{prop:Z3Z2}.  Define $\mathcal{C} = \left(C_2 \boxtimes C_h\right)w_{23}$ to be the composite of $w_{23}$ and $C_2 \boxtimes C_h$.  Here $C_2 \boxtimes C_h = C_2 C_h = C_h C_2$ is considered as an element of $\mathrm{Aut}(\g(J) \otimes \C)$.  This is our explicit Cayley transform.  We will make a good choice of bases of $\k$ and $\p$, and then compute $\mathcal{C}^{-1}$ on these bases.

\subsection{Good bases of $\k$ and $\p$} We begin by describing a basis of $\k$.  Recall that 
\[r_0(Z) = (1,-Z,Z^\#,-n(Z)) \in W_J \otimes \C.\]
Set
\begin{itemize}
\item $e_{\ell} = \frac{1}{4}(i e + f) \otimes r_0(i)$
\item $f_{\ell} = \frac{1}{4}(ie-f) \otimes r_0(-i)$
\item $h_{\ell} = \frac{i}{2}\left( \mm{0}{1}{-1}{0} + n_L(-1) + n_L^\vee(\iota(1))\right)$.\end{itemize}
Then since $J_2 r_0(i) = i r_0(i)$ and $J_2 (ie+f) = -i(ie+f)$, $\Theta(e_{\ell}) = e_{\ell}$ and thus $e_{\ell} \in \k$.  Similarly, $f_{\ell}$ and $h_{\ell}$ are in $\k$.  In fact, these three elements form an $\sl_2$-triple.  As we will check below, the elements $e_{\ell}, h_{\ell}, f_{\ell}$ span the long root $\mathfrak{su}_2$ in $\k$ (hence the subscript $\ell$.)

We will now define a nice basis of what will be the Lie algebra $\mathfrak{l}_0(J)$ of $L_0(J)$.  This Lie algebra contains $\a(J)$.  Recall that $\a(J)$ is contained in $\h(J)^0 \subseteq \g(J)_{0}$.  Also observe that $\a(J)$ commutes with $e_{\ell}, f_{\ell}, h_{\ell}$ because the $\sl_2$-triple does not involve elements of $J$ other than $1_J$.  

We will momentarily define other elements of $\mathfrak{l}_0(J)$.  First, for $X \in J$, define
\[V(X) = \frac{1}{2}(\tr(X),-i\tr(X)+2iX,\tr(X)-2X,-i\tr(X))\]
and
\[V(X)^* = \frac{1}{2}(\tr(X),i\tr(X)-2iX,\tr(X)-2X,i\tr(X)),\]
elements of $W_J(\C)$.  Note that $J_2V(X) = -iV(X)$ and $J_2 V(X)^* = i V(X)^*$. Furthermore,
\begin{equation}\label{eqn:VX} V(X) = \frac{\tr(X)}{2}r_0(i) + i n_{G}(-i)(0,X,0,0) = n_{G}(-i)\left(\frac{\tr(X)}{2},iX,0,0\right)\end{equation}
and taking conjugates $V(X)^* = n_{G}(i)\left(\frac{\tr(X)}{2},-iX,0,0\right)$.

Define
\begin{itemize}
\item $n_{E}(X) = \frac{i}{2}\left(e \otimes V(X)^* + J_2e\otimes J_2V(X)^*\right)= \frac{1}{2}(ie+f) \otimes V(X)^*$
\item $n_{H}(X) = \frac{i}{2}\left(\tr(X) \mm{0}{1}{-1}{0} + n(\tr(X)-2X) + n^\vee(\iota(2X-\tr(X)))\right)$
\item $n_{F}(X) = \frac{i}{2}\left(e \otimes V(X) + J_2e \otimes J_2V(X)\right) = \frac{1}{2}(ie-f) \otimes V(X)$.\end{itemize}
It is clear that these elements are in $\k$, from the definition of the Cartan involution $\Theta_{\g}$.  These elements commute with the long root $\mathfrak{su}_2$, as will be clear below.

We now define a good basis of $\p$.  Set
\begin{itemize}
\item $h_3 = \frac{1}{2} \mm{-1}{i}{i}{1}$.
\item $h_1(X) = \frac{1}{2}(ie+f) \otimes V(X)$
\item $h_{-1}(Z) = \frac{i}{2}(n(Z) + n^\vee(\iota(Z))) + \frac{1}{2}M(\{Z,\bullet\}) = \frac{i}{2}(n(Z) + n^\vee(\iota(Z))) + \frac{1}{2}M(\Phi_{1,Z})$.
\item $h_{-3} = \frac{1}{4}(ie-f) \otimes r_0(i)$.\end{itemize}
By Lemma \ref{lem:nconjs},
\[h_{-1}(Z) = \frac{i}{2} n_{G}(-i)n_{L}^\vee(Z) n_{G}(i).\]

\subsection{Action of $\mathcal{C}^{-1}$} The following theorem explains the sense in which $\mathcal{C}$ is a Cayley transform.
\begin{theorem}\label{thm:Cayley} One has the following identities:
\begin{enumerate}
\item $\mathcal{C}^{-1}(h_3) = -i e \otimes (1,0,0,0)$
\item $\mathcal{C}^{-1}(h_1(X)) = -i e \otimes (0,X,0,0)$
\item $\mathcal{C}^{-1}(h_{-1}(Z)) = -ie \otimes (0,0,Z,0)$
\item $\mathcal{C}^{-1}(h_{-3}) = -i e \otimes (0,0,0,1)$
\item $\mathcal{C}^{-1}(n_{E}(X)) = n_{L}^\vee(\iota(X))$
\item $\mathcal{C}^{-1}(n_{F}(Y)) = n_{L}(Y)$
\item $\mathcal{C}^{-1}(e_{\ell}) = \mm{0}{1}{0}{0}$
\item $\mathcal{C}^{-1}(f_{\ell}) = \mm{0}{0}{1}{0}$.
\item $\mathcal{C}^{-1}(\overline{h_3}) = -i f \otimes (0,0,0,1)$
\item $\mathcal{C}^{-1}(\overline{h_{1}(X)}) = i f \otimes (0,0,X,0)$
\item $\mathcal{C}^{-1}(\overline{h_{-1}(Z)}) = -i f \otimes (0,Z,0,0)$
\item $\mathcal{C}^{-1}(\overline{h_{-3}}) = i f \otimes (1,0,0,0)$.\end{enumerate}
\end{theorem}
\begin{proof} The proof of the theorem is a rather direct verification.  In doing this verification, one uses the following relations:
\begin{enumerate}
\item $n_{G}^\vee(i/2)n_{G}(i)V(X) = (0,iX,0,0)$
\item $n_{G}^\vee(i/2)n_{G}(i)V(X)^* = (0,0,-2X,0)$
\item $n_{G}^\vee(i/2)n_{G}(i)r_0(i) = (1,0,0,0)$
\item $n_{G}^\vee(i/2)n_{G}(i)r_0(-i) =-8i(0,0,0,1)$
\item $C_2^{-1}(ie+f) = \sqrt{2}e$
\item $C_2^{-1}(ie-f) = -\sqrt{2} i f$
\item $\frac{1}{2} C_2^{-1}\mm{-1}{i}{i}{1} C_2 = -i \mm{0}{1}{0}{0}$.
\item $\frac{1}{2} C_2^{-1}\mm{-1}{-i}{-i}{1} C_2 = i \mm{0}{0}{1}{0}$
\item $J_2 n_{L}(x) J_2^{-1} = n_{L}^\vee(-\iota(x))$
\item $J_2 n_{L}^\vee(\gamma) J_2^{-1} = n_{L}(-\iota(\gamma))$.
\item $h_{-1}(Z) = \frac{i}{2}\left(n_{G}^\vee(i) n_{L}(Z) n_{G}^\vee(-i)\right)$.
\item $\left(n_{G}^\vee\left(\frac{i}{2}\right) n_{G}\left(i\right)\right) \cdot \overline{h_{-1}}(Z) = -2i n_{L}(Z)$.
\item $\eta(\lambda) n_{L}(Z) \eta(\lambda)^{-1} = n_{L}(\lambda^{-2} Z)$.
\end{enumerate}

\end{proof}

\section{The Cartan and Iwasawa decompositions of $\g(J)$}\label{sec:Cartan} In this section, we give a few (more) facts about the Cartan and Iwasawa decomposition for $\g(J)$.

\subsection{Good basis of $\k$}\label{subsec:kbasis} Recall the elements $e_{\ell}, f_{\ell}, h_{\ell}$ from above.  We have the following lemma.
\begin{lemma} The elements $e_{\ell}, h_{\ell}, f_{\ell}$ form an $\sl_2$-triple.  That is, one has $[e_{\ell},f_{\ell}] = h_{\ell}$, $[h_{\ell},e_{\ell}] = 2 e_{\ell}$ and $[h_{\ell},f_{\ell}] = -2 f_{\ell}$.\end{lemma}
\begin{proof} To see that $[e_{\ell},f_{\ell}] = h_{\ell}$, one must compute $\Phi_{r_0(i),r_0(-i)}$.  First note that $\frac{1}{2}(r_0(i) + r_0(-i)) = (1,0,-1,0)$.  Thus
\[\frac{1}{2}\Phi_{(1,0,-1,0),(1,0,-1,0)} = \Phi_{r_0(i),r_0(-i)} + \frac{1}{2}\Phi_{r_0(i),r_0(i)} + \frac{1}{2}\Phi_{r_0(-i),r_0(-i)} = \Phi_{r_0(i),r_0(-i)}.\]
The last equality is because $r_0(Z)$ is rank one, and thus $\Phi_{r_0(Z),r_0(Z)} = 0$.  Thus applying Proposition \ref{prop:NPhi}, $\Phi_{r_0(i),r_0(-i)} = 2n_L(1) + 2n^\vee_L(-1)$.  The commutator $[e_{\ell},f_{\ell}]$ is now quickly verified to be $h_{\ell}$.

That $[h_{\ell},e_{\ell}] = 2 e_{\ell}$ and $[h_{\ell},f_{\ell}] = -2f_{\ell}$ now follows from the Cayley transform, Theorem \ref{thm:Cayley}.  Alternatively, one can check these identities directly: A short computation shows that $(n_L(-1) + n_L(\iota(1))) r_0(i) = -3i r_0(i)$.  The fact that $[h_{\ell},e_{\ell}] = 2 e_{\ell}$ follows, and then that $[h_{\ell},f_{\ell}] = -2 f_{\ell}$ follows from this by complex conjugation.
\end{proof}

\subsection{Good basis of $\p$}\label{subsec:pbasis} We now record a few more facts about our basis of $\p$. First, we have the following relations: $[h_3, h_{-3}] = -2 e_{\ell}$ and $[h_1(X),h_{-1}(Y)] = -(X,Y)[h_3,h_{-3}] = -2(X,Y)e_{\ell}$.

The pairings between the $h_j$'s and the $\overline{h_k}$'s is given by
\begin{itemize}
\item $B_{\g}(h_3,\overline{h_3}) = 1$
\item $B_{\g}(h_1(X),\overline{h_1(Y)}) = (X,Y)$
\item $B_{\g}(h_{-1}(Z),\overline{h_{-1}(W)}) = (Z,W)$
\item $B_{\g}(h_{-3},\overline{h_{-3}}) = 1$\end{itemize}
and all other pairings between the $h_j$'s, $h_k$'s, $\overline{h_j}$'s are $0$.  Here $X,Y,Z,W$ are real, i.e. in $J = J \otimes \R$.

If $a, d \in \R$ and $b, c \in J$, then we define
\[(a,b,c,d)_{\p} = ah_{3} + h_1(b) + h_{-1}(c) + dh_{-3}.\]
With this notation, one now computes the following bracket relations between the $n_{?}(X)$ and the $h_k$'s.
\begin{proposition}\label{prop:NFNE} One has
\[ [n_{F}(X), (a,b,c,d)_{\p}] = (0,aX,b \times X, (X,c))_{\p}\]
and
\[[n_{E}(X),(a,b,c,d)_{\p}] = ((b,X),c \times X, dX,0)_{\p}.\]
Additionally,
\begin{equation}\label{eqn:NH} [n_{H}(X),(a,b,c,d)_{\p}] = (\tr(X)a,\tr(X)b-\{X,b\},\{X,c\}-\tr(X)c,-\tr(X)d)_{\p}.\end{equation}
\end{proposition}
\begin{proof} The actions of $n_{E}(X)$ and $n_{F}(X)$ follow immediately from the explicit Cayley transform, Theorem \ref{thm:Cayley}. For the action of $n_{H}(X)$, it follows from Claim \ref{claim:NvNw} below that 
\[n_H(X) = [n_{E}(X),n_{F}(1)] = [n_{E}(1),n_{F}(X)].\]
From this, (\ref{eqn:NH}) follows from the statements for $n_{E}(X)$ and $n_{F}(X)$. \end{proof}

If $h$ is in the span of $h_{3}, h_{1}(X), h_{-1}(Y), h_{-3}$ then $[h_{\ell}, h] = h$ and $[h_{\ell},\overline{h}] = -\overline{h}$. Finally, one has the following relations:
\begin{itemize}
\item $[f_{\ell}, h_3] = - \overline{h_{-3}}$
\item $[f_{\ell},h_1(X)] = \overline{h_{-1}(X)}$
\item $[f_{\ell}, h_{-1}(X)] = -\overline{h_1(X)}$
\item $[f_{\ell},h_{-3}] = \overline{h_3}$.\end{itemize}

\subsection{The Cartan decomposition in the $\Z/3$-model} In this subsection we record the explicit Cartan decompositions in the $\Z/3$-model, and other related facts.

Define $\so_3^{\ell}:V_3 \rightarrow \g(J)$ as follows.  We denote $\so_3^{long}$ the image of this map; it is the long root $\so_3$ in $\g(J)$.  First, we define $u: V_3 \rightarrow \so(V_3)$ as $v_j \mapsto E_{j-1,j+1}-E_{j+1,j-1}$ with indices taken modulo $3$.  Then one computes $[v_j, v_{j+1}] = v_{j+2}$, and thus $[u(v),u(w)] = u(\iota(v \wedge w))$.  Additionally, $u(v) \cdot w = \iota(v \wedge w)$.  More canonically, one has
\[u(\iota(v \wedge w)) = w \otimes \iota(v) - v \otimes \iota(w).\]

One defines $\so_3^{\ell}(v) = \frac{1}{4}\left(u(v) + v \otimes 1 + \iota(v) \otimes 1\right)$.  With this definition, one computes 
\[[\so_3^{\ell}(v), \so_3^{\ell}(w)] = \so_3^{\ell}(\iota(v \wedge w)).\]

We set $e_{\ell} := \so_3^{\ell}(v_1 -iv_3)$, $f_{\ell} := \so_3^{\ell}(-v_1-iv_3)$, $h_{\ell} := \so_{3}^{\ell}(2iv_2)$.  Then these form an $\sl_2$-triple: $[e_{\ell},f_{\ell}] = h_{\ell}$, $[h_{\ell},e_{\ell}] = 2 e_{\ell}$, and $[h_{\ell},f_{\ell}] = -2f_{\ell}$.  We remark that $\overline{e_{\ell}} = -f_{\ell}$, $\overline{h_{\ell}}=-h_{\ell}$, $\overline{f_{\ell}} = -e_{\ell}$.  These elements correspond to the $e_{\ell}, f_{\ell}, h_{\ell}$ of subsection \ref{subsec:kbasis} under the explicit identification of the $\Z/3$ model of $\g(J)$ with the $\Z/2$-model.

Recall that $K = (\SU(2) \times L_0(J))/\mu_2$.  We now describe the Lie algebra $\mathfrak{l}_0(J)$ of $L_0(J)$ in the $\Z/3$-model.  Namely, for $v \in V_3$ and $X \in J$, define
\[n_{v}(X) = \frac{\tr(X)}{4} u(v) + \frac{1}{2} v\otimes \left(X- \frac{\tr(X)}{2}\right)+ \frac{1}{2}\iota(v)\otimes \left(X - \frac{\tr(X)}{2}\right).\]
With this definition, one has $n_{E}(X) = n_{v_1-iv_3}(X)$, $n_{H}(X) = n_{2iv_2}(X)$ and $n_F(X) = n_{-(v_1+iv_3)}(X)$.  Then we have $\mathfrak{l}_0(J) = \a(J) \oplus V_3 \otimes J$ with the element $v \otimes X$ of $V_3 \otimes J$ represented by $n_{v}(X)$.

One has the important relation $[\so_{3}^{long},n_{v}(X)] = 0$, as stated in the following claim. 
\begin{claim} If $v, w \in V_3$ and $X \in J$ then $[\so_{3}^{long}(v),n_{w}(X)] = 0$.\end{claim}
\begin{proof} We have
\[4 n_w(X) = \tr(X)u(w) + w \otimes (2X-\tr(X)) + \iota(w) \otimes (2X-\tr(X)).\]
Now
\begin{align*} [u(v),4n_w(X)] &= \tr(X)u(v \wedge w) + \iota(v \wedge w) \otimes(2X-\tr(X)) + v \wedge w \otimes (2X-\tr(X)),\\ [v \otimes 1,4 n_w(X)] &= [v \otimes 1, \iota(w) \otimes (2X-\tr(X))] + \tr(X) \iota(v \wedge w) \otimes 1 + v \wedge w \otimes (1 \times (2X-\tr(X))),\\ [\iota(v) \otimes 1, 4n_w(X)] &= [\iota(v) \otimes 1,w \otimes (2X-\tr(X))] + \iota(v \wedge w) \otimes (1 \times (2X-\tr(X))) +\tr(X) v \wedge w \otimes 1.\end{align*}
Since
\[(2X-\tr(X)) + \tr(X) + 1 \times (2X - \tr(X)) = 0,\]
the coefficient of $v \wedge w$ in the sum is $0$.  Identically, the coefficient of $\iota(v \wedge w)$ in the sum is $0$.

Now, in general, one has the relation
\begin{align*} [\iota(v)\otimes \iota(X), w \otimes Y] - [\iota(w)\otimes \iota(Y), v\otimes X] &= (X,Y)(w \otimes \iota(v) - v \otimes \iota(w))  \\ &\;\;+ (v,w)\left(\Phi_{\iota(X),Y} - \Phi_{\iota(Y),X}\right) \\ &= (X,Y) u(\iota(v \wedge w)) + (v,w)\Phi_{X \wedge Y}.\end{align*}
Since $(1,2X-\tr(X)) = -\tr(X)$ and $\Phi_{1 \wedge Y} = 0$ for any $Y \in J$, the coefficient of $u(\iota(v \wedge w))$ vanishes as well.  This completes the proof of the claim.
\end{proof}

For $X,Y \in J$, recall that $\Phi_{X \wedge Y} = \Phi_{\iota(X),Y} - \Phi_{\iota(Y),X}$.  Then
\begin{equation}\label{NvNw} [n_{v}(X),n_{w}(Y)] = n_{\iota(v \wedge w)}\left(\frac{ \{X,Y\}}{2}\right) + \frac{(v,w)}{4}\Phi_{X \wedge Y}.\end{equation}
\begin{claim}\label{claim:NvNw} The equality (\ref{NvNw}) is true.\end{claim}

Claim \ref{claim:NvNw} shows that the decomposition $\mathfrak{l}_0(J) = \a(J) \oplus V_3 \otimes J$ is the so-called Tits construction (\cite{TitsPlane}, \cite[Theorem 14]{jacobsonBook}) of $\mathfrak{l}_0(J)$ or $\h(J)$.
\begin{proof}[Proof of Claim \ref{claim:NvNw}] We have
\begin{align*} [\tr(X)u(v),4n_w(Y)] &= \tr(X)\tr(Y)u(v \wedge w) + \tr(X)\iota(v \wedge w) \otimes(2Y-\tr(Y)) \\ &\;\; + \tr(X) v \wedge w \otimes (2Y-\tr(Y)),\\ [v \otimes (2X-\tr(X)),4 n_w(Y)] &= [v \otimes (2X-\tr(X)), \iota(w) \otimes (2Y-\tr(Y))] \\ &\;\; + \tr(Y) \iota(v \wedge w) \otimes (2X-\tr(X)) \\ &\;\; + v \wedge w \otimes ((2X-\tr(X)) \times (2Y-\tr(Y))), \\ [\iota(v) \otimes (2X-\tr(X)), 4n_w(Y)] &= [\iota(v) \otimes (2X-\tr(X)),w \otimes (2Y-\tr(Y))] \\ &\;\;+ \iota(v \wedge w) \otimes ((2X-\tr(X)) \times (2Y-\tr(Y))) \\ &\;\; +\tr(Y) v \wedge w \otimes (2X-\tr(X)).\end{align*}

We first compute the term in degree $0$ of $[4n_v(X),n_w(Y)]$.  Note that
\[\tr(X)\tr(Y) + (2X-\tr(X),2Y-\tr(Y)) = 4(X,Y).\]
Furthermore, note that $\Phi_{(2X-\tr(X)) \wedge (2Y-\tr(Y))} = 4 \Phi_{X \wedge Y}$, because $\Phi_{1 \wedge W} = 0$ for any $W \in J$.  Thus the degree $0$ term is $4(X,Y)u(\iota(v \wedge w)) + 4(v,w)\Phi_{X\wedge Y}$.

We now compute the term in degree $1$.  It is $\iota(v \wedge w) \otimes T$, where
\begin{align*} T &= \tr(X)(2Y-\tr(Y)) + \tr(Y)(2X-\tr(X)) + (2X -\tr(X)) \times (2Y-\tr(Y)) \\ &= 4\left(X \times Y + \tr(X)Y + \tr(Y)X-\tr(X)\tr(Y)\right) \\ &= 4\left(\{X,Y\} -(X,Y)\right).\end{align*}
Here we have used
\[\{X,Y\} = \Phi_{1,X}(Y) = - 1 \times (X \times Y) + \tr(X)Y + \tr(Y)X = X\times Y + \tr(X)Y + \tr(Y)X - \tr(X \times Y).\]
The term in degree $2$ is computed identically.  

Thus
\begin{align*} 16[n_u(X),n_v(Y)] &= 4(X,Y)u(\iota(v \wedge w)) + 4(v,w)\Phi_{X\wedge Y} + 4 \iota(v \wedge w) \otimes (\{X,Y\} - (X,Y)) \\&\;\; + 4 v \wedge w \otimes (\{X,Y\}- (X,Y)).\end{align*}
Since $\tr(\{X,Y\}) = 2(X,Y)$, comparing with the definition of $n_{\iota(v \wedge w)}(\{X,Y\})$  gives the claim. \end{proof}

\subsection{The Iwasawa decomposition}\label{subsec:Iwasawa} We now give the Iwasawa decomposition.  First, we make a notation. We set $\epsilon = \mm{1}{}{}{-1} \in \sl_2 \subseteq \g(J)_0$. In the $\Z/3$-model, $\epsilon = E_{11}-E_{33}$. 

We have
\begin{align*} h_3 &= \frac{1}{2}\mm{-1}{i}{i}{1} \\  &= \frac{-1}{2}\mm{1}{}{}{-1} -\frac{i}{2} \mm{}{1}{-1}{} + i\mm{0}{1}{0}{0} \\ &= -\frac{1}{2}\epsilon + i\mm{0}{1}{0}{0} -\frac{1}{4}h_{\ell} - \frac{1}{4}n_{H}(1) \\ h_1(X) &= ie \otimes V(X) - n_{F}(X) \\ h_{-1}(Z) &= \frac{i}{2}(n(Z) + n^\vee(\iota(Z)) + \frac{1}{2}M(\Phi_{1,Z}) \\ &= \frac{1}{2}M(\Phi_{1,Z}) + in(Z) + \frac{i}{2}\left(n(-Z) + n^\vee(\iota(Z))\right) \\ &= \frac{1}{2}M(\Phi_{1,Z}) + in(Z) + \frac{\tr(Z)}{4}h_{\ell} + \frac{1}{2} n_{H}(Z - \tr(Z)/2) \\ h_{-3} &= \frac{1}{4}(ie-f) \otimes r_0(i) \\ &= i\frac{1}{2 } e \otimes r_0(i) - e_{\ell}.\end{align*}

\section{The differential equations}\label{sec:diffEq} In this section, we record in coordinates the differential operator $\mathcal{D}_n$.  In other words, we write down the differential equations satisfied by modular forms $F: G_J^{0} \rightarrow \Vm^\vee$, as explained in the introduction.  The operator $\mathcal{D} = \mathcal{D}_n$ is sometimes called the Schmid operator; see \cite{schmid} and \cite[Theorem A]{yamashita1}. 

\subsection{The Schmid operator} For a function $F: G(\R) \rightarrow \Vm^\vee$ satisfying $F(gk) = k^{-1} \cdot F(g)$, define $\widetilde{D}F$ as
\[\widetilde{D}(F) = \sum_{i}{b_iF \otimes b_i^\vee},\]
a function from $G(\R) \rightarrow \Vm^\vee \otimes \p^\vee$.  Here $b_i$ are a basis of $\p$, $b_i^\vee$ is the dual basis of $\p^\vee$, and $b_iF$ denotes the right-regular action.  Denote by $V_{-}$ the $K$-representation $Sym^{2n-1}(V_2) \boxtimes W_J$.  There is a $K$-equivariant surjection $pr_{-}: \Vm^\vee \otimes \p^\vee \rightarrow V_{-}$. The Schmid operator $\mathcal{D}$ is defined as the composition $\mathcal{D} = pr_{-} \circ \widetilde{D}$ of $pr_{-}$ and $\widetilde{D}$.  We make this explicit below for our case of interest.  See, e.g., \cite[Theorem A]{yamashita1} for the definition of the Schmid operator in the general case.

In our case,
\begin{align*} \widetilde{D}F &= h_{3} F \otimes \overline{h_{3}} + \overline{h_{3}}F \otimes h_{3} + h_{-3} F \otimes \overline{h_{-3}} + \overline{h_{-3}}F \otimes h_{-3} \\ &\; + \sum_{\alpha}{h_1(E_{\alpha})F \otimes \overline{h_{1}}(E_{\alpha}^\vee)} + \sum_{\alpha}{\overline{h_1}(E_{\alpha})F \otimes h_{1}(E_{\alpha}^\vee)} \\ &\; + \sum_{\alpha}{h_{-1}(E_{\alpha})F \otimes \overline{h_{-1}}(E_{\alpha}^\vee)} + \sum_{\alpha}{\overline{h_{-1}}(E_{\alpha})F \otimes h_{-1}(E_{\alpha}^\vee)}.\end{align*}

We use the following notation for the right regular action of certain elements of $\g(J)$ on functions on $G_J(\R)$:
\begin{enumerate}
\item We denote by $\partial_{Heis}$ the action of the element $\mm{0}{1}{0}{0}$ in $\sl_2 \subseteq \g(J)_0$.
\item We denote by $\partial^{W}_v$ the action of the element $e \otimes v$ in $\g(J)_1$.
\item For $E \in J$, we denote by $D_{Z}(E)$ the action of the Lie algebra element $\frac{1}{2}M(\Phi_{1,E}) - in_{L}(E)$.
\item Again, for $E \in J$, we denote by $D_{Z^*}(E)$ the action of the Lie algebra element $\frac{1}{2}M(\Phi_{1,E}) + in_{L}(E)$.
\end{enumerate}

\subsection{The expression for $\widetilde{D}F$} Applying the Iwasawa decomposition of $\p$ given in subsection \ref{subsec:Iwasawa}, we get the following for $\widetilde{D}F$.
\begin{align*} \widetilde{D}F &= (1-\epsilon/2)F \otimes \overline{h_3} + \frac{1}{4}\left(h_{\ell} + n_{H}(1)\right) \cdot \left(F \otimes \overline{h_3}\right) + i \partial_{Heis}F \otimes \overline{h_3} \\ &\; +  (1-\epsilon/2)F \otimes h_3 - \frac{1}{4}\left(h_{\ell} + n_{H}(1)\right) \cdot \left(F \otimes h_3\right) -i\partial_{Heis}F \otimes h_3\\ &\; +\frac{i}{2} \partial^{W}_{r_0(i)} F \otimes \overline{h_{-3}} + e_{\ell}\cdot \left( F\otimes \overline{h_{-3}}\right) + F \otimes h_3 \\ &\; -\frac{i}{2} \partial^W_{r_0(-i)} F \otimes h_{-3} - f_{\ell}\cdot \left( F\otimes h_{-3}\right) + F \otimes \overline{h_3} \\ &\; + \sum_{\alpha}{\left(i \partial^{W}_{V(E_{\alpha})} F \otimes \overline{h_1}(E_{\alpha}^\vee) + n_{F}(E_{\alpha}) \cdot \left(F \otimes \overline{h_1}(E_{\alpha}^\vee)\right) + F \otimes \overline{h_3}\right)} \\ &\; + \sum_{\alpha}{\left(-i\partial^{W}_{V(E_{\alpha})^*} F\otimes h_1(E_{\alpha}^\vee) - n_{E}(E_{\alpha}) \cdot \left(F \otimes h_1(E_{\alpha}^\vee)\right) + F \otimes h_3\right)}\\ &\; + \sum_{\alpha}{\left( D_{Z^*(E_{\alpha})}F \otimes \overline{h_{-1}}(E_{\alpha}^\vee) -\frac{1}{2} F \otimes \overline{h_{-1}}(\{E_{\alpha},E_{\alpha}^\vee\})\right)} \\ &\; + \sum_{\alpha}{\left(- \left(\frac{\tr(E_{\alpha})}{4} h_{\ell} + \frac{1}{2} n_{H}(E_{\alpha} - \tr(E_{\alpha})/2)\right) \cdot \left(F \otimes \overline{h_{-1}}(E_{\alpha}^\vee)\right) \right)} \\ &\; + \sum_{\alpha}{\left( D_{Z(E_{\alpha})}F \otimes h_{-1}(E_{\alpha}^\vee) -\frac{1}{2} F \otimes h_{-1}(\{E_{\alpha},E_{\alpha}^\vee\})\right)} \\ &\; + \sum_{\alpha}{\left(\left(\frac{\tr(E_{\alpha})}{4} h_{\ell} + \frac{1}{2} n_{H}(E_{\alpha} - \tr(E_{\alpha})/2)\right) \cdot \left(F \otimes h_{-1}(E_{\alpha}^\vee)\right) \right)} \end{align*}

\subsection{Simplified equations} Now we assume that $\Vm = Sym^{2n} V_2 \boxtimes 1$, and compute $\mathcal{D}F$.  The operator $\mathcal{D}F$ is the composition of $\widetilde{D}F$ and the contraction 
\[pr_{-}:\Vm \otimes \p \simeq \Vm \otimes (V_2 \otimes W_J) = (Sym^{2n}V_2 \otimes V_2) \boxtimes W_J \rightarrow Sym^{2n-1}V_2 \boxtimes W_J.\]
Thus we employ the contractions
\[ \frac{x^{n+k} y^{n-k}}{(n+k)!(n-k)!} \otimes y \mapsto \frac{x^{n+k-1} y^{n-k}}{(n+k-1)!(n-k)!}\]
and
\[\frac{x^{n+k} y^{n-k}}{(n+k)!(n-k)!} \otimes x \mapsto - \frac{x^{n+k} y^{n-k-1}}{(n+k)!(n-k-1)!}.\]

Set $[x^j] = \frac{x^j}{j!}$ and similarly $[y^k] = \frac{y^{k}}{k!}$.  Then these contractions are
\[ [x^{n+k}][y^{n-k}] \otimes y \mapsto [x^{n+k-1}][ y^{n-k}]\]
and
\[[x^{n+k}] [y^{n-k}] \otimes x \mapsto - [x^{n+k}][y^{n-k-1}].\]
For $-n \leq k \leq n$, we denote by $F_k$ the coefficient of $[x^{n+k}][y^{n-k}]$ in $F$, i.e., 
\[F = \sum_{-n \leq k \leq n}{F_k [x^{n+k}][y^{n-k}]}.\]

The isomorphism $\p \simeq V_2 \boxtimes W_J$ is given as follows.  Denote $x, y$ the standard basis of $V_2$, so that $h_{\ell} x = x$, $h_{\ell} y = -y$, $e_{\ell} x = 0$, $f_{\ell} x = y$, $f_{\ell} y = 0$.  By the Cayley transform, Theorem \ref{thm:Cayley}, we have the following correspondence: 
\begin{align*} h_3 &\mapsto x \boxtimes (1,0,0,0) & \overline{h_{-3}} &\mapsto -y\boxtimes (1,0,0,0) \\ h_1(X) &\mapsto x \boxtimes (0,X,0,0) & \overline{h_{-1}}(X) &\mapsto y \boxtimes (0,X,0,0) \\ h_{-1}(Y) &\mapsto x \boxtimes (0,0,Y,0) & \overline{h_1}(Y) &\mapsto -y \boxtimes (0,0,Y,0) \\ h_{-3} & \mapsto x \boxtimes (0,0,0,1) & \overline{h_3} & \mapsto y \boxtimes (0,0,0,1).\end{align*}

One finds the following for the Schmid operator.  Denote by $E_{\alpha}$ the elements of a basis of $J$ and $E_{\alpha}^\vee$ the dual basis for the trace pairing.
\begin{theorem}\label{thm:Schmid1} Up to a single nonzero constant, the Schmid operator $\mathcal{D}F$ is given as follows.
\begin{enumerate}
\item The coefficient of $[x^{n+v-1}][y^{n-v}] \boxtimes (0,0,0,1)$ in $DF$ is
\[-\frac{1}{2}(\epsilon - 2(n+1) + v)F_v + \frac{i}{2} \partial^W_{r_0(-i)} F_{v-1} + i \partial_{Heis}F_v.\]
\item The coefficient of $[x^{n+v}][y^{n-v-1}] \boxtimes (1,0,0,0)$ in $DF$ is
\[ \frac{1}{2}(\epsilon - 2(n+1) - v)F_v - \frac{i}{2} \partial^{W}_{r_0(i)} F_{v+1}+ i\partial_{Heis}F_v.\]
\item The coefficient of $[x^{n+v}][y^{n-v-1}] \boxtimes (0,0,E_{\alpha}^\vee,0)$ in $DF$  is
\[- (D_{Z(E_{\alpha})} + \frac{v}{2} \tr(E_{\alpha}))F_{v} -i \partial^{W}_{V(E_{\alpha})}F_{v+1}.\]
\item The coefficient of $[x^{n+v-1}][y^{n-v}]\boxtimes (0,E_{\alpha}^\vee, 0,0)$ in $DF$ is
\[(D_{Z^*(E_{\alpha})} - \frac{v}{2} \tr(E_{\alpha}))F_{v} + i \partial^{W}_{V(E_{\alpha})^*} F_{v-1}.\]
\end{enumerate}
\end{theorem}
\begin{proof} This is a slightly tedious, but direct verification.  To get these coefficients from our expression for $\widetilde{D}F$, we use the following several contractions and computations:
\begin{align*}
&pr_{-}(F \otimes \overline{h_3}) =\sum_{k}{F_{k}[x^{n+k-1}][y^{n-k}] \boxtimes (0,0,0,1)}
\\ & (h_{\ell} + n_{H}(1))\cdot \left([x^{n+k-1}][y^{n-k}] \boxtimes (0,0,0,1)\right) = (2k-4) [x^{n+k-1}][y^{n-k}] \boxtimes (0,0,0,1)
\\ & pr_{-}(F \otimes h_3) = \sum_{k}{(-1)F_k [x^{n+k}][y^{n-k-1}] \boxtimes (1,0,0,0)}.
\\ & (h_{\ell} + n_{H}(1))\cdot \left([x^{n+k}][y^{n-k-1}] \boxtimes (1,0,0,0)\right) = (2k+4)[x^{n+k}][y^{n-k-1}] \boxtimes (1,0,0,0).
\\ & pr_{-}(F \otimes \overline{h_{-3}}) = \sum_{k}{(-1)F_k [x^{n+k-1}][y^{n-k}] \boxtimes (1,0,0,0)}.
\\ & pr_{-}(e_{\ell}(F \otimes \overline{h_{-3}})) = \sum_{k}{(-1)(n+k)F_{k}[x^{n+k}][y^{n-k-1}] \boxtimes (1,0,0,0)}.
\\ & pr_{-}(F \otimes h_{-3}) = \sum_{k}{(-1)F_k [x^{n+k}][y^{n-k-1}] \boxtimes (0,0,0,1)}.
\\ & pr_{-}(f_{\ell} (F \otimes h_{-3}))= \sum_{k}{(-1)(n-k)F_k [x^{n+k-1}][y^{n-k}] \boxtimes (0,0,0,1)}.
\\ & pr_{-}(F \otimes \overline{h_1}(E_\alpha^\vee))= \sum_{k}{(-1)F_k [x^{n+k-1}][y^{n-k}] \boxtimes (0,0,E_{\alpha}^\vee,0)}.
\\ & pr_{-}(n_{F}(E_{\alpha}) \cdot (F \otimes \overline{h_1}(E_{\alpha}^\vee))) = \sum_{k}{(-1)F_k [x^{n+k-1}][y^{n-k}] \boxtimes (0,0,0,1)}.
\\ & pr_{-}(F \otimes h_1(E_{\alpha}^\vee)) = \sum_{k}{(-1)F_k [x^{n+k}][y^{n-k-1}] \boxtimes (0,E_{\alpha}^\vee,0,0)}.
\\ & -pr_{-}(n_{E}(E_{\alpha}) \cdot (F \otimes h_1(E_{\alpha}^\vee)))= \sum_{k}{F_k [x^{n+k}][y^{n-k-1}] \boxtimes (1,0,0,0)}.
\\ & -\frac{1}{2} pr_{-}(F \otimes \overline{h_{-1}}(\{E_{\alpha},E_{\alpha}^\vee\}))= \sum_{k}{(-1) F_k [x^{n+k-1}][y^{n-k}] \boxtimes (0,\{E_{\alpha},E_{\alpha}^\vee\}/2,0,0)}.
\\ & pr_{-}(F \otimes \overline{h_{-1}}(E_{\alpha}^\vee))= \sum_{k}{F_k [x^{n+k-1}][y^{n-k}]\boxtimes (0,E_{\alpha}^\vee,0,0)}.
\\ & -\frac{\tr(E_{\alpha})}{4}pr_{-}(h_{\ell} \cdot (F \otimes \overline{h_{-1}}(E_{\alpha}^\vee)))= \sum_{k}{(2k-1)F_k [x^{n+k-1}][y^{n-k}] \boxtimes \left\{(-1/4)(0,\tr(E_{\alpha})E_{\alpha}^\vee,0,0)\right\}}.
\\ & -\frac{1}{2}pr_{-}(n_{H}(E_{\alpha}-\tr(E_{\alpha})/2) \cdot ( F \otimes \overline{h_{-1}}(E_{\alpha}^\vee))) \\ &\qquad = \sum_{k}{F_k [x^{n+k-1}][y^{n-k}] \boxtimes (0,\{E_{\alpha},E_{\alpha}^\vee\}/2 - \tr(E_{\alpha})E_{\alpha}^\vee/4,0,0)}.
\\ & pr_{-}(F \otimes h_{-1}(E_{\alpha}^\vee)) = \sum_{k}{(-1)F_k [x^{n+k}][y^{n-k-1}] \boxtimes (0,0,E_{\alpha}^\vee,0)}.
\\ & -\frac{1}{2}pr_{-}(F \otimes h_{-1}(\{E_{\alpha},E_{\alpha}^\vee\})) = \sum_{k}{F_k [x^{n+k}][y^{n-k-1}] \boxtimes (0,0,\{E_{\alpha},E_{\alpha}^\vee\}/2,0)}.
\\ & \frac{\tr(E_{\alpha})}{4} pr_{-}(h_{\ell} (F \otimes h_{-1}(E_{\alpha}^\vee))) = \sum_{k}{(-1)(2k+1) F_k [x^{n+k}][y^{n-k-1}] \boxtimes (0,0,\tr(E_{\alpha})E_{\alpha}^\vee/4,0)}
\\ & \frac{1}{2} pr_{-}(n_{H}(E_{\alpha}-\tr(E_{\alpha})/2) \cdot (F \otimes h_{-1}(E_{\alpha}^\vee))) \\ &\qquad = \sum_{k}{(-1) F_k [x^{n+k}][y^{n-k-1}] \boxtimes (0,0,\{E_{\alpha},E_{\alpha}^\vee\}/2 - \tr(E_{\alpha})E_{\alpha}^\vee/4,0)}.
\end{align*}
\end{proof}

\subsection{Differential operators in coordinates} Suppose given a function $F$ or $F_v$ as above.  Then define $\phi = \phi_{F}$, a function of $(\mu,x,M)$, $\mu \in \R_{> 0}$, $x\in  W_J(\R)$, $M \in H_J(\R)^0$ as
\[\phi_{F}(\mu,x,M) = F(\exp(\mu \mm{0}{1}{0}{0}) \exp( e\otimes x) M).\]
Furthermore, for $M \in H_J(\R)^{0}$, we set $w = |\nu(M)|^{1/2}$.

One has the relation 
\[\exp(e \otimes u) \exp(e \otimes v) = \exp(\alpha \langle u,v\rangle \mm{0}{1}{0}{0})\exp(e \otimes (u+v)),\]
when the bracket on $\g(J)$ is defined in terms of the constant $\alpha$.  Thus, taking $\alpha = \frac{1}{2}$, one gets the following for how some differential operators act in coordinates:
\begin{enumerate}
\item $\partial_{Heis} = \nu(M) \partial_{u} = w^2 \partial_{\mu}$
\item $\partial^{W}_{v} = \frac{1}{2}\langle x,M \cdot v\rangle \partial_{\mu} + D_{M \cdot v}^x$.\end{enumerate}
Here $M \cdot v$ is the action of $M \in H_J$ on $v \in W_J$, and $D_{v}^{x}$ is the partial derivative of the coordinate $x$ in the $v$ direction, i.e., $D^{x}_{v}\phi(\mu,x,M) = \frac{d}{dt} \phi(\mu,x+ tv,M)|_{t=0}$.

For $M \in H_J(\R)$, we define $\widetilde{Z} = M \cdot r_0(i)$ and $\widetilde{Z}^* = M \cdot r_0(-i)$. 

\subsection{The differential equations, again} We now write the differential equations $\mathcal{D}_n F = 0$ in coordinates.  

Here and below, set $H_J(\R)^{0,1} = H_J(\R)^{0} \cap H_J(\R)^{1}$, the elements of $H_J(\R)^{0}$ with similitude equal to $1$.  If $F: G_J^{0} \rightarrow \Vm^\vee$ is a function, we denote by 
\[\phi_{F}: \R \times W_J(\R) \times H_J(\R)^{0,1} \times \R_{>0} \rightarrow \Vm^\vee\]
the function defined by
\[\phi_F(\mu,x,m,w) = F(\exp(\mu \mm{0}{1}{0}{0}) \exp( e\otimes x) m w).\]
Here we are identifying $w \in \R^\times$ with the element in the center of $H_J(\R)$ that acts on $W_J$ as multiplication by $w$.  Note that in these coordinates, $\epsilon \phi_F = w \partial_{w} \phi_F$.
 
\begin{theorem}\label{thm:Schmid2} Suppose $F: G_J^{0} \rightarrow \Vm^\vee$ is $K$-equivariant with $\mathcal{D}_n F = 0$, and $\phi_F$ is defined as above. For $v \in \Z$ with $-n \leq v \leq n$, denote by $\phi_{v}$ the function $\R \times W_J(\R) \times H_J(\R)^{0,1} \times \R_{>0} \rightarrow \C$ defined by $\phi_{F} = \sum_{-n \leq v \leq n}{\phi_{v} [x^{n+v}][y^{n-v}]}$ where $x,y$ is the standard basis of $V_2$, normalized above. Then the functions $\phi_v$ satisfy the following differential equations.
\begin{align*} \left(w \partial_{w} - 2(n+1) + v -2iw^2\partial_{\mu}\right)\phi_{v} &= i\left(D^{x}_{\widetilde{Z}^*} + \frac{1}{2}\langle x,\widetilde{Z}^*\rangle \partial_{\mu}\right)\phi_{v-1}.\\
 \left(w \partial_{w} - 2(n+1) - v +2iw^2\partial_{\mu}\right)\phi_{v} &= i\left(D^{x}_{\widetilde{Z}} + \frac{1}{2}\langle x,\widetilde{Z}\rangle \partial_{\mu}\right)\phi_{v+1}.\end{align*}
And for all $E \in J$,
\begin{align*} \left(D_{Z(E)} + \frac{v}{2}\tr(E)\right)\phi_{v} &= - i\left(D^{x}_{M V(E)} + \frac{1}{2} \langle x, M V(E)\rangle \partial_{\mu}\right) \phi_{v+1}.\\
\left(D_{Z^*(E)} - \frac{v}{2}\tr(E)\right)\phi_{v} &=- i\left(D^{x}_{M V(E)^*} + \frac{1}{2} \langle x, M V(E)^*\rangle \partial_{\mu}\right) \phi_{v-1}.\end{align*}
Here $M = m w \in H_J(\R)$.
\end{theorem}
\begin{proof} This follows immediately from Theorem \ref{thm:Schmid1} and the remarks above.\end{proof}

\subsection{Character equations} Now assume $\Wh$ satisfies the equations $\mathcal{D}_n \Wh = 0$, and also satisfies the equivariance 
\[\Wh(\exp(X)g) = \psi(\langle \qch, X \rangle)\Wh(g) = e^{ i \langle \qch, X\rangle}\Wh(g)\]
for $g \in G_J(\R)$ and $X \in \mathfrak{n}$, the Lie algebra of $N(\R)$, and some $\qch \in W_J$.  Then $\mathfrak{n}$ acts on $\Wh$ via $(X\Wh)(m) = i \langle \qch, m \cdot X\rangle \Wh(m)$.  Here $X \in \mathfrak{n}$ and $m \in H_J(\R)$ is in the Levi subgroup of the Heisenberg parabolic of $G_J$.  

We obtain the following corollary of Theorem \ref{thm:Schmid2}.
\begin{corollary}\label{cor:SchmidChar} Let the notation be as in Theorem \ref{thm:Schmid2}, and assume moreover that $\Wh(\exp(X)g) =  e^{ i \langle \qch, X\rangle}\Wh(g)$ for all $g \in G_J(\R)$ and $X \in \mathfrak{n}$.  Then
\begin{align*} (w\partial_{w} - 2(n+1) + k)\phi_k &= -  \langle \qch, \widetilde{Z}^*\rangle \phi_{k-1}.\\ 
(w\partial_{w} - 2(n+1) - k)\phi_k  &=-  \langle \qch, \widetilde{Z} \rangle \phi_{k+1}.\end{align*}
Moreover, for all $E \in J$,
\begin{align*} (D_{Z(E)} + \frac{k}{2} \tr(E))\phi_{k} =  \langle \qch, M V(E) \rangle \phi_{k+1}.\\
(D_{Z^*(E)} - \frac{k}{2} \tr(E))\phi_{k} =  \langle \qch, M V(E)^* \rangle \phi_{k-1}.\end{align*}
\end{corollary}
\begin{proof} To go from the general differential equations of Theorem \ref{thm:Schmid2} to the ones above, one simply imposes $\partial_{\mu} = 0$ and $D_{v}^{x}\phi = i\langle \qch,v\rangle\phi$.\end{proof}

\subsection{Useful formulae and definitions} We now briefly give some formulae which will be used in the next section. 

Recall that
\[V(E) = \frac{\tr(E)}{2}r_0(i) + i(0,E,-i\tr(E) + iE,-\tr(E)) = \frac{\tr(E)}{2}r_0(i) + in_G(-i)(0,E,0,0).\]
Thus $M V(E) = \frac{\tr(E)}{2} \widetilde{Z} + i Mn_{G}(-i)(0,E,0,0)$.  

Suppose $M = w n(-X)M_Y$, where $M_Y := M(n(Y)^{1/2}, U_{Y^{1/2}})$.  Then 
\begin{align*} Mn_{G}(-i) &= w n_G(-X)M_Y n_G(-i) = w n_{G}(-X) M_Y n_G(-i) M_Y^{-1} M_Y = w n_{G}(-X)n_{G}(-iY)M_Y \\ &= w n_{G}(-Z)M_Y.\end{align*}
Hence in this case 
\[Mn_{G}(-i)(0,E,0,0) = w n(Y)^{-1/2}n_{G}(-Z)(0,E_Y,0,0) = wn(Y)^{-1/2}(0,E_Y, -Z \times E_Y, (Z^\#,E_Y))\]
where $E_Y = U_{Y^{1/2}}(E)$, which equals $Y^{1/2}EY^{1/2}$ if the cubic norm structure $J$ is special.  Thus for $M = w n(-X)M_Y$,
\begin{equation}\label{eqn:MVE} M V(E) = wn(Y)^{-1/2}\left(\frac{\tr(E)}{2} r_0(Z) + (0,E_Y, -Z \times E_Y, (Z^\#,E_Y))\right).\end{equation}

\section{Solution to the equations}\label{sec:solutions} We now solve the differential equations of Corollary \ref{cor:SchmidChar}, assuming $\qch \neq 0$.  Before proceeding, we note that if $M = w n(-X)M_Y$, then $\langle \qch, M r_0(i)\rangle = w n(Y)^{-1/2} p_\chi(Z)$, and thus $w^{-1}n(Y)^{1/2} \langle \qch, M r_0(i) \rangle = w^{-1}n(Y)^{1/2} \langle \qch, \widetilde{Z} \rangle$ is holomorphic in $Z$.

\subsection{First steps} Restrict the functions $\Wh_v$ to $H_J(\R)^{0,1} \times \R_{>0}$ as in Theorem \ref{thm:Schmid2}.  Define $G_v$ via the equality $\Wh_v = w^{2n+2}G_v$.  Then the $G_v$ satisfy the following differential equations:
\[(w\partial_w + v)G_v = - \langle \qch, \widetilde{Z}^*\rangle G_{v-1}\]
and
\[(w\partial_w - v) G_v = - \langle \qch, \widetilde{Z}\rangle G_{v+1}.\]
Hence 
\[((w\partial_w)^2-v^2)G_v = |\langle \qch, \widetilde{Z}\rangle|^2 G_v = w^2 |\langle \qch,m r_0(i)|^2 G_v.\]
Denote by $K_v$ the $v$'th $K$-Bessel function, so that $((z\partial_{z})^2 - v^2)K_v(z) = z^2K_v(z)$. Assuming $\Wh_v = w^{2n+2} G_v$ is of moderate growth as $w \rightarrow \infty$, we obtain 
\[G_v(m,w) = Y_v(m) K_v(w |\langle \qch,m r_0(i)|) = Y_{v}(m) K_v(|\langle \qch, \widetilde{Z} \rangle|)\]
for some function $Y_v(m)$ that does not depend on $w$. 

We record the following well-known properties of the $K$-Bessel functions:
\begin{lemma}\label{lem:KBes} The $K$-Bessel functions satisfy the following identities:
\begin{enumerate}
\item $((z\partial_{z})^2 - v^2)K_v(z) = z^2K_v(z)$
\item $-z^{-v} \partial_z(z^v K_v(z)) = K_{v-1}(z)$
\item $-z^{v} \partial_{z}(z^{-v} K_v(z)) = K_{v+1}(z)$
\item $-(z \partial_{z}-v)K_v(z) = zK_{v+1}(z)$
\item $-(z\partial_z+v)K_v(z) = zK_{v-1}(z)$.\end{enumerate}
\end{lemma}

Set $u = |\langle \qch, \widetilde{Z}\rangle |$.  Applying Lemma \ref{lem:KBes}, one has
\[  \langle \qch, \widetilde{Z}\rangle Y_{v+1} K_{v+1}(u) =  \langle \qch, \widetilde{Z} \rangle G_{v+1} = -(w\partial_w -v)G_{v} = -(w\partial_w - v)K_v(u) Y_v(m).\]
But 
\[-(w\partial_w - v)K_v(u) = -(u\partial_u -v)K_v(u) = uK_{v+1}(u) = |\langle \qch, \widetilde{Z}\rangle| K_{v+1}(u).\]
Therefore we obtain
\[ \langle \qch, \widetilde{Z}\rangle  Y_{v+1} = |\langle \qch, \widetilde{Z}\rangle| Y_{v}.\]
Thus, on an open set where $\langle \qch, \widetilde{Z} \rangle \neq 0$, 
\[ Y_{v+1} = \left(\frac{|\langle \qch, \widetilde{Z}\rangle| }{\langle \qch, \widetilde{Z}\rangle}\right) Y_{v}\]
and so
\[G_v(m,w) = Y_0(m)\left(\frac{ | \langle \qch, \widetilde{Z} \rangle|}{ \langle \qch, \widetilde{Z}\rangle }\right)^{v} K_v( |\langle \qch, \widetilde{Z} \rangle|)\]
for some function $Y_0(m)$ that does not depend on $w$.  We will use the other two differential equations to show that $Y_0$ is constant.

\subsection{Second steps} Recall the differential operator $D_{Z(E)}$, which by definition given by the action of $\frac{1}{2}M(\Phi_{1,E}) - i n_{L}(E)$, and $D_{Z^*(E)}$, which by definition is given by the action of $\frac{1}{2}M(\Phi_{1,E}) + i n_{L}(E)$.

Suppose $F$ is a function on the Levi $H_J(\R)$ of the Heisenberg parabolic.  Abusing notation slightly, we denote by $F$ a new function, defined in terms of $F$, on the variables $w, X, Y$, so that $w \in \R^{\times}_{>0}$, $X, Y \in J$ and $Y > 0$.  Namely, one defines
\[\phi_{F}(w,X,Y) = \Wh\left(w n_{G}(-X) M(n(Y)^{1/2},U_{Y^{1/2}})\right) = F\left(w n_{G}(-X)M_Y\right).\]
We calculate how $D_{Z(E)}$ and $D_{Z^*(E)}$ act on $\phi_{F}$ in the $w,X,$ and $Y$ variables.   Recall the notation $M_Y = M(n(Y)^{1/2},U_{Y^{1/2}})$ and $E_{Y} = U_{Y^{1/2}}(E)$.
\begin{lemma}\label{lem:DZ} Suppose $F$ is right invariant under $A_J$.  One has
\[n_{L}(-E)\phi_{F}(w,X,Y) = \frac{d}{dt}\left(\phi_{F}(w, X + tE_{Y},Y)\right)|_{t=0}\]
and
\[\frac{1}{2}M(\Phi_{1,E})\phi_{F}(w,X,Y) = \frac{d}{dt}\left(\phi_{F}(w,X,Y + tE_Y)\right)|_{t=0}.\]
\end{lemma}
\begin{proof} The first identity is immediate from the definitions and the relation $M_Y n_{G}(E)M_Y^{-1} = n_{G}(E_Y)$.  For the second identity, first note that 
\[\exp\left(\frac{1}{2}M(t\Phi_{1,E})\right) = U_{\exp(tE)^{1/2}} = M_{\exp(tE)}.\]
Now
\[U_{Y^{1/2}}U_{\exp(tE/2)} (1_J) = U_{Y^{1/2}}(\exp(tE)) =Y + tE_Y + O(t^2).\]
Thus $M_{Y}M_{\exp(tE)} = M_{Y+ tE_Y} a + O(t^2)$ for some $a \in A_J$ (that depends on $Y,E,$ and $t$).  Hence 
\[\frac{1}{2}M(\Phi_{1,E})\phi_{F}(w,X,Y) = \frac{d}{dt}\left(\phi_{F}(w,X,Y + tE_Y)\right)|_{t=0},\]
as claimed.\end{proof}

We have the following lemma.
\begin{lemma}\label{lem:DZq} Suppose $\alpha, \delta \in \C$ and $\beta \in J_{\C}$ and $\gamma \in J_{\C}^\vee$.  Then 
\[\frac{d}{dt}\left( \alpha N(W + tV) + (\beta, (W+tY)^\#) + (\gamma, W+tY) + \delta\right)|_{t=0} = (\alpha W^\# + \beta \times W + \gamma, V).\]
\end{lemma}
\begin{proof} This is an immediate and simple computation.\end{proof}

One makes the following computation:
Suppose $\qch = -(a,b,c,d)$, so that 
\[\langle \qch, r_0(Z)\rangle = p_\chi(Z) = aN(Z) + (b,Z^\#) + (c,Z) +d.\]
Then applying Lemmas \ref{lem:DZ} and \ref{lem:DZq}, one obtains
\begin{equation}\label{eqn:DZq} D_{Z(E)}(a N(Z) + (b,Z^\#) + (c,Z) +d) = 2i(az^\# + b \times Z + c,E_Y).\end{equation}
Applying Lemma \ref{lem:DZ}, one obtains
\begin{equation}\label{eqn:DZY} D_{Z(E)}(N(Y)) = \frac{d}{dt}(N(Y + tE_Y))|_{t=0} = (Y^\#,E_Y) = N(Y)\tr(E).\end{equation}
Thus, from (\ref{eqn:DZq}) and (\ref{eqn:DZY}), one gets that $\left(D_{Z(E)}\right)(\langle \qch,r_0(Z)\rangle N(Y)^{-1})$ is equal to
\[2i(a Z^\# + b \times Z + c, E_Y)N(Y)^{-1} - \tr(E)(aN(Z) + (b,Z^\#)+(c,Z)+d)N(Y)^{-1}.\]

Set $m = n_G(-X)M_Y$.  Then by (\ref{eqn:MVE})
\[\langle \qch, m V(E) \rangle = n(Y)^{-1/2} \tr(E) \langle \qch, r_0(Z)\rangle/2 - i n(Y)^{-1/2}(aZ^\# + b \times Z+c,E_Y).\]
One obtains 
\[D_{Z(E)}\left(\langle \qch, r_0(Z)\rangle N(Y)^{-1}\right) = -2 n(Y)^{-1/2} \langle \qch, m V(E)\rangle.\]
The conclusion is that
\begin{align*} D_{Z(E)}(|\langle \qch,r_0(Z)\rangle| n(Y)^{-1/2}) &= \frac{|\langle \qch, r_0(Z)\rangle|}{2 \langle \qch,r_0(Z)\rangle} n(Y)^{1/2} D_{Z(E)}\left(\langle \qch, r_0(Z)\rangle N(Y)^{-1}\right) \\ &= - \frac{|\langle \qch, r_0(Z)\rangle|}{ \langle \qch,r_0(Z)\rangle}\langle \qch, m V(E)\rangle.\end{align*}

We record what we have just proved in the following lemma.
\begin{lemma}\label{lem:DZu} Recall that $u = |\langle \qch, \widetilde{Z}\rangle|$.  One has
\[D_{Z(E)}(u) =  - \left(\frac{|\langle \qch, r_0(Z)\rangle|}{  \langle \qch,r_0(Z)\rangle}\right)\langle \qch, wm V(E)\rangle.\]
\end{lemma}

Applying Lemma \ref{lem:DZu} and the $D_{Z}$-differential equation to $G_0$ one finds $D_{Z(E)}(Y_0) = 0$.  Similarly, $D_{Z^*(E)}(Y_0) = 0$.  Hence $Y_0$ is constant, and so we have just proved the following result.  Here and below, by the ``Schmid equations'' we mean the equations of Corollary \ref{cor:SchmidChar}.
\begin{proposition}\label{prop:SchmidUnique} Suppose $\qch \neq 0$, but $\langle \qch, \widetilde{Z}\rangle = 0$ for some $Z \in \mathcal{H}_J$.  Then the $0$ function is the only solution to the Schmid equations of moderate growth.  If $\langle \qch, \widetilde{Z} \rangle \neq 0$ for all $Z \in \mathcal{H}_J$, then there is at most a one-dimensional space of solutions to the Schmid equations that are of moderate growth, and this space is spanned by 
\begin{equation}\label{propeqn:WwZ}\Wh_v(w,Z) = w^{2n+2} \left(\frac{|\langle \qch, \widetilde{Z}\rangle|}{\langle \qch, \widetilde{Z}\rangle}\right)^{v} K_{v}(|\langle \qch, \widetilde{Z} \rangle|).\end{equation}
\end{proposition}
\begin{proof} We have proved that the above $\Wh_v(w,Z)$ is the only possible moderate-growth solution to the equations $\mathcal{D}_n \Wh = 0$ on the open subset where $\langle \omega, \widetilde{Z} \rangle \neq 0$.  Thus the second statement of the proposition has been proved.  

For the first part, denote by $p_{\omega}$ the cubic polynomial on $\mathcal{H}_J$ associated to $\omega$, and suppose that there exists $Z_0 \in \mathcal{H}_J$ with $p_{\omega}(Z_0) =0$.  Furthermore, suppose that $\Wh$ is a solution to the Schmid equations; we must check $\Wh =0$.  To see this, fix a closed ball $\mathcal{B}$ of positive radius with center $Z_0$ and contained inside $\mathcal{H}_J$.   Denote by $U$ the subset of the interior of $\mathcal{B}$ with $p_{\omega} \neq 0$.  Because $p_{\omega}$ is a polynomial, $U$ is connected, open and dense in $\mathcal{B}$.  Now, consider the set $V$ of all pairs $(w,Z)$ with $w \in \R^\times_{>0}$ and $Z \in U$.  By what we have already proved, $\Wh$ is determined on $V$ to be a scalar multiple of the function (\ref{propeqn:WwZ}).  But since $p_{\omega}(Z_0) = 0$, if $\Wh$ is nonzero this function diverges as $Z \rightarrow Z_0$ but $w$ stays fixed.  It follows that $\Wh = 0$ and the proposition is proved.
\end{proof}

Note that in Proposition \ref{prop:SchmidUnique}, we have not assumed that the character of $N$ defined by $\qch$ is generic, only that it is nontrivial.  In case $\qch$ is generic, one can relate the condition that $\langle \qch, \widetilde{Z}\rangle$ is never $0$ to Wallach's admissibility condition in \cite{wallach}.

\subsubsection{Existence} The above arguments showed that the only solution to the Schmid equations of moderate growth takes the form
\begin{equation}\label{eqn:Fveqn} \Wh_v(w,Z) = w^{2n+2} \left(\frac{|\langle \qch, \widetilde{Z}\rangle|}{\langle \qch, \widetilde{Z}\rangle}\right)^{v} K_{v}(|\langle \qch, \widetilde{Z} \rangle|) = w^{2n+2} \left(\frac{\langle \qch, \widetilde{Z}\rangle^*}{|\langle \qch, \widetilde{Z}\rangle|}\right)^{v} K_{v}(|\langle \qch, \widetilde{Z} \rangle|).\end{equation}
Here, recall $\widetilde{Z} = M r_0(i)$.  We now show that the right-hand side of (\ref{eqn:Fveqn}) does satisfy all the Schmid equations.
\begin{proposition} The right-hand side of (\ref{eqn:Fveqn}) satisfies all the Schmid equations. \end{proposition}
\begin{proof} We already know that $\Wh_v$ satisfies the two equations with the differential operators $w\partial_{w}$; this is immediate from the arguments above.  We will show that $\Wh_v$ satisfies the differential equations with the operators $D_{Z(E)}$.  That $\Wh_v$ satisfies the equations with the operators $D_{Z^*(E)}$ proceeds through a nearly identical argument.

Recall that above we defined $u = |\langle \qch, \widetilde{Z}\rangle|$, and showed that 
\[D_{Z(E)}(u) = - \left(\frac{|\langle \qch, \widetilde{Z} \rangle|}{\langle \qch, \widetilde{Z}\rangle}\right) \langle \qch, MV(E) \rangle = - \langle \qch,\widetilde{Z}^*\rangle u^{-1} \langle \qch, MV(E)\rangle\]
for $M = w n_G(-X)M_Y$.  We also verified $D_{Z(E)}(N(Y)) = N(Y)\tr(E)$.  Finally, observe that $N(Y)^{1/2} \langle \qch, \widetilde{Z}\rangle^*$ is antiholomorphic, so 
\[D_{Z(E)}(N(Y)^{1/2} \langle \qch, \widetilde{Z}\rangle^* ) = 0.\]

With these facts recalled, we can now compute.  We have
\[G_v(w,Z) = \left(\langle \qch, \widetilde{Z}\rangle^* N(Y)^{1/2}\right)^{v} N(Y)^{-v/2} u^{-v}K_v(u).\]
Thus
\begin{align*} D_{Z(E)}(G_v) &= D_{Z(E)}\left(\left(\langle \qch, \widetilde{Z}\rangle^* N(Y)^{1/2}\right)^{v} N(Y)^{-v/2} u^{-v}K_v(u)\right) \\ &= \left(\langle \qch, \widetilde{Z}\rangle^* N(Y)^{1/2}\right)^{v} D_{Z(E)}(N(Y)^{-v/2} u^{-v} K_v(u)).\end{align*}

But now 
\begin{align*} D_{Z(E)}(N(Y)^{-v/2} u^{-v} K_v(u)) &= ((-v/2)N(Y)^{-v/2}\tr(E))u^{-v}K_v(u) \\ &\qquad + N(Y)^{-v/2}\partial_{u}(u^{-v}K_v(u))D_{Z(E)}(u) \\ &= ((-v/2)N(Y)^{-v/2}\tr(E))u^{-v}K_v(u) \\ &\qquad + N(Y)^{-v/2}(-u^{-v}K_{v+1}(u))(-\langle \qch,\widetilde{Z}\rangle^* u^{-1} \langle \qch, MV(E)\rangle) \\ &= N(Y)^{-v/2}\left\{\left(\frac{-v}{2}\right)\tr(E) u^{-v} K_{v}(u) \right. \\ &\qquad \left.+ \langle \qch, MV(E)\rangle \langle \qch, \widetilde{Z}\rangle^* u^{-(v+1)}K_{v+1}(u)\right\}.\end{align*}
Thus
\[(D_{Z(E)} + \frac{v}{2}\tr(E))G_{v} = \langle \qch, MV(E) \rangle G_{v+1}\]
as desired.
\end{proof}

\section{Final formula}\label{sec:final} In this section we work out the final details to get an exact formula for $\Wh$ when the character $\chi$ on $N(\R)$ is nontrivial.  That is, we prove the $\chi$-nontrivial part of Theorem \ref{intro:mainThm} and Corollary \ref{intro:Cor}.  In the proof below, set $H_J(\R)^{0}$ the connected component of the identity of $H_J(\R)$, and recall that $H_J(\R)^{0,1}$ denotes the subgroup of elements of $H_J(\R)^0$ with similitude equal to $1$.  Additionally, denote by $H_J(\R)^{\pm}$ the subgroup of $H_J(\R)$ generated by $H_J(\R)^{0}$ and $w_0 := \eta(-1) \in H_J(\R)$.  In case $G_J$ is of type $G_2, F_4, E_6, E_7$ or $E_8$, $H_J(\R)$ has two connected components, given by $\nu > 0$ and $\nu < 0$.  Thus in these exceptional Dynkin types, $H_J(\R) = H_J(\R)^{\pm}$. 

Observe that $w_0 \in K$.  Indeed, this follows from the formulas for the Cartan involution on $\h(J)^0$ and $\g(J)$.  Finally, denote by $K_H^{1}$ those elements of $H_J^{1}(\R)$ that commute with $J_2$, or equivalently, those elements of $H_J^{0,1}(\R)$ that preserve the positive definite inner product on $W_J$.

\begin{proof}[Proof of Theorem \ref{intro:mainThm}] We consider the case when $\chi$ is nontrivial.  Then, the first part of Theorem \ref{intro:mainThm} has already been proved.  That is, from Proposition \ref{prop:SchmidUnique}, if there exists $Z \in \mathcal{H}_J$ so that $p_{\chi}(Z) = 0$, then $\mathrm{Hom}_{N(\R)}(\pi_n,\chi) = 0$.

For the second part, we have already checked the formula of Theorem \ref{intro:mainThm} when 
\[g = n_{G}(-x)M(n(Y)^{1/2}, U_{Y^{1/2}}) w = n_{G}(-x)M_Y w,\]
by solving the differential equations.  If $g \in H_J(\R)^0$, $g = n_{G}(-x)M_Y w k$ for some $k \in K_H^{1}$.  Thus, to check the formula of Theorem \ref{intro:mainThm} on all of $H_J(\R)^{\pm}$, we must verify that the formula given in Theorem \ref{intro:mainThm} is equivariant for $K_H^1$ and $w_0$.   That is, if $k \in K_H^1$ or $k = w_0$, then $\Wh(gk) = k^{-1} \Wh(g)$.

Recall that if $k \in K_H^{1}$, we have $k r_0(i) = j(k,i) r_0(i)$ (this defines $j(k,i)$), and that $|j(k,i)| = 1$.
\begin{lemma} Suppose $k \in K_H^{1}$.  Then $k\cdot e_{\ell} = j(k,i) e_{\ell}$, $k \cdot f_{\ell} = j(k,i)^{-1} f_{\ell}$, and $k \cdot h_{\ell} = h_{\ell}$.  Furthermore, $w_0 \cdot e_{\ell} = f_{\ell}$, $w_{0}\cdot f_{\ell} = e_{\ell}$, and $w_0\cdot h_{\ell} = -h_{\ell}$.\end{lemma}
\begin{proof} The actions of $k$ and $w_0$ on $e_{\ell}$ and $f_{\ell}$ are immediately computed.  The actions on $h_{\ell}$ follow from these, as $h_{\ell} = [e_{\ell},f_{\ell}]$.\end{proof}

Now, $Sym^2(V_2) = \mathrm{Span}\{x^2, xy, y^{2}\}$ is identified with $\mathrm{Span}\{e_{\ell},h_{\ell},f_{\ell}\}$ via the map $e_{\ell} \mapsto x^2$, $h_{\ell} \mapsto -2xy$, and $f_{\ell} \mapsto -y^2$.  Indeed, we send $e_{\ell} \mapsto x^2$ as the are both highest weight vectors for the same Cartan, and the rest follows by applying lowering operators, i.e., by acting by $f_{\ell}$.  With this normalization, we arrive at the following.
\begin{lemma}\label{lem:kwaction} If $k \in K_H^{1}$, then $k$ acts on $x^{n+v}y^{n-v}$ as $j(k,i)^{v}$.  The element $w_0$ acts on $Sym^{2n}(V_2)$ as the element $\mm{0}{i}{i}{0}$, and thus takes $x^{n+v}y^{n-v}$ to $(-1)^{n} x^{n-v}y^{n+v}$. \end{lemma}
\begin{proof} These statements follow from the actions of $k$ and $w_0$ on $Sym^2(V_2) = \mathrm{Span}\{e_{\ell},h_{\ell},f_{\ell}\}$, and the equivariant surjection $S^{n}(S^2(V_2)) \rightarrow S^{2n}(V_2)$.\end{proof}

It follows from Lemma \ref{lem:kwaction} that the formula of Theorem \ref{intro:mainThm} is appropriately equivariant, as desired.  The part of Theorem \ref{intro:mainThm} concerning the case when $\chi = 1$ is proved below in Proposition \ref{prop:const}.\end{proof}

\begin{proof}[Proof of Corollary \ref{intro:Cor}] The corollary follows immediately from Theorem \ref{intro:mainThm}, except for the following two facts, which have not yet been proved:
\begin{enumerate}
\item If $w$ is rank four and $a(w) \neq 0$, then $q(w) < 0$.
\item If $\varphi$ is a cusp form and $a(w) \neq 0$, then $w$ is rank four (and thus $q(w) < 0$.)
\end{enumerate}
But these facts follow right away from Proposition \ref{prop:vRnks} below.\end{proof}

\section{Positivity}\label{sec:positivity} In this section we analyze whether the quantity $|\langle \qch, g r_0(i)\rangle|$ is bounded away from $0$ as $g \in H_J(\R)^{1}$ varies, for different elements $\qch \in W_J$.  For simplicity, in this section, we assume $J$ is simple, so that $J = \R$ or $J = H_3(C)$ for a composition algebra $C$.

We prove the following result.
\begin{proposition}\label{prop:vRnks} Suppose $J = \R$ or $J = H_3(C)$ for some composition algebra $C$.  
\begin{enumerate}
\item If $v \in W_J$ is rank four, and $q(v) > 0$, then there exists $g \in H_J(\R)^{1}$ with $\langle v, g r_0(i) \rangle = 0$.
\item If $v \in W_J$ is rank four, and $q(v) < 0$, then $|\langle v, g r_0(i)\rangle|$ is bounded away from $0$ for $g \in H_J(\R)^{1}$.
\item If $v \in W_J$ is of rank one, two, or three, then $|\langle v, g r_0(i)\rangle|$ is not bounded away from $0$ for $g \in H_J(\R)^{1}$.
\end{enumerate}
\end{proposition}

It follows from Proposition \ref{prop:vRnks} that the function
\[ g \mapsto \nu(g)^{n}|\nu(g)| \left(\frac{|\langle \qch, g r_0(i)\rangle|}{\langle \qch, g r_0(i)\rangle}\right)^{v} K_{v}(|\langle \qch, g r_0(i) \rangle|)\]
can only be bounded when $\qch$ is rank four.
 
To prove Proposition \ref{prop:vRnks}, we require the following two lemmas.
\begin{lemma}\label{lem:rk123} Assume $J=\R$ or $J = H_3(C)$ for some composition algebra $C$.  Suppose $x = (a,b,c,d) \in W$ is rank $1,2, $ or $3$.  Then there exists $g \in H_J^{1}$ so that $gx =(a',b',0,0)$. \end{lemma}
\begin{proof} This is clear if $x$ is rank one.  If $x$ is rank two, then we may assume $x = (1,0,c,0)$, with $c^\# = 0$ but $c \neq 0$.  By moving $c$, we may assume $c = \lambda e_{11}$ for some $\lambda \in \R^\times$.  Thus, in particular, we may assume $c = -v^\#$ for some $v \in J$.  But then $n_{G}(v)x = (1,v,c + v^\#, (c,v)+n(v)) = (1,v,0,0)$ as $c + v^\# = 0$, so $n(v) = 0$, and $(c,v) = (-v^\#,v) = -3n(v) = 0$.

Hence, we are reduced to the case that $x$ is rank three.  Then $x^\flat$ is rank one.  We require the following claim.
\begin{claim}\label{lem:SiegelFil} Set $f = (0,0,0,1)$.  Then
\begin{equation}\label{eqn:fP}\{(0,0,*,*)\} = \{x \in W_J: (f,x, y,y') = 0 \text{ for all } y, y' \in (\R f)^\perp\}.\end{equation}
\end{claim}
\begin{proof} One verifies (\ref{eqn:fP}) by direct computation.\end{proof}

Now, we claim that if $x$ is rank three, then $(x^\flat,x,y,y') = 0$ for all $y, y'$ in $(\R x^\flat)^\perp$.  Indeed, one has $6t(x,x^\flat,y) = \langle y,x\rangle x^\flat + \langle y,x^\flat \rangle x$, for any $y \in W_J$.  Hence, for any $y, y' \in W_J$, one has
\[6(x,x^\flat,y,y') = \langle y,x \rangle \langle y',x^\flat \rangle + \langle y,x^\flat \rangle \langle y',x \rangle.\]
Hence if $y,y' \in (\R x^\flat)^\perp$, the above is $0$.

It follows from Lemma \ref{lem:SiegelFil} that by moving $x^\flat$ to $(1,0,0,0)$, $x$ gets moved to an element of the form $(a',b',0,0)$.  This completes the proof of the lemma.\end{proof}

\begin{lemma}\label{lem:away0} The quantity
\[\frac{N(Z+i) - N(Z-i)}{N(Y)^{1/2}} = 2i\frac{\tr(Z^\#)-1}{N(Y)^{1/2}}\]
is bounded away from $0$ on $\mathcal{H}_J$.
\end{lemma}
\begin{proof} Observe that
\[\left|\frac{N(Z+i) - N(Z-i)}{N(Y)^{1/2}}\right| \geq \frac{|N(Z+i)|}{N(Y)^{1/2}} - \frac{|N(Z-i)|}{N(Y)^{1/2}}.\]
Furthermore, if $g\in H_J(\R)^{1}$ and $g \cdot (i1_J) = Z$ in $\mathcal{H}_J$, then one has $\langle g r_0(i), r_0(i)\rangle = N(Z-i)/N(Y)^{1/2}$ and $\langle g r_0(i), r_0(-i)\rangle = N(Z+i)/N(Y)^{1/2}$.  Thus $\frac{|N(Z+i)|}{N(Y)^{1/2}}$ and $\frac{|N(Z-i)|}{N(Y)^{1/2}}$ are each bi-$K_H^{1}$-invariant functions.  Hence we may assume $g$ is in the Levi $M_J$, and thus only consider the quantity $N(Y+1)/N(Y)^{1/2} - |N(Y-1)|/N(Y)^{1/2}$.

The quantity 
\[\frac{1}{2}\left(N(Y+1)/N(Y)^{1/2} - |N(Y-1)|/N(Y)^{1/2}\right)\]
is either equal to $\frac{\tr(Y^\#)+1}{N(Y)^{1/2}}$ if $N(Y-1) > 0$ or $\frac{N(Y) +\tr(Y)}{N(Y)^{1/2}}$ if $N(Y-1) < 0$.  By the AM-GM inequality, one has $\tr(Y^\#) \geq 3 N(Y)^{2/3}$ and $\tr(Y) \geq 3 N(Y)^{1/3}$.  Thus
\[\frac{\tr(Y^\#)+1}{N(Y)^{1/2}} \geq 3N(Y)^{1/6} + N(Y)^{-1/2} \geq 4\]
and
\[ \frac{N(Y) +\tr(Y)}{N(Y)^{1/2}} \geq N(Y)^{1/2} + 3N(Y)^{-1/6} \geq 4\]
by applying AM-GM again.  Thus $|\tr(Z^\#)-1|/N(Y)^{1/2}$ is bounded away from $0$, as desired.
\end{proof}

We now prove Proposition \ref{prop:vRnks}.
\begin{proof}[Proof of Proposition \ref{prop:vRnks}] We prove the statements in turn.
\begin{enumerate}
\item If $q(v) > 0$, then there is $g \in H_J(\R)^{1}$ with $g v = -(1,0,0,d)$, with $d \in \R^\times$.  Hence we may assume $v = -(1,0,0,d)$, so that $p_{v}(Z) = N(Z) + d$.  Then $Z = -\zeta d^{1/3} 1_J$ gives a zero of $p_{v}$, where $\zeta$ is a cube root of unity.
\item Suppose now that $q(v) < 0$.  Then, because we are assuming $J$ is simple, there is $g \in H_J(\R)^{1}$ so that $g v$ is equal to $(0,-1,0,1)$, up to scaling.  Thus that $|\langle v, g r_0(i)\rangle |$ is bounded away from $0$ follows from Lemma \ref{lem:away0}.
\item Finally, suppose that $v$ is of rank $1,2,$ or $3$.  Then by Lemma \ref{lem:SiegelFil}, we may assume $v = (a,b,0,0)$.  But then translating $g \in H_J(\R)^{1}$ by $\eta(t)$ for $t \in \R^\times$, one sees that $|\langle v, g r_0(i)\rangle |$ can be made arbitrarily close to $0$.
\end{enumerate}
\end{proof}

\section{The constant term and Eisenstein series}\label{sec:constAndEis} We have now analyzed all the nontrivial Fourier coefficients of a weight $n$ modular form $F$ on $G$.  In particular, we have seen that the space of associated generalized Whittaker functions is at most one-dimensional.  In this section, we analyze the constant term of weight $n$ modular forms.  We will see that these behave differently from non-constant terms.  There is no longer a multiplicity at most one result.  Using our analysis of the constant term, we then discuss Eisenstein series on $G_J$ associated to holomorphic modular cusp forms on $H_J$.

\subsection{The constant term} First, setting $\omega = 0$ in Corollary \ref{cor:SchmidChar}, and keeping the notation just as in this corollary, we obtain that the constant term satisfies the follow differential equations:
\begin{enumerate}
\item If $k > -n$, i.e., if $k \in \{-n+1,-n+2,\ldots,n-1,n\}$, then $(w \partial_w - 2(n+1) + k)\phi_k = 0$ and $(D_{Z^*(E)} - \frac{k}{2} \tr(E)) = 0$ for all $E \in J$.
\item If $k < n$, i.e., if $k \in \{-n,-n+1,\ldots, n-2,n-1\}$, then $(w \partial_w - 2(n+1) - k)\phi_k = 0$ and $(D_{Z(E)} + \frac{k}{2} \tr(E)) = 0$ for all $E \in J$.
\end{enumerate}

Hence, if $-n < k < n$, then subtracting the two $w\partial_w$ equations gives $2k \phi_k = 0$.  Hence if $-n < k < n$ and $k \neq 0$, $\phi_k = 0$.  We are thus left with the following equations:
\begin{enumerate}
\item $(w\partial_w - 2(n+1) + n)\phi_n = 0$, and $(D_{Z^*(E)} - \frac{n}{2} \tr(E)) = 0$ for all $E \in J$.
\item $(w\partial_w - 2(n+1) + n)\phi_{-n} = 0$, and $(D_{Z(E)} + \frac{n}{2} \tr(E)) = 0$ for all $E \in J$.
\item $(w\partial_w - 2(n+1))\phi_0 = 0$, and $D_{Z(E)}\phi_0 = D_{Z^*(E)}\phi_0 = 0$ for all $E \in J$.
\end{enumerate}

Analyzing these equations results in the following proposition.
\begin{proposition}\label{prop:const} The constant term of a modular form of weight $n$ on $G_J$, restricted to $H_J(\R)^{\pm}$, is of the form
\begin{equation}\label{eqn:constTerm} \nu(g)^{n}|\nu(g)|\left(\Phi(g) [x^{2n}] + \beta [x^n][y^n] + \Phi'(g)[y^{2n}]\right)\end{equation}
where $\beta \in \C$ is a constant, $\Phi'(g) = \Phi(g w_0)$, and $\Phi$ has the following properties:  Define $H_{\Phi}(g) = j(g,i)^{n}\Phi(g)$.  Then for $g \in H_J(\R)^{\pm}$, $H_{\Phi}(zgk) = H(g)$, for all $k \in K_H^1$ and $z \in Z_H(\R)$ the real points of the connected center of $H_J$.  Thus $H_{\Phi}$ descends to a function on the disconnected symmetric space $\mathcal{H}_J^{\pm}$.  It is holomorphic on $\mathcal{H}_J^{+}$, and antiholomorphic on $\mathcal{H}_J^{-}$. Conversely, suppose $H: \mathcal{H}_J^{\pm} \rightarrow \C$ is holomorphic on $\mathcal{H}_J^{+}$ and antiholomorphic on $\mathcal{H}_J^{-}$.  Set $\Phi_H(g) = j(g,i)^{-n}H(g)$ and $\Phi'_H(g) = \Phi_H(gw_0)$, and for any constant $\beta \in \C$, define $\mathcal{W}^0: H_J(\R)^{\pm} \rightarrow \Vm_n^\vee$ as
\[\mathcal{W}^0(g)= \nu(g)^{n}|\nu(g)|\left(\Phi_H(g) [x^{2n}] + \beta [x^n][y^n] + \Phi_H'(g)[y^{2n}]\right).\]
Then $\mathcal{W}^0(gk) = k^{-1} \cdot \mathcal{W}^0(g)$ for all $k \in K \cap H_J(\R)^{\pm}$ and $\mathcal{D}_n\mathcal{W}^0 = 0$. \end{proposition}

In other words, the function $\Phi$ appearing in (\ref{eqn:constTerm}) is an automorphic form on $H_J$ associated to a holomorphic modular form of weight $n$ for $H_J(\R)$.  In fact, one can show that the equations 
\[\left(D_{Z^*(E)} - \frac{n}{2}\tr(E)\right)\Phi(g) = 0\]
for $g \in H_J(\R)^{0}$ and all $E \in J$ are what the Schmid operator gives for holomorphic modular forms of weight $n$ on the group $H_J(\R)$.

\begin{proof}[Proof of Proposition \ref{prop:const}]  First note that, as in section \ref{sec:final}, the action of $w_0 \in K$ immediately implies that $\Phi'(g) = \Phi(gw_0)$, and that $\Phi(gk) = j(k,i)^{-n}\Phi(g)$, for all $k \in K_H^1$.  Now, recall that for $g = w n_G(-X)M_Y$, $j(g,i) = w N(Y)^{-1/2}$.  Thus, the function $H_{\Phi}(g) = j(g,i)^{n}\Phi(g)$ is right $K_H^1$ invariant, and satisfies the equations
\begin{enumerate}
\item $w\partial_w H_{\Phi}(w n_G(-X)M_Y) = 0$ and 
\item $D_{Z^*(E)} H_{\Phi}(w n_G(-X)M_Y) = 0$, for all $E \in J$. \end{enumerate}
It follows that the function $H_{\Phi}(g)$ for $g \in H_{J}(\R)^{0}$ descends to $\mathcal{H}_J$, and is holomorphic there.  Similarly, analyzing the differential equations for $\Phi'$ and $g \in H_J(\R)^{0}$ yields that the function $H'_{\Phi'}(g):= j(g,-i)^{n}\Phi'(g)$ descends to $\mathcal{H}_J$, and is antiholomorphic there.  Because $\Phi(g) = \Phi'(gw_0)$, the fact that $H_{\Phi}(g)$ is antiholomorphic on $\mathcal{H}_J^{-}$ follows.

The converse follows in essentially the same way.\end{proof}

\subsection{Eisenstein series} We now consider ``Klingen-type" Eisenstein series on the group $G_J$.  That is, we consider the Eisenstein series on $G_J$ associated to holomorphic cusp forms on $H_J$.  For simplicity, we restrict to the case that $J = \R$ or $J = H_3(C)$ with $C$ a definite composition algebra, so that $H_J(\R) = H_J(\R)^{\pm}$ and $G_J$ is exceptional, with $G_J(\R)$ connected.

Thus, suppose that $\Phi: H_J(\Q)\backslash H_J(\A) \rightarrow \C$ is a cusp form on $H_J$.  For any $g_f \in H_J(\A_f)$, define the function $H_{\Phi,g_f}: H_J(\R) \rightarrow \C$ as $H_{\Phi,g_f}(g) = j(g,i)^{n}\Phi(g_f g)$.  We say that $\Phi$ is associated to a holomorphic modular form of weight $n$ if the following condition is satisfied: For all $g_f \in H_J(\A_f)$, the function $H_{\Phi,g_f}$ descends to $\mathcal{H}_J^{\pm}$, is holomorphic on $\mathcal{H}_J^{+}$, and antiholomorphic on $\mathcal{H}_J^{-}$.  That $H_{\Phi,g_f}$ descends to $\mathcal{H}_J^{\pm}$ means that $H_{\Phi,g_f}(zgk) = H_{\Phi,g_f}(g)$ for all $z \in Z_H(\R)$ and $k \in K_H^{1}$.  In particular, $\Phi(gk) = j(k,i)^{-n}\Phi(g)$ for all $k \in K_H^1$ and $\Phi(wg) = w^{-n}\Phi(g) = |\nu(w)|^{-n/2}\Phi(g)$ for all $w \in Z_H(\R)^{0}$.

Associated to $\Phi$, define $f_{\Phi}^0: H_J(\A) \rightarrow \Vm_n^\vee$ as
\[f_{\Phi}^0(g) = \nu(g)^{n}|\nu(g)|\left(\Phi(g)[x^{2n}] + \Phi(gw_0)[y^{2n}]\right).\]
Then $f_{\Phi}^0(wg) = w^{n+2}f_{\Phi}^0(g) = |\nu(w)|^{(n+2)/2}f_{\Phi}^0(g)$ for all $w \in Z_H(\R)^{0}$.  Also, $f_{\Phi}^0(gk) = k^{-1} f_{\Phi}^0(g)$ for all $k \in K \cap H_J(\R)$.

Now, suppose $f_{\Phi}: G_J(\A) \rightarrow \Vm_n^\vee$ extends $f_{\Phi}^0$, in the sense that
\begin{enumerate}
\item $f_{\Phi}(gk) = k^{-1} \cdot f_{\Phi}(g)$ for all $g \in G_J(\A)$ and $k\in K \subseteq G_J(\R)$;
\item $f_{\Phi}(nm) = f_{\Phi}^{0}(m)$ for all $n \in N_J(\A)$ and $m \in H_J(\A)$. \end{enumerate}

Then, we may defined the Eisenstein series 
\begin{equation}\label{eqn:EisSum} E(g,f_{\Phi}) = \sum_{\gamma \in P(\Q)\backslash G_J(\Q)}{f_{\Phi}(\gamma g)}.\end{equation}
We have the following proposition regarding this Eisenstein series.
\begin{proposition}\label{prop:Eis} Suppose $\Phi$ is associated to a modular form of weight $n$ on $H_J$, and suppose that $n > \dim W_J$.  Then the sum (\ref{eqn:EisSum}) converges absolutely, and for any $g_f \in G_J(\A_f)$, the function $G_J(\R) \rightarrow \Vm_n^\vee$ defined by $g \mapsto E(g_fg,f_{\Phi})$ is in the kernel of $\mathcal{D}_n$.\end{proposition}
\begin{proof} One has $\delta_{P}(g) = |\nu(g)|^{(\dim W_J+2)/2}$.  Thus when $n > W_J$, the sum (\ref{eqn:EisSum}), and any $\mathcal{U}(\g(J))$-derivatives of it, converges absolutely, and uniformly on compact subsets of $G_J(\R)$.  By Proposition \ref{prop:const}, the term with $\gamma = 1$ of (\ref{eqn:EisSum}) is annihilated by $\mathcal{D}_n$.  Because the differential operator commutes with left-translation, so too do the terms with $\gamma \neq 1$. That $E(g,f_{\Phi})$ is annihilated by $\mathcal{D}_n$ now follows from uniform convergence.\end{proof}

\appendix

\section{Lie algebras of orthogonal groups}\label{sec:orthog} In this section we explain how orthogonal groups $\SO(N,4)$ fit into the framework of cubic Jordan algebras and spaces $W_J$.  This is of course well-known.  The point of this section is to record all the precise normalizations so that Theorem \ref{intro:mainThm} applies to the quaternionic adjoint groups of type $B_{\ell}$ with $\ell \geq 3$ and $D_{\ell}$ with $\ell \geq 4$.  More precisely, associated to a quadratic space $S$ of signature $(1,r)$ with $r\geq 0$, we will define a cubic norm structure $J$ and a quadratic space $\mathbb{V}$ of signature $(r+3,4)$.  Then, in Proposition \ref{prop:VVgJ}, we specify a Lie algebra isomorphism $\so(\mathbb{V}) \simeq \g(J)$ between $\so(\mathbb{V}) \simeq \wedge^2(\mathbb{V})$ and $\g(J)$.

\subsection{Setup} We now give the setup and assumptions on our quadratic space.
\subsubsection{The quadratic space $S$}
Suppose $S$ is a quadratic space, which has signature $(1,*)$, and $q_{S}$ is the quadratic form on $S$.  We fix $1_S\in S$, with $q_{S}(1_S) = 1$.  Furthermore, fix an involution $\iota_{S}: S \rightarrow S$ with $\iota_S(1_S) = 1_{S}$ and $\iota_S$ equal to $-Id$ on $(1_S)^{\perp}$.  We write $(\,,\,)_{S}$ for the bilinear form on $S$ associated to $q_{S}$, so that $(x,y)_{S} = q_{S}(x+y) - q_{S}(x) - q_{S}(y)$.  The involution $\iota_{S}$ is such that $(x,\iota_{S}(y))_{S}$ is a symmetric positive definite quadratic form.

\subsubsection{The cubic norm structure $J$} Set $J = F\times S$.  Define $n: J \rightarrow F$ via $n((\beta,t)) = \beta q_{S}(t)$.  Define $\#: J \rightarrow J$ via $(\beta,t)^\# = (q_{S}(t),\beta \iota_{S}(t))$.  Define $1_{J} \in J$ as $1_{J} = (1,1_{S})$.  Define a pairing on $J$ via $((\beta,t),(\beta',t'))_{J} = \beta \beta' + (t,\iota_{S}(t'))_{S}$.  Theorem \ref{thm:quadJ} below states that with these definitions, $J$ is a cubic norm structure. Denote by $W_J$ Freudenthal's construction (see subsection \ref{subsec:FWJ}) associated to this cubic norm structure.

\subsubsection{The quadratic space $V$} Define $V = F \oplus S \oplus F$ with quadratic form $q_{V}((\alpha,s,\beta))= \alpha \beta - q_{S}(s)$.  Then $V$ has signature $(*+1,2)$. Define the bilinear form on $V$ via $(x,y)_{V} = q_{V}(x+y) - q_{V}(x) - q_{V}(y)$.  

Suppose $V_2 = Fe \oplus Ff$.  Define a map $V_2 \otimes V \rightarrow W_J$ via
\begin{equation}\label{eqn:V2VW} e \otimes (a,s,\beta) + f \otimes (\gamma,t,\delta) \mapsto (\alpha,(\gamma,s),(\beta,\iota_{S}(t)),\delta).\end{equation}

On $V_2 \otimes V$ one has the symplectic form given by $\langle v\otimes x, v'\otimes x'\rangle = \langle v,v'\rangle (x,x')_{V}$, where here $v,v' \in V_2$ and $x,x' \in V$.  Suppose $x, y \in V$ and $e,f$ our fixed symplectic basis of $V_2$.  One has the quartic form given by 
\[Q(e \otimes x + f \otimes y) = (x,y)_{V}^2 - 4 q_{V}(x)q_{V}(y) = (x,y)_{V}^2 - (x,x)_{V}(y,y)_{V} = -\det\left(\begin{array}{cc} (x,x)_{V} & (x,y)_{V} \\ (x,y)_{V} & (y,y)_{V} \end{array}\right).\]

With the above signatures, $\SO(V)$ has a Hermitian symmetric space.  The symmetric space associated to $\SO(V)$ is the set of negative definite two-planes in $V$.  Equivalently, this symmetric space can be defined as the set of isotropic lines $\ell$ in $V \otimes \C$ with $(\ell, \overline{\ell})_{V} < 0$.  If $U$ is a negative definite two-plane with orthogonal basis $e_1, e_2$, with $(e_1,e_1) = (e_2,e_2) < 0$, then the isotropic line associated to $U$ is $\C(e_1 + ie_2)$.  

With the way we have defined $q_V$, note that $(1,z,q_S(z)) \in V\otimes \C$ is isotropic, where $z \in S\otimes \C$.  One computes
\[((1,z,q_S(z)),(1,z,q_{S}(z))^*)_V = q_{S}(z) + q_{S}(z)^* - (z,z^*)_S = q_S(z-z^*) = - q_{S}\left(\frac{z-z^*}{i}\right).\]
Thus if $z = x + iy$, with $x,y \in S$, the condition for $(1,z,q_S(z))$ to define a point in the Hermitian symmetric space is $q_{S}(y) > 0$.

\subsection{Freudenthals space} With the above notations, one has the following fact.
\begin{theorem}\label{thm:quadJ} With definitions as above, $J$ is a cubic norm structure.  Furthermore, the symplectic and quartic form on $V_2 \otimes V$ are Freudenthal's symplectic and quartic form on $W_J$.\end{theorem}

On $W_J$, one has the involution $J_W$, given by $J_{W}(a,b,c,d) = (d,-c,b,-a)$.  One can write this in terms of involutions on $V_2$ and $V$.  Namely, put on $V$ the involution $\iota_{V}$ given by $\iota_{V}(\alpha,s,\beta) = (\beta,-\iota_{S}(s),\alpha)$.  Then $\iota_{V}$ gives rise to a Cartan involution because $(v,\iota_{V}(v'))$ is symmetric and positive definite on $V$.  The involution $\iota_{V}$ is the indentity on a positive definite space and minus the identity on a negative definite one.  Identifying $W_J$ with $V_2 \otimes V$ via the map (\ref{eqn:V2VW}), one finds that the involution $J_W$ is $J_2 \boxtimes \iota_{V}$.

\subsection{The total space} Denote by $M_2$ the $2\times 2$ matrices, and define $\mathbb{V} = M_2 \oplus V = V_2 \otimes V_2 \oplus V$, with quadratic form $q_{\mathbb{V}}(m,v) = \det(m) + q_{V}(v)$.  The bilinear form is
\[((m,v),(m',v'))_{\mathbb{V}} = \tr(\widetilde{m} m') + (v,v')_{V} = \tr(J_2 (\,^tm) J_2^{-1} m') + (v,v')_{V}.\]
Here $\widetilde{m} = J_2 \,^{t}m J_2^{-1}$, so that $m \widetilde{m} = \det(m)$. If $m = \mm{a}{b}{c}{d}$, then in $V_2 \otimes V_2$, $m$ corresponds to $e \otimes \left(\begin{array}{c} a \\ c \end{array}\right) + f \otimes \left(\begin{array}{c} b \\ d \end{array}\right)$.  The bilinear form on $\mathbb{V}$, in the $V_2 \otimes V_2$ notation is
\[((w_1\otimes w_2,v),(w_1'\otimes w_2',v'))_{\mathbb{V}} = \langle w_1,w_1' \rangle \langle w_2, w_2'\rangle + (v,v')_{V}.\]

Suppose $\iota_{V}$ is an involution on $V$ giving rise to a Cartan involution, i.e., so that $(v,\iota_{V}(v'))$ is positive definite and symmetric.  One defines on $\mathbb{V}$ the involution $\iota_{\mathbb{V}}$ given by $\iota_{\mathbb{V}}(m,v) = (J_2 m J_2^{-1}, \iota_{V}(v))$.  If $m$ corresponds to $w_1 \otimes w_2$, then $J_2 m J_2^{-1}$ corresponds to $J_2w_1 \otimes J_2w_2$.

The quadratic space $M_2$ has $\so(M_2) = \wedge^2(M_2) \simeq \mathfrak{sl}_2 \oplus \mathfrak{sl}_2$.  The following lemma makes precise this isomorphism.
\begin{lemma}\label{lem:SO4} Define a map $\wedge^2(V_2^{(1)} \otimes V_2^{(2)}) \rightarrow Sym^2(V_2^{(1)}) \oplus Sym^2(V_2^{(2)})$ via
\begin{equation}\label{eqn:SO4} (v_1 \otimes v_2) \wedge (v_1' \otimes v_2') \mapsto -\frac{1}{2}\left( \langle v_2,v_2' \rangle v_1 \cdot v_1' + \langle v_1, v_1' \rangle v_2 \cdot v_2'\right).\end{equation}
Then under the identification $Sym^2(V_2) \simeq \sp(V_2) =\mathfrak{sl}_2$ on the right-hand side, and $\wedge^2(V) \simeq \mathfrak{so}(V)$ on the left-hand side, both given in section \ref{sec:LieI}, the map (\ref{eqn:SO4}) is a Lie algebra isomorphism.\end{lemma}
\begin{proof} Apply both sides to $u_1 \otimes u_2$, for $u_1 \in V_{2}^{(1)}$ and $u_2 \in V_{2}^{(2)}$.  To check that the two sides agree, use the identity $\langle x,y \rangle z + \langle y,z\rangle x + \langle z,x \rangle y = 0$ for any $x,y,z \in V_2$.\end{proof}

\subsection{The Lie algebra of the total space} To continue understanding the Lie algebra $\wedge^2(\mathbb{V})$ of the total space $\mathbb{V}$, we must compute the commutator on $M_2 \otimes V$.  Suppose $m = v_1 \otimes v_2, m'=v_1'\otimes v_2' \in M_2$ and $v,v' \in V$.  Then applying Lemma \ref{lem:SO4} one obtains
\begin{align*} [m\otimes v,m'\otimes v'] &= -(v,v') [m,m'] - (m,m') v\wedge v' \\ &= \frac{1}{2}(v,v')\langle v_2,v_2'\rangle v_1 v_1' + \frac{1}{2}(v,v')\langle v_1,v_1'\rangle v_2 v_2' -\langle v_1,v_1'\rangle \langle v_2,v_2'\rangle v \wedge v'.\end{align*}

To relate $\wedge^2(\mathbb{V})$ to $\g(J)$, we use the following proposition.
\begin{proposition}\label{prop:SOalpha} One has
\begin{equation}\label{eqn:SOalpha}\frac{1}{2}(v,v') v_2 v_2' - \langle v_2,v_2' \rangle v \wedge v'= \frac{1}{2}\Phi_{v_2\otimes v, v_2' \otimes v'}.\end{equation}
\end{proposition}
The proposition will follow easily from the following lemma.
\begin{lemma}\label{lem:6qpure} Suppose $w_1, w_2, w_3, w_4 \in V_2$ and $v_1,v_2,v_3,v_4 \in V$.  Then
\begin{align}\label{6qpure}\nonumber 6(w_1\otimes v_1,w_2\otimes v_2, w_3 \otimes v_3, w_4 \otimes v_4) &= \langle w_1,w_2 \rangle \langle w_3,w_4\rangle \left((v_2,v_3)(v_1,v_4) - (v_1,v_2)(v_3,v_4)\right) \\ \nonumber &\; + \langle w_1,w_3 \rangle \langle w_2, w_4\rangle \left((v_2,v_3)(v_1,v_4) - (v_1,v_3)(v_2,v_4)\right) \\ & \;+ 2 \langle w_1,w_4\rangle \langle w_2,w_3 \rangle \left((v_1,v_3)(v_2,v_4) - (v_1,v_2)(v_3,v_4)\right).\end{align}
In particular, if $w,w' \in V_2$ and $v,v' \in V$, then 
\begin{equation}\label{eqn:3qww'}3(w\otimes v,w\otimes v,w'\otimes v', w'\otimes v') = \langle w,w'\rangle^2 \left((v,v')^2 - (v,v)(v',v')\right).\end{equation}
\end{lemma}
\begin{proof} We first verify (\ref{eqn:3qww'}). For this, note that the quartic form on $V_2 \otimes V$ is described as follows.  First one applies $q_{V}$ to get a quadratic map $V_2 \otimes V \rightarrow Sym^2 V_2$.  Then, if $\varphi \in Sym^2V_2$, one applies the quadratic map $\varphi \mapsto -\varphi^2$, where $\varphi^2$ is the contraction of $\varphi \otimes \varphi \in S^2V_2 \otimes S^2V_2$ to the trivial representation.  Consequently, 
\begin{align*} q(w\otimes v + w' \otimes v') &= - (q(v) w^2 + (v,v')ww' + q(v') (w')^2)^2 \\ &\; - \left(q(v) w^2 + (v,v')ww' + q(v') (w')^2\right)\left(q(v) w^2 + (v,v')ww' + q(v') (w')^2\right) \\ &= - \left(4q(v)q(v') \langle w,w'\rangle^2 - (v,v')^2 \langle w,w'\rangle^2\right) \\ &= \langle w,w'\rangle^2 \left((v,v')^2 - (v,v)(v',v')\right).\end{align*}

Now, from this formula it follows that if $X = w \otimes v$ is a pure tensor in $V_2 \otimes V$, then $(X,X,X,Y) = 0$ for any $Y \in V_2 \otimes V$.  Hence for $X,Y$ both pure tensors, $3(X,X,Y,Y) = q(X+Y)$.  Formula (\ref{eqn:3qww'}) follows.  Linearizing (\ref{eqn:3qww'}), one gets
\begin{equation}\label{eqn:6ww'}6(w\otimes v_1, w\otimes v_2, w' \otimes v_1', w' \otimes v_2') = \langle w,w'\rangle^2 \left( (v_1,v_1')(v_2,v_2') + (v_1,v_2')(v_2,v_1') - 2(v_1,v_2)(v_1',v_2')\right).\end{equation}

To verify the formula for $6(w_1\otimes v_1,w_2\otimes v_2, w_3 \otimes v_3, w_4 \otimes v_4)$ we deduce it from (\ref{eqn:6ww'}).  Indeed, first note if any three of $w_1,w_2,w_3,w_4$ are parallel, then $(w_1\otimes v_1,w_2\otimes v_2, w_3 \otimes v_3, w_4 \otimes v_4)$ vanishes by our above remarks.  One verifies easily that the same is true of the right-hand side of (\ref{6qpure}).  Furthermore, since both sides (\ref{6qpure}) are linear in the input, it suffices to verify (\ref{6qpure}) when all the $w_i$ are the basis vectors $e$ or $f$.  Moreover, by what we have just said, we may assume two of the $w_i$ are $e$ and two of the $w_i$ are $f$.  One considers the three cases $w_1 = w_2 = w, w_3 = w_4 = w'$, $w_1 = w_3=w, w_2 = w_4=w'$, $w_1 = w_4= w, w_2 = w_3=w'$ separately.  In each case, one finds that the right-hand side of (\ref{6qpure}) simplifies to (\ref{eqn:6ww'}).  The lemma follows.\end{proof}

Using Lemma \ref{lem:6qpure}, we can now deduce Proposition \ref{prop:SOalpha}. 
\begin{proof}[Proof of Proposition \ref{prop:SOalpha}] Changing notation and multiplying by $2$, we wish to check
\begin{align*} (v_2,v_3) w_2 w_3 -2\langle w_2,w_3\rangle v_2 \wedge v_3 &= \Phi_{w_2 \otimes v_2, w_3 \otimes v_3} \\ &= 6(w_2 \otimes v_2, w_3 \otimes v_3, \cdot) + \langle w_2 \otimes v_2, \cdot \rangle w_3 \otimes v_3 + \langle w_3 \otimes v_3, \cdot \rangle w_2 \otimes v_2.\end{align*}
Equivalently, by the non-degeneracy of the symplectic form on $W_J = V_2 \otimes V$, we wish to check equality when applying both sides above to $w_4 \otimes v_4$ and pairing on the left with $w_1\otimes v_1$.  Rearranging, one gets the equality of Lemma \ref{lem:6qpure}, and Proposition \ref{prop:SOalpha} follows.\end{proof}

We can now define the map $\so(\mathbb{V}) \simeq \wedge^2 \mathbb{V} \rightarrow \g(J)$ with $\alpha=\frac{1}{2}$.  Namely, we have
\[\wedge^2(\mathbb{V}) = \wedge^2(M_2 \oplus V) = \wedge^2(M_2) \oplus \wedge^2(V) \oplus M_2 \otimes V.\]
We have already defined in Lemma \ref{lem:SO4} an isomorphism $\wedge^2(M_2) \simeq \sl_2^{(1)} \oplus \sl_2^{(2)}$, where we have labelled the two copies of $\sl_2$ with superscripts $(1)$ and $(2)$.  Thus, we have defined
\[\wedge^2(M_2) \oplus \wedge^2(V) \rightarrow \sl_2^{(1)} \oplus \left(\sl_2^{(2)} \oplus \so(V)\right) = \sl_2^{(1)} \oplus \h(J)^{0} = \g(J)_0.\]

The map $M_2 \otimes V \rightarrow \g(J)_{1}$ is given as
\[M_2 \otimes V \simeq V_2^{(1)} \otimes \left(V_2^{(2)} \otimes V\right) \simeq V_2^{(1)} \otimes W_J = \g(J)_1.\]
Here the identification $M_2 \simeq V_2^{(1)} \otimes V_2^{(2)}$ is given by 
\[m = \left(\begin{array}{cc} a & b \\ c & d \end{array}\right) \mapsto e_1 \otimes \left(\begin{array}{c} a \\ c \end{array}\right) + f_1 \otimes \left(\begin{array}{c} b \\ d \end{array}\right).\]
The identification $V_2 \otimes V \simeq W_J$ is given by (\ref{eqn:V2VW}).  With these identifications, one has the following proposition.

\begin{proposition}\label{prop:VVgJ} The identifications above induce a Lie algebra isomorphism $\so(\mathbb{V}) \rightarrow \g(J)$ when $\alpha = \frac{1}{2}$.  This map carries the Cartan involution and invariant pairing on $\so(\mathbb{V})$ over to the Cartan involution and invariant pairing on $\g(J)$. \end{proposition}
\begin{proof} The proof is a direct application of the formulas above, in particular Proposition \ref{prop:SOalpha}.\end{proof}

\bibliography{QDS_Bib} 
\bibliographystyle{amsalpha}
\end{document}